\RequirePackage{fix-cm}
\documentclass[a4paper,10pt]{svjour3}                     
\smartqed  
\usepackage{graphicx}
\usepackage{cite}
\usepackage{amsmath}
\usepackage{amssymb}
\usepackage{subfigure}
\usepackage{graphicx}
\usepackage{multirow}
\usepackage{epstopdf}
\usepackage[justification=centering]{caption}
\usepackage{marvosym}
\usepackage[colorlinks,
            linkcolor=blue,
            anchorcolor=blue,
            citecolor=blue]{hyperref}                  
\usepackage{geometry}
\begin{document}

\title{Fast method and convergence analysis of fractional magnetohydrodynamic coupled flow and heat transfer model for generalized second-grade fluid
}

\titlerunning{Fast method and convergence analysis of fractional MHD coupled flow and heat transfer model}        

\author{Xiaoqing Chi         \and
        Hui Zhang$^\ast$  \and  Xiaoyun Jiang}

\authorrunning{X. Chi, H. Zhang and X. Jiang} 

\institute{   X. Chi \at
              School of Mathematics and Statistics, Shandong University,
              Weihai 264209, People's Republic of China \\
              \email{cxq15905413702@163.com} \\ \\
              H. Zhang \at
              School of Mathematics, Shandong University,
              Jinan 250100, People's Republic of China \\
              $^\ast$Corresponding \email{zhangh@sdu.edu.cn}\\ \\
              X. Jiang \at
              School of Mathematics, Shandong University,
              Jinan 250100, People's Republic of China \\
              \email{wqjxyf@sdu.edu.cn}  \\        
}

\date{Received: date / Accepted: date}

\maketitle
\begin{abstract}
In this paper, we first establish a new fractional magnetohydrodynamic (MHD) coupled flow and heat transfer model for a generalized second-grade fluid. This coupled model consists of a fractional momentum equation and a heat conduction equation with a generalized form of Fourier law. The second-order fractional backward difference formula is applied to the temporal discretization and the Legendre spectral method is used for the spatial discretization. The fully discrete scheme is proved to be stable and convergent with an accuracy of $O(\tau^2+N^{-r})$, where $\tau$ is the time step size and $N$ is the polynomial degree. To reduce the memory requirements and computational cost, a fast method is developed, which is based on a globally uniform approximation of the trapezoidal rule for integrals on the real line. And the strict convergence of the numerical scheme with this fast method is proved. We present the results of several numerical experiments to verify the effectiveness of the proposed method. Finally, we simulate the unsteady fractional MHD flow and heat transfer of the generalized second-grade fluid through a porous medium. The effects of the relevant parameters on the velocity and temperature are presented and analyzed in detail.
\keywords{Fractional MHD coupled flow and heat transfer model \and Generalized second-grade fluid \and Fast method \and Convergence analysis \and Numerical simulation}
\subclass{76W05\and 35R11\and 65M12 \and 65M70}
\end{abstract}
\section{Introduction}
Recently, magnetohydrodynamic (MHD) flow and heat transfer have attracted considerable attention because of their important applications in nuclear reactors, metallurgical processes, aerospace engineering, and biomedical engineering \cite{1,2,3,4,5,Wang2015}. These problems involve fluid flow, heat transfer, and multi-field coupling. Thus, their mathematical models include coupling, nonlinearity, and multivariable calculations, all of which make it difficult to obtain the exact form of the solution. With advances in calculation methods, the development of efficient numerical algorithms to simulate the MHD flow and heat transfer phenomena is now a feasible possibility.


{
Second-grade fluid is a common non-Newtonian viscoelastic fluid in the industrial field, which can simulate a variety of liquids such as dilute polymer solutions, slurry flow and industrial oils. The classical linear constitutive model for second-grade fluids has the following form \cite{ey0}}
{
\begin{equation}
{\widetilde{S}(t)}=\mu \varepsilon(t)+\alpha_0 \frac{\partial \varepsilon(t)}{\partial t},
\label{ey1}
\end{equation}
where $\widetilde{S}$ is the additional stress tensor, $\varepsilon$ is the shear strain, $\mu$ is the dynamic viscosity, and $\alpha_0$ is the viscoelastic coefficient.
Although the mathematical models established by using the relation (\ref{ey1}) can provide a reasonable qualitative description of fluid flow, they are not satisfactory from a quantitative point of view \cite{ey1}. Some scholars \cite{ey2,ey3} even found that it is not enough to describe viscoelastic materials with integer order models from both qualitative and quantitative perspective, and they proposed to use the fractional constitutive relation to simulate viscoelastic behaviors of real materials. The fractional constitutive relation of the generalized second-grade fluid can be expressed as follows \cite{25,26}:
\begin{equation}
{\widetilde{S}}(t)=\mu \mathbf{\varepsilon(t)}+\alpha_{1}\ {^{RL}_{~~0}D_{t}^{\gamma}} \varepsilon(t),
\end{equation}
where $\alpha_1$ is the fractional viscoelastic coefficient. The Riemann--Liouville fractional derivative $_{~~0}^{RL}D^{\gamma}_{t} (0<\gamma<1)$ is defined as \cite{33}:
\begin{eqnarray*}
_{~~0}^{RL}D^{\gamma}_{t}u(t)=\frac{1}{\Gamma(1-\gamma)}\frac{d}{dt}\int^t_0\frac{u(\xi)}{(t-\xi)^{\gamma}}d\xi,\quad t>0.
\label{xj140}
\end{eqnarray*}}

{
Many studies on fractional models of the generalized second-grade fluid were carried out. Tan and Xu \cite{ey4,ey5} used the fractional model of generalized second-grade fluid to study the sudden flow of fluid near the wall and the unsteady flow between two parallel plates. Mahmood et al. \cite{ey6} introduced the fractional constitutive relation to study the velocity field and shear stress of longitudinal oscillating flow of generalized second-grade fluid between two infinite coaxial cylinders. Moreover, the effect of transverse magnetic fields on the unsteady flow of generalized second-grade fluid through infinite plates in porous media was investigated in \cite{ey7}.}

{
Fluid flow is usually accompanied by heat transfer. Fractional derivatives are also used in the derivation of the generalized heat conduction law of viscoelastic fluids, which is used to describe the heat transfer phenomenon in complex fluid flow. For instance,  Ezzat et al. \cite{Ezzat} constructed a fractional heat conduction model with a modified form of Fourier law that can be applied to Stokes' first problem for a viscoelastic fluid with heat sources, and Zhang et al. \cite{e121} studied the flow and heat transfer of fractional Oldroyd-B nanofluid between two coaxial cylinders by using the fractional heat conduction equation in spherical coordinate. Liu and Guo \cite{29} introduced a new fractional model for Fourier law of heat conduction, and combined this with the constitutive relationship to study the unsteady MHD flow of a generalized Maxwell fluid.}


Motivated by the abovementioned studies, a new MHD coupled flow and heat transfer model with time-fractional derivatives is established for a generalized second-grade fluid. The flow is induced by a moving infinite plate and influenced by a magnetic field, radiation, and a heat source. The fractional MHD coupled flow and heat transfer model is written as follows:
\begin{equation}
\begin{aligned}
\frac{\partial u}{\partial t}=&\left(\upsilon+\frac{\alpha_{1}}{\rho}\ {}_{~~0}^{RL}D_{t}^{\gamma}\right) \frac{\partial^{2} u}{\partial z^{2}}+\frac{\sigma B_{0}^{2}}{\rho(1+m^{2})}(mv-u)-\frac{\phi}{k_{1}}\left(\upsilon+\frac{\alpha_{1}}{\rho}\ {_{~~0}^{RL}D_{t}^{\gamma}}\right) u +g \beta_{T}\left(T-T_{\infty}\right),
\end{aligned}
\label{01}
\end{equation}
\begin{equation}
\begin{aligned}
&\frac{\partial v}{\partial t}=\left(\upsilon+\frac{\alpha_{1}}{\rho}\ { }_{~~0}^{RL}D_{t}^{\gamma}\right) \frac{\partial^{2} v}{\partial z^{2}}-\frac{\sigma B_{0}^{2}}{\rho(1+m^{2})}(v+m u)
-\frac{\phi}{k_{1}}\left(\upsilon+\frac{\alpha_{1}}{\rho}\ {_{~~0}^{RL}D_{t}^{\gamma}}\right) v,
\end{aligned}
\label{02}
\end{equation}
\begin{equation}
\begin{aligned}
\left(1+\lambda_{1}\ {_{~~0}^{RL}{D}_{t}^{\beta}}\right) \frac{\partial T}{\partial t}=&\frac{k_0}{\rho c_{p}} \frac{\partial^{2} T}{\partial z^{2}}+\frac{1}{\rho c_{p}}\frac{16 \sigma^{*} T_{\infty}^{3}}{3 k^{*}} \frac{\partial^{2} T}{\partial z^{2}}+\frac{Q_{0}}{\rho c_{p}}\left(1+\lambda_{1}\ {_{~~0}^{RL}{D}_{t}^{\beta}}\right)\left(T-T_{\infty}\right),
\end{aligned}
\label{03}
\end{equation}
where $\upsilon$ is the kinematic viscosity coefficient, $\alpha_1$ is the fractional viscoelastic coefficient, $\rho$ is the fluid density, $\phi$ is the porosity of the porous medium, $k_1$ is the permeability of the porous medium, $\sigma$ is the electrical conductivity of the fluid, $B_0$ is the uniform magnetic field, $m$ is the Hall parameter, $g$ is the acceleration due to gravity, $\beta_T$ is the volumetric coefficient of thermal expansion, $T_{\infty}$ is the constant temperature, $\lambda_1$ is the relaxation time, $k_0$ is the thermal conductivity, $c_\rho$ is the specific heat capacity, $\sigma^{\ast} $ is the Stefan--Boltzmann constant, $k^{\ast} $ is the mean absorption coefficient, and $Q_0$ is the heat absorption/generation coefficient.  $_{~~0}^{RL}D^{\gamma}_{t}$, $_{~~0}^{RL}D^{\beta}_{t} (0<\gamma, \beta<1)$ are the Riemann--Liouville fractional derivatives.
A detailed discussion of this model is presented in Section \ref{sec:2}.

There are many effective numerical algorithms for solving time-fractional partial differential equations \cite{Zeng,Li2019,Chen,Liao,Jiang2019}, some of which have been developed to describe the dynamic behavior of fractional MHD fluids. Anwar and Rasheed \cite{Anwar} proposed a finite element--finite difference algorithm for solving a time-fractional Cattaneo--Maxwell model, while Cao et al. \cite{Cao} studied the MHD flow and heat transfer of a fractional Maxwell viscoelastic nanofluid over a moving plate using a finite difference method combined with an $L1$  algorithm. Jiang et al. \cite{Jiang} developed a second-order fractional backward difference spectral collocation scheme for solving the unsteady MHD flow of a generalized second-grade fluid through a porous medium. Although there are many numerical algorithms for simulating fractional MHD flows, the theoretical analysis for these schemes is limited.

Note that the time-fractional derivative operators are non-local, which can cause a lot of computational difficulties. Generally, the fractional derivative operator $_{~~0}^{RL}D^{\gamma}_{t} u$ at $t=t_k$ is usually approximated by a discrete convolution of the form \cite{Lubich,Sun1}
\begin{eqnarray}
_{~~0}^{RL}D^{\gamma}_{t} u^k=\frac{1}{\tau^{\gamma}}\sum_{j=0}^{k} \omega^{(\gamma)}_{k-j} u^{j}, \quad 0 \leq j \leq k, 0<k\leq \bar{K},
\end{eqnarray}
where $\tau$ is the time step and $\omega^{(\gamma)}_{k}$ is the convolution quadrature weight. Computing this convolution directly requires $O(\bar{K} )$ active memory and $O({\bar{K}}^2)$ operations, which is computationally costly. To overcome this disadvantage, some fast memory-saving algorithms have been proposed~\cite{Baffet,Jiang2,Sun,Zeng2,Guo}. In this paper, we propose a fast method \cite{Guo} in which the quadrature weight $\omega_{k}^{(\gamma)}$ is taken as an integral on the half line by the Hankel contour, which improves the computational efficiency and reduces the memory cost. Convergence analysis for fast methods is usually complicated. However, the convergence analysis of a fast time-stepping numerical method for the time-fractional nonlinear sub-diffusion equation has been carried out in a relatively simple manner \cite{39}. Inspired by \cite{39}, we focus on the convergence analysis of the fast method for the coupled model of (\ref{01})--(\ref{03}) with multiple time-fractional derivatives.

This work is devoted to developing an efficient numerical algorithm for the fractional MHD coupled flow and heat transfer model given by (\ref{01})--(\ref{03}). The stability and convergence of the numerical scheme are proved both with and without the fast method, providing original research into MHD flow and heat transfer. The main contributions of this paper are as follows:

$\bullet$ First, we establish a time-fractional MHD coupled flow and heat transfer model describing the flow and heat transfer behavior of a generalized second-grade fluid. The motion equation combined with the fractional constitutive relation and the energy equation with a modified form of Fourier law are coupled with each other.

$\bullet$ We provide the second-order fractional backward difference formula with the Legendre spectral method to solve this coupled model. Stability and convergence analysis are carried out, and it is shown that the numerical scheme can achieve an optimal error estimate of $O(\tau^2+N^{-r})$. 

$\bullet$ A fast method that reduces the computation time and memory requirements is also developed. The strict convergence of the numerical scheme with this fast method is proved.
The proof is based on the convergence of the direct method, which is simpler than that of existing fast methods.

$\bullet$ We numerically simulate the unsteady MHD flow and heat transfer of a generalized second-grade fluid with a Hall current passing through a porous medium near a vertical infinite plate. According to the results, the effects of relevant parameters on MHD flow and heat transfer are analyzed.

The remainder of this paper is organized as follows. Section \ref{sec:2} describes the fractional MHD coupled flow and heat transfer model for a generalized second-grade fluid. In Section \ref{sec:3}, some useful definitions and notation are given, and then the fully discrete spectral scheme is presented. Section \ref{sec:4} presents the stability and convergence analysis of the proposed numerical scheme. In Section \ref{sec:5}, we further develop the numerical scheme with a fast method and analyze the convergence of this scheme. In Section \ref{sec:6}, we verify the effectiveness of the numerical method through specific examples and numerically simulate the unsteady MHD flow and heat transfer of a generalized second-grade fluid through a porous medium. Finally, the conclusions to this study are summarized in Section \ref{sec:7}.

\section{Formulation of the flow and heat transfer problem}
\label{sec:2}
We study the unsteady MHD flow and heat transfer of a generalized second-grade fluid with a Hall current passing through a porous medium near a vertical infinite plate. The geometry of the problem is shown in Figure \ref{fig1}. The $x$-axis is taken along the plate, and the $z$-axis is perpendicular to the plate. We apply a magnetic field $\mathbf{B}$ that is horizontally parallel to the $z$-axis. Initially, the fluid and the plate are both at rest with a constant temperature of $T_{\infty}$. {When $t\ge 0$, we apply a force along the $x$-axis to the plate so that the plate starts to move up with velocity $U_0t^3$, then the fluid starts to move by the shear force,} and the temperature of the plate either rises or falls to $T_{\infty}+(T_w-T_{\infty})t^2/t_0^2$. For $t\ge t_0$, the plate is maintained at a constant temperature of $T_w$. We also assume that the incompressible magnetic fluid obeys the Boussinesq approximation and we neglect electromagnetic induction.

\begin{figure}[htbp]
\centering
\includegraphics[width=6.0cm]{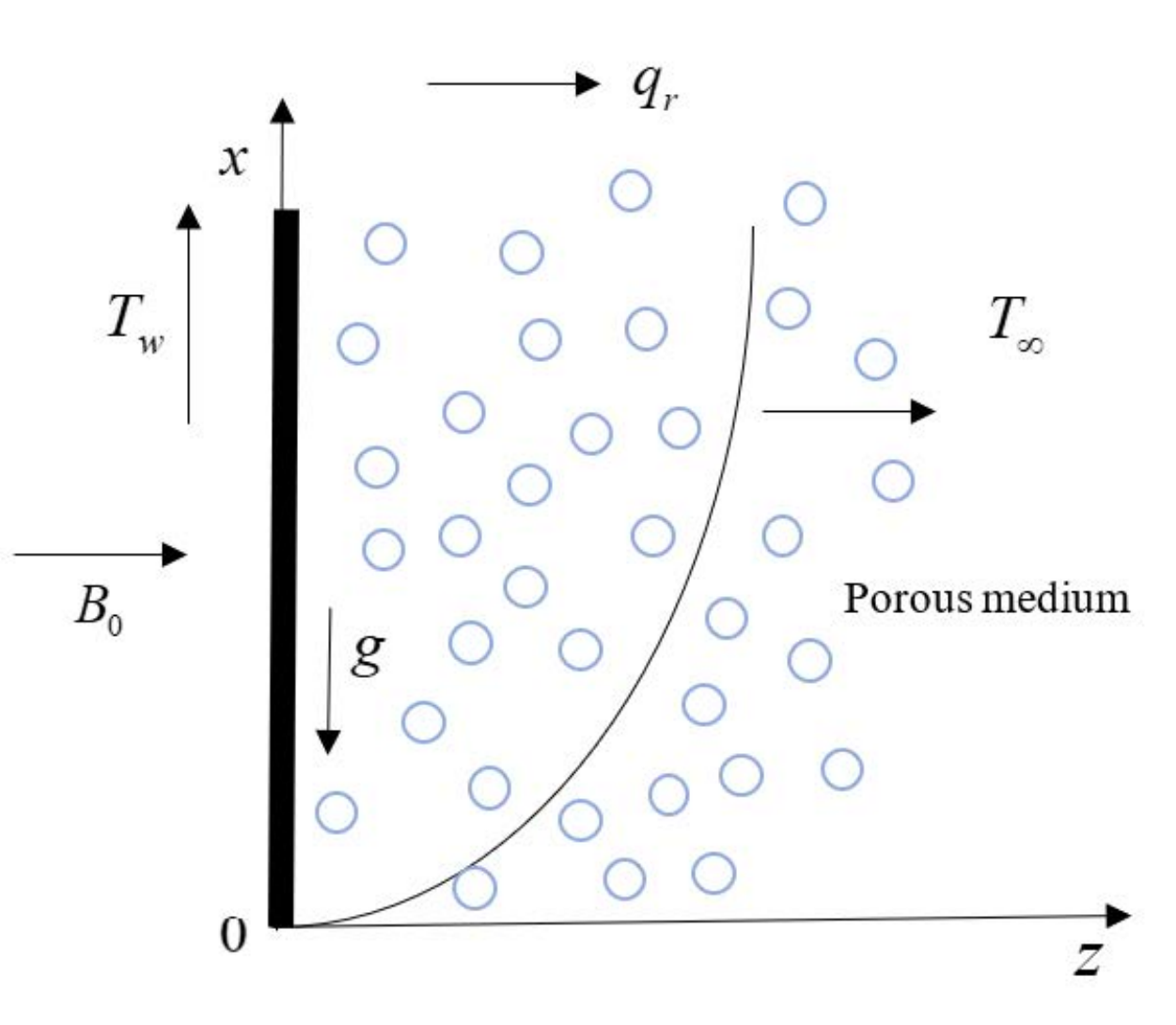}
\caption{ Geometry of the problem.}
\label{fig1}
\end{figure}

Considering the magnetic field and thermodiffusion, the momentum equation of the unsteady incompressible MHD flow in a porous medium is \cite{Jiang}
\begin{eqnarray}
\nabla \cdot\mathbf{V}=0,
\label{1}
\end{eqnarray}
\begin{equation}
\rho\left(\frac{\partial \mathbf{V}}{\partial t}+(\mathbf{V} \cdot \nabla) \mathbf{V}\right)=\nabla \cdot \mathbf{S}+\mathbf{J} \times \mathbf{B}+\rho \beta\mathbf{g} \left(T-T_{\infty}\right)+\mathbf{R},
\label{2}
\end{equation}
where $\mathbf{V}$ is the velocity vector, $\mathbf{S}$ is the Cauchy stress tensor, $\mathbf{J}$ is the current density vector, $\mathbf{B}$ is the magnetic field vector, $\mathbf{g}$ is the gravitational acceleration vector, $\rho$ is the density of the fluid, $\beta$ is the coefficient of thermal expansion, and $\mathbf{R}$ is Darcy's resistance in the porous medium.

Note that the velocity field is denoted by $\mathbf{V}=(u(z,t),v(z,t),0)$,  and we obtain the constitutive equation from (\ref{xj140}) as
\begin{equation}
\begin{array}{r}
\widetilde{S}_{x z}=\left(\mu+\alpha_{1}\ { }_{~~0} ^{RL}D_{t}^{\gamma}\right) \frac{\partial u}{\partial z}, \quad \widetilde{S}_{y z}=\left(\mu+\alpha_{1}\ { }_{~~0} ^{RL}D_{t}^{\gamma}\right) \frac{\partial v}{\partial z}, \\
\end{array}
\end{equation}
with $\widetilde{S}_{x x}=\widetilde{S}_{x y}=\widetilde{S}_{y y}=\widetilde{S}_{z z}=0$.

Based on the relation between the pressure gradient and the Darcian velocity, Darcy's resistance $\mathbf{R}$ for the generalized second-grade fluid in an unbounded porous medium is \cite{ey4}
\begin{equation}
\mathbf{R}=-\frac{\phi}{k_{1}}\left(\mu+\alpha_{1}\ {} _{~~0}^{RL}D_{t}^{\gamma}\right) \mathbf{V},
\end{equation}
where $\phi$ is the porosity and $k_1$ is the permeability of the porous medium.

The generalized Ohm's law with Hall effects can be written as \cite{28}
\begin{equation}
\mathbf{J}=\sigma(\mathbf{E}+\mathbf{V} \times \mathbf{B})-\frac{\sigma}{e n_{e}} \mathbf{J}\times \mathbf{B},
\label{4}
\end{equation}
where $\mathbf{E}$ is the electric field vector, $\sigma$ is the conductivity of the fluid, $e$ is the electric charge, and $n_e$ is the number density of electrons. In (\ref{4}), the electron pressure gradient, ion slip, and thermoelectric effect are all ignored. Furthermore, the magnetic field is given by $\mathbf{B}=(0,0,B_0)$ and we assume the electric field $\mathbf{E}=\mathbf{0}$. Then, $\mathbf{J}\times \mathbf{B}$ is calculated as
\begin{equation}
\mathbf{J}\times \mathbf{B}=\left(\frac{\sigma B_{0}^{2}}{1+m^{2}}(mv-u),-\frac{\sigma B_{0}^{2}}{1+m^{2}}(v+mu),0\right),
\label{5}
\end{equation}
where $m=\frac{\sigma B_0}{en_e}$ is the Hall parameter.

From (\ref{1})--(\ref{5}), the governing equation of momentum for the generalized second-grade fluid in a porous medium is
\begin{equation}
\begin{aligned}
\frac{\partial u}{\partial t}=&\left(\upsilon+\frac{\alpha_{1}}{\rho}{ }\ _{~~0}^{RL}D_{t}^{\gamma}\right) \frac{\partial^{2} u}{\partial z^{2}}+\frac{\sigma B_{0}^{2}}{\rho(1+m^{2})}(m v-u)-\frac{\phi}{k_{1}}\left(\upsilon+\frac{\alpha_{1}}{\rho}{ }\ _{~~0}^{RL}D_{t}^{\gamma}\right) u +g \beta_{T}\left(T-T_{\infty}\right),
\end{aligned}
\label{6e}
\end{equation}
\begin{equation}
\begin{aligned}
\frac{\partial v}{\partial t}&=\left(\upsilon+\frac{\alpha_{1}}{\rho}{ }\ _{~~0}^{RL} D_{t}^{\gamma}\right) \frac{\partial^{2} v}{\partial z^{2}}-\frac{\sigma B_{0}^{2}}{\rho(1+m^{2})}(v+m u)
-\frac{\phi}{k_{1}}\left(\upsilon+\frac{\alpha_{1}}{\rho}{ }\ _{~~0}^{RL} D_{t}^{\gamma}\right) v,
\end{aligned}
\label{6}
\end{equation}
where $\upsilon=\frac{\mu}{\rho}$ is the kinematic viscosity coefficient, and the pressure  effect in the flow is ignored .

The energy equation for temperature can be expressed as
\begin{equation}
\frac{\partial T}{\partial t}=-\frac{1}{\rho c_{p}} \nabla q+\frac{Q_{0}}{\rho c_{p}}\left(T-T_{\infty}\right),
\label{7}
\end{equation}
where $q$ is the heat flux, $Q_0$ is the heat generation/absorption coefficient, and $c_\rho$ is the specific heat capacity.

The generalized form of Fourier law of heat conduction involving the radiation heat effect $q_r$ can be written as \cite{29}
\begin{equation}
\left(1+\lambda_{1}\ {_{~~0}^{RL}D_{t}^{\beta}}\right) q=-k_0 \frac{\partial T}{\partial z}+q_r,
\label{8}
\end{equation}
where $k_0$ is the thermal conductivity and $\lambda_1$ is the relaxation time of temperature.


According to the Rosseland approximation, the radiative heat flux $q_{r}$ is \cite{30,31}
\begin{equation}
q_{r}=-\frac{4 \sigma^{*}}{3 k^{*}} \frac{\partial T^{4}}{\partial z},
\label{10}
\end{equation}
where $\sigma^{\ast} $ is the Stefan--Boltzmann constant and $k^{\ast} $ is the mean absorption coefficient.

We assume that the temperature difference within the flow is very small, so that $T^{4}$ can be expressed as a linear function of $T$. Here, $T^{4}$ can be expanded as a Taylor series expansion around $T_{\infty}$ without higher-order terms:
\begin{equation}
T^{4}=4 T_{\infty}^{3} T-3 T_{\infty}^{4}.
\label{11}
\end{equation}

Combining (\ref{7})--(\ref{11}), we obtain the following time-fractional heat conduction equation:
\begin{equation}
\begin{aligned}
\left(1+\lambda_{1}\ {_{~~0}^{RL}D_{t}^{\beta}}\right) \frac{\partial T}{\partial t}=&\frac{k_0}{\rho c_{p}} \frac{\partial^{2} T}{\partial z^{2}}+\frac{1}{\rho c_{p}}\frac{16 \sigma^{*} T_{\infty}^{3}}{3 k^{*}} \frac{\partial^{2} T}{\partial z^{2}}+\frac{Q_{0}}{\rho c_{p}}\left(1+\lambda_{1}\ {_{~~0}^{RL}D_{t}^{\beta}}\right)\left(T-T_{\infty}\right).
\end{aligned}
\label{12}
\end{equation}

The initial and boundary conditions of (\ref{6e})--(\ref{6}) and (\ref{12}) are
\begin{equation}
\begin{array}{l}
u=0, v=0, T=T_{\infty}, \text { as } z \geq 0 \text { and } t=0,
\end{array}
\end{equation}
\begin{equation}
\begin{array}{l}
u=U_{0} t^{3}, v=0, \text { as } z=0 \text { and } t>0,
\end{array}
\end{equation}
\begin{equation}
T=\left\{\begin{array}{ll}
T_{\infty}+\left(T_{w}-T_{\infty}\right) t^2 / t_{0}^2 & \text { if } 0<t<t_0, {z=0}, \\
T_{w} & \text { if } t \geq t_{0}, {z=0},
\end{array}\right.
\end{equation}
\begin{equation}
\begin{array}{l}
u \rightarrow 0, v \rightarrow 0, T \rightarrow T_{\infty}, \text { as } z \rightarrow \infty \text { and } t>0.
\end{array}
\end{equation}

We now introduce the following dimensionless variables:
$$
\widehat{z}=\frac{U_{0} t_{0}^{3} z}{\upsilon}, \quad \widehat{u}=\frac{u}{t_{0}^{3} U_{0}},\quad \widehat{v}=\frac{v}{t_{0}^{3} U_{0}},\quad \widehat{t}=\frac{t}{t_{0}},\quad
\theta=\frac{T-T_{\infty}}{T_{w}-T_{\infty}}, \quad M^{2}=\frac{\sigma B_{0}^{2} t_{0}}{\rho},  \quad K=\frac{k_{1}}{\phi \upsilon t_{0}},$$
$$
\alpha=\frac{\alpha_{1}}{\mu t_{0}^{\gamma}}, \quad \lambda=\frac{\lambda_{1}}{ t_{0}^{\beta}},\quad Gr=\frac{g \beta_{T}\left(T_{w}-T_{\infty}\right)}{U_{0} t_{0}^2},
\quad R=\frac{16 \sigma^{*} T_{\infty}^{3}}{3 k_0 k^{*}}, \quad H=\frac{Q_{0} v t_{0}}{k_0}, \quad Pr=\frac{\rho C_{p} \upsilon}{k_0},
$$
where $M$ is the Hartmann number, $K$ is the permeability parameter, $Gr$ is the thermal Grashof number, $R$ is the thermal radiation parameter, $H$ is the heat absorption/generation parameter, and $Pr$ is the Prandtl number. For simplicity, $t_0=(\frac{\upsilon}{U^2_0})^{\frac{1}{7}}$. The fractional viscoelastic coefficient $\alpha_1$ can be denoted as $\alpha_1=\alpha_0 c_0$, where $c_0$ is a dimension-balancing coefficient with a dimension of $[t_0]^{\gamma-1}$. The dimension of the fractional relaxation time $\lambda_1$ can be regarded as $[t_0]^{\beta}$.

Applying the dimensionless variables, the governing equations of (\ref{6e})--(\ref{6}) and (\ref{12}) become (omitting $\,\widehat{}\,$ for simplicity)
\begin{equation}
\begin{aligned}
\frac{\partial u}{\partial t}&=\left(1+\alpha\ {_{~~0}^{RL}D_{t}^{\gamma}}\right) \frac{\partial^{2} u}{\partial z^{2}}+\frac{M^{2}}{1+m^{2}}(m v-u)
-\frac{1}{K}\left(1+\alpha\ {_{~~0}^{RL}D_{t}^{\gamma}}\right) u+Gr\theta,
\end{aligned}
\end{equation}
\begin{equation}
\begin{aligned}
\frac{\partial v}{\partial t}&=\left(1+\alpha\ {_{~~0}^{RL}D_{t}^{\gamma}}\right) \frac{\partial^{2} v}{\partial z^{2}}-\frac{M^{2}}{1+m^{2}}(v+m u)
-\frac{1}{K}\left(1+\alpha\ {_{~~0}^{RL}D_{t}^{\gamma}}\right) v,
\end{aligned}
\end{equation}
\begin{equation}
\begin{aligned}
(1+\lambda\ {_{~~0}^{RL}D_t^{\beta}})\frac{\partial \theta}{\partial t}&=\frac{1+R}{Pr}\frac{\partial^2 \theta}{\partial z^2}+\frac{H}{Pr}(1+\lambda\ {_{~~0}^{RL}D_t^{\beta}})\theta,
\end{aligned}
\end{equation}
with
\begin{equation}
u=0, v=0, \theta=0, \text { as } z \geq 0 \text { and } t=0,
\end{equation}
\begin{equation}
u=t^{3}, v=0, \text { as } z=0 \text { and } t > 0,
\end{equation}
\begin{equation}
\theta=\left\{\begin{array}{ll}
t^2 & \text { if } 0<t<1, {z=0},\\
1 & \text { if } t \geq 1, {z=0},
\end{array}\right.
\end{equation}
\begin{equation}
u \rightarrow 0, v \rightarrow 0, \theta \rightarrow 0, \text { as } z \rightarrow \infty \text { and } t > 0.
\end{equation}

Considering $t\in[0,\bar{T}]$, $z \in[0,L]$, where $0<\bar{T}\le1$ and $L$ is a constant corresponding to $z\to \infty$, we transform the non-homogeneous boundary conditions into homogeneous boundary conditions. Let
\begin{equation*}
\begin{aligned}
w(z,t)&=u(z,t)-t^3(1-\frac{z}{L}),\quad \widetilde{\theta}=\theta(z,t)-t^2(1-\frac{z}{L}),
\end{aligned}
\end{equation*}
so that $w(z,t), v(z,t),\widetilde{\theta}(z,t)$ satisfy the following system:
\begin{equation}
\begin{aligned}
&\frac{\partial w}{\partial t}-\left(1+\alpha\ {_{~~0}^{RL}D_{t}^{\gamma}}\right) \frac{\partial^{2} w}{\partial z^{2}}-\frac{M^{2}}{1+m^{2}}(m v-w)
+\frac{1}{K}\left(1+\alpha\ {_{~~0}^{RL}D_{t}^{\gamma}}\right) w-Gr\widetilde{\theta}=f(z,t),
\end{aligned}
\label{m1}
\end{equation}
\begin{equation}
\begin{aligned}
\frac{\partial v}{\partial t}-\left(1+\alpha\ {_{~~0}^{RL}D_{t}^{\gamma}}\right) \frac{\partial^{2} v}{\partial z^{2}}+\frac{M^{2}}{1+m^{2}}(v+m w)
+\frac{1}{K}\left(1+\alpha\ {_{~~0}^{RL}D_{t}^{\gamma}}\right) v=g(z,t),
\end{aligned}
\label{m2}
\end{equation}
\begin{equation}
\begin{aligned}
(1+\lambda\ {_{~~0}^{RL}D_{t}^{\beta}})\frac{\partial \widetilde{\theta}}{\partial t}-\frac{1+R}{Pr}\frac{\partial^2 \widetilde{\theta}}{\partial z^2}
-\frac{H}{Pr}(1+\lambda\ {_{~~0}^{RL}D_{t}^{\beta}})\widetilde{\theta}=p(z,t),
\end{aligned}
\label{m3}
\end{equation}
\begin{equation}
w(z,0)=0, v(z,0)=0, \widetilde{\theta}(z,0)=0,\quad z\in[0,L],
\label{m4}
\end{equation}
\begin{equation}
w(0,t)=0, v(0,t)=0, \widetilde{\theta}(0,t)=0, \quad t\in[0,\bar{T}],
\label{m5}
\end{equation}
\begin{equation}
w(L,t)=0, v(L,t)=0, \widetilde{\theta}(L,t)=0, \quad t\in[0,\bar{T}],
\label{m6}
\end{equation}
where
\begin{eqnarray*}
\begin{split}
f(z,t)=&\left(-3t^2-\frac{6\alpha}{K\Gamma(4-\gamma)}t^{3-\gamma}-\frac{M^2}{1+m^2}t^3+\frac{1}{K}t^3
+Grt^2\right)(1-\frac{z}{L}),\\
g(z,t)=&-\frac{M^2m}{1+m^2}t^3(1-\frac{z}{L}),\\
p(z,t)=&\left(-2t+\-\frac{2\lambda}{\Gamma(2-\beta)}t^{1-\beta}+\frac{H}{Pr}t^{2}+\frac{2\lambda H}{Pr\Gamma(3-\beta)}t^{2-\beta}\right)(1-\frac{z}{L}).
\end{split}
\end{eqnarray*}


\section{Numerical method}
\label{sec:3}
\subsection{Preliminaries and notation}
Denoting $\Omega=[0,L]$, we define $L^2(\Omega)$ as an $L^2$-space with the inner product
$(u,v)_\Omega=\int_\Omega uv dz$ and the norm $\lVert u\lVert_\Omega=\bigg(\int_\Omega u^2 dz\bigg)^{\frac{1}{2}}$. When not confusing, we usually omit the subscript $\Omega$.
$H^l(\Omega)$ is the usual Sobolev space defined as $H^l(\Omega)=\{u\lvert u\in L^2(\Omega),\frac{\partial^k u}{\partial z^k}\in L^2(\Omega), 1\le k\le l\}$, and its norm is denoted by $$\lVert u\rVert_{l,\Omega}= \bigg(\|u\|^2+\sum\limits_{k=1}^l\left\|\frac{\partial^k u}{\partial z^k}\right\|^2 \bigg).$$
Furthermore, $ H_{0}^{1}(\Omega)=\left\{u \lvert u\in H^{1}(\Omega), u|_{\partial \Omega}=0\right\}$.

Let $Y$ be a Banach space,and the space $L^{2}(0,\bar{T};Y)$ is defined as
\begin{equation*}
L^{2}(0,\bar{T};Y)=\bigg\{v:(0,\bar{T})\to Y \bigg| \int^{\bar{T}}_0 \lVert v\rVert^2_Y dt <+\infty\bigg\},
\end{equation*}
and is endowed with the norm
$$
\lVert v\rVert_{L^{2}(Y)}= \big( \int^{\bar{T}}_0 \lVert v\rVert^2_Y dt\big)^{\frac{1}{2}}<+\infty.
$$

Let $N$ be a positive integer, and let $P_N(\Omega)$ be the space of all polynomials with a degree no greater than $N$. Then, the approximation space
$V^0_{N}$ is defined as
$$
V_{N}^0=P_{N}(\Omega)\cap H^1_0(\Omega).
$$
We define $\Pi_{N}^{1,0}$ as the $H_{0}^{1}$-orthogonal projection operator $H_{0}^{1}(\Omega)$ $\to$ $V_{N}^{0}$ such that, for all $u \in H_{0}^{1}(\Omega)$, we have
$$
\left(\partial_{z}\left(\Pi_{N}^{1,0} u-u\right), \partial_{z} \varphi\right)=0, \quad \forall \varphi \in V_{N}^{0}.
$$

For the projection operator $\Pi_{N}^{1,0}$, we have the following approximation result.
\begin{lemma}\cite{40}
For all $u \in H_{0}^{1}(\Omega) \cap H^{r}(\Omega)$, we have
$$\|u-\Pi_{N}^{1,0} u\|_{l, \Omega} \leq C N^{l-r}\|u\|_{r, \Omega}, \quad l=0,1, r \geq 1,$$
where $C$ is a positive constant that is independent of $N$.
\label{le3.1}
\end{lemma}

We now introduce some definitions and properties of fractional operators.
For a given function $u(t)$ and $s>0$, the left-sided Riemann--Liouville fractional integral operator of order $s$ is defined as \cite{33}
\begin{eqnarray}
_{~~0}^{RL}D^{-s}_{t}u(t)=\frac{1}{\Gamma(s)}\int^t_0\frac{u(\xi)}{(t-\xi)^{1-s}}d\xi,\quad t>0.
\end{eqnarray}
For $m-1<s<m$, the left-sided Riemann--Liouville fractional derivative  operator of order $s$ is defined as \cite{33}
\begin{eqnarray}
_{~~0}^{RL}D^{s}_{t}u(t)=\frac{1}{\Gamma(m-s)}\frac{d^m}{dt^m}\int^t_0\frac{u(\xi)}{(t-\xi)^{s+1-m}}d\xi,\quad t>0.
\end{eqnarray}
For $m-1<s<m$, the left-sided Caputo fractional derivative  operator of order $s$ is defined as \cite{33}
\begin{eqnarray}
_{0}^{C}D^{s}_{t}u(t)=\frac{1}{\Gamma(m-s)}\int^t_0\frac{u^{(m)}(\xi)}{(t-\xi)^{s+1-m}}d\xi,\quad t>0.
\end{eqnarray}
We also give the following useful properties of fractional calculations \cite{34}:
\begin{eqnarray}
_{~~0}^{RL}D^{-s}_{t}{_{0}^{C}D^{s}_{t}}u(t)=u(t)-\sum^{m-1}_{k=0}\frac{u^{(k)}(0)}{k!}t^k, \quad s >0,
\label{x1}
\end{eqnarray}
\begin{eqnarray}
_{~~0}^{RL}D^{-s}_{t}{_{~~0}^{RL}D^{\eta}_{t}}u(t)={_{~~0}^{RL}D^{-s+\eta}_{t}}u(t),\quad s >0, 0<\eta<1 \text{~or~} s >0, 1<\eta<2, u(0)=0.
\label{x2}
\end{eqnarray}

\subsection{The fully discrete  scheme}
For simplicity, we replace the notations $w, \widetilde{\theta}$ with $u, {\theta}$ in (\ref{m1})--(\ref{m6}), and apply $_{~~0}^{RL}D_t^{-\beta}$ to (\ref{m3}). By the properties stated in (\ref{x1})--(\ref{x2}), we can simplify the MHD flow and heat transfer problem in (\ref{m1})--(\ref{m3}) to the following time-fractional coupled equations:
\begin{equation}
\begin{aligned}
\frac{\partial u}{\partial t}+ a_1\ {_{~~0}^{RL}D_t^{\gamma}}u-a_2\ {_{~~0}^{RL}D_t^{\gamma}}\frac{\partial^{2} u}{\partial z^{2}}-\frac{\partial^{2} u}{\partial z^{2}}+a_3u-a_4v-a_5\theta=f(z,t),
\end{aligned}
\label{1.1}
\end{equation}
\begin{equation}
\begin{aligned}
\frac{\partial v}{\partial t}+a_1\ {}_{~~0}^{RL}D_t^{\gamma}v-a_2\ {_{~~0}^{RL}D_t^{\gamma}}\frac{\partial^{2} v}{\partial z^{2}}-\frac{\partial^{2} u}{\partial z^{2}}+a_3v+a_4u=g(z,t),
\label{1.2}
\end{aligned}
\end{equation}
\begin{equation}
\begin{aligned}
\frac{\partial\theta}{\partial t}+b_1{_{~~0}^{RL}D_t^{1-\beta}}\theta-b_2{_{~~0}^{RL}D_t^{-\beta}}\frac{\partial^{2} \theta}{\partial z^{2}}-b_3{_{~~0}^{RL}D_t^{-\beta}}\theta-b_4\theta=\widetilde{p}(z,t),
\end{aligned}
\label{1.3}
\end{equation}
where $a_1=\frac{\alpha}{K}, a_2=\alpha, a_3=\frac{M^2}{1+m^2}+\frac{1}{K}, a_4=\frac{M^2m}{1+m^2}, a_5=Gr, b_1=\frac{1}{\lambda}, b_2=\frac{1+R}{Pr\lambda}, b_3=\frac{H}{Pr\lambda},b_4=\frac{H}{Pr}$, and $\widetilde{p}(z,t)={_{~~0}^{RL}D_t^{-\beta}}p(z,t)$.

Let $\bar{K}$ denote a positive integer, and $t_k=k\tau,k=1,\dots, \bar{K}$, where $\tau=\bar{T}/\bar{K}$ is the time step size. Taking $u^k=u(\cdot,t_k)$, we have the following approximation result for the first-order derivative:
\begin{eqnarray}
\left[\frac{\partial u}{\partial t}\right]_{t=t_{k}}=\left\{ \begin{array}{ll}
\frac{u^1-u^0}{\tau}+O(\tau), & k=1,\\
\frac{3u^k-4u^{k-1}+u^{k-2}}{2\tau}+O(\tau^2), & k\ge2.
\end{array} \right.
\end{eqnarray}

The fractional derivatives are approximated by the fractional backward difference formula (FBDF) method \cite{Yin}. We define the following notation for $k\ge 1$:
\begin{eqnarray}
\begin{aligned}
\bigg[{ }_{~~0}^{RL} D_{t}^{\gamma} u\bigg]_{t=t_{k}}=D_{\tau}^{\gamma} u^{k}+O(\tau^{2}),\quad
\bigg[{ }_{~~0}^{RL} D_{t}^{\gamma} \frac{\partial^{2} u}{\partial z^{2}}\bigg]_{t=t_{k}}=D_{\tau}^{\gamma} \frac{\partial^{2} u^{k}}{\partial z^{2}}+O(\tau^{2}), \\
\end{aligned}
\end{eqnarray}
in which
$$
D_{\tau}^{\gamma} u^{k}=\frac{1}{\tau^{\gamma}} \sum_{l=0}^{k} \omega_{k-l}^{(\gamma)} u^{l}, \quad -1\le \gamma\le 1. 
$$
Here, $\left\{\omega_{k}^{(\gamma)}\right\}$ are the coefficients of the Taylor expansion of the following generating function:
\begin{eqnarray}
\omega^{(\gamma)}(x)=\left(\frac{3}{2}-2 x+\frac{1}{2} x^{2}\right)^{\gamma}=\sum_{k=0}^{\infty} \omega_{k}^{(\gamma)} x^{k}.
\label{FBDF}
\end{eqnarray}

Then, the time semi-discrete scheme of (\ref{1.1})--(\ref{1.3}) at $t=t_{k}$ is
\begin{eqnarray}
\begin{split}
\partial_t u^{k}&+a_1D^{\gamma}_{\tau}u^{k}
+a_2D^{\gamma}_{\tau}\Delta u^{k}+\Delta u^{k}+a_3u^{k}-a_4v^k-a_5\theta^k=f^{k}+R_1^k,
\end{split}
\label{4.2}
\end{eqnarray}
\begin{eqnarray}
\begin{split}
\partial_t v^{k}+a_1D^{\gamma}_{\tau}v^{k}+a_2D^{\gamma}_{\tau}\Delta v^{k}+\Delta v^{k}
+a_3v^{k}+a_4u^k=g^{k}+R_2^k,
\end{split}
\label{4.3}
\end{eqnarray}
\begin{eqnarray}
\begin{split}
\partial_t \theta^{k}+b_1D^{1-\beta}_{\tau}\theta^{k}
+b_2D^{-\beta}_{\tau}\Delta \theta^{k}-b_3D^{-\beta}_{\tau} \theta^{k}
-b_4\theta^{k}=\widetilde{p}^{k}+R_3^k,
\end{split}
\label{4.4}
\end{eqnarray}
where 
$
\partial_t u_{N}^{k}=\left\{ \begin{array}{ll}
\frac{u^1-u^0}{\tau}, & k=1,\\
\frac{3u^k-4u^{k-1}+u^{k-2}}{2\tau}, & k\ge2,
\end{array} \right.,$
$R_i^k=\left\{ \begin{array}{ll}
O(\tau), & k=1,\\
O(\tau^2), & k\ge2,
\end{array} \right.,$ for $i=1,2,3$.

Based on the time semi-discrete scheme(\ref{4.2})--(\ref{4.4}), the fully discrete Legendre spectral scheme of (\ref{1.1})--(\ref{1.3}) determines $u_{N}^{k}, v_{N}^{k}, \theta_{N}^{k}\in V_{N}^0$ such that, for any $\varphi_{N}\in V_{N}^0$ and $k\ge 1$, we have
\begin{eqnarray}
\begin{aligned}
&\big(\partial_t u_{N}^{k},\varphi_{N}\big)+a_1\big(D^{\gamma}_{\tau}u_{N}^{k},\varphi_{N}\big)
+a_2\big(D^{\gamma}_{\tau}\nabla u_{N}^{k},\nabla \varphi_{N}\big)+\big(\nabla u_{N}^{k},\nabla \varphi_{N}\big)
+a_3\big(u_{N}^{k}, \varphi_{N}\big)-a_4\big(v^k_{N}, \varphi_{N}\big)\\&-a_5\big(\theta^k_{N}, \varphi_{N}\big)=\big(f^{k},\varphi_{N}\big),
\end{aligned}
\label{4.5}
\end{eqnarray}
\begin{eqnarray}
\begin{split}
&\big(\partial_t v_{N}^{k},\varphi_{N}\big)+a_1\big(D^{\gamma}_{\tau}v_{N}^{k},\varphi_{N}\big)
+a_2\big(D^{\gamma}_{\tau}\nabla v_{N}^{k},\nabla \varphi_{N}\big)+\big(\nabla v_{N}^{k},\nabla \varphi_{N}\big)
+a_3\big(v_{N}^{k}, \varphi_{N}\big)+a_4\big(u^k_{N}, \varphi_{N}\big)\\&=\big(g^{k},\varphi_{N}\big),
\end{split}
\label{4.6}
\end{eqnarray}
\begin{eqnarray}
\begin{split}
&\big(\partial_t \theta_{N}^{k},\varphi_{N}\big)+b_1\big(D^{1-\beta}_{\tau}\theta_{N}^{k},\varphi_{N}\big)
+b_2\big(D^{-\beta}_{\tau}\nabla \theta_{N}^{k},\nabla \varphi_{N}\big)-b_3\big(D^{-\beta}_{\tau} \theta_{N}^{k},\varphi_{N}\big)
-b_4(\theta_{N}^{k},\varphi_{N}\big)=\big(\widetilde{p}^{k},\varphi_{N}\big),
\end{split}
\label{4.7}
\end{eqnarray}
with the initial conditions $u_{N}^{0}=0, v_{N}^{0}=0, \theta_{N}^{0}=0$.
\section{Theoretical analysis of stability and convergence}
\label{sec:4}
In the remainder of this paper, let $C, C_{s,i}, C_{e,i}, C_{b,i},C_{f,i}, i\in \mathbb{N}^{+}$ denote generic positive constants independent of $\tau$ and $N$, and we assume $a_i>0, b_i>0$ in the theoretical analysis for convenience.  First, we present some lemmas  which are helpful for analyzing the stability and convergence of the numerical scheme in (\ref{4.5})--(\ref{4.7}).
\begin{lemma}\cite{35}
For a series ${u^n} (n\ge2) $, the following results hold:
\begin{eqnarray}
\big(\partial_t u^{n},u^{n}\big)\ge \frac{1}{4\tau}\big(\Psi[u^n]-\Psi[u^{n-1}]\big),
\end{eqnarray}
{where}
\begin{eqnarray}
\Psi[u^n]=3\lVert u^n\rVert^2-\lVert u^{n-1}\rVert^2+2\lVert u^n-u^{n-1}\rVert^2,
\end{eqnarray}
and
\begin{eqnarray}
\Psi[u^n]\ge \lVert u^n\rVert^2.
\end{eqnarray}
\label{le4.1}
\end{lemma}
{
\begin{lemma}(Gronwall inequality) \cite{36}
Suppose that $q_0\ge0$, $\{s_n\},\{\sigma_n\}$ are non-negative sequences and that $\{\sigma_n\}$ satisfies
$$
\left\{ \begin{array}{ll}
\sigma_0\leq q_0, & \\
\sigma_n\leq q_0+\tau\sum^{n-1}\limits_{j=0}s_j\sigma_j, \quad n\ge1.
\end{array} \right.
$$
Then, it follows that
$$
\sigma_n\leq q_0\exp\bigg(\tau\sum^{n-1}_{j=0}s_j\bigg),\quad n\ge1.
$$
\label{le4.2}
\end{lemma}}

\begin{lemma}
Let $\left\{\omega_{k}^{(\gamma)}\right\}, \gamma \in(-1,1)$ be the coefficients of $\omega^{(\gamma)}(x)$ defined by (\ref{FBDF}). Then, for any vector $\left(u^{0}, u^{1}, \ldots, u^{n-1}\right) \in \mathbb{R}^{n}$, we have that
$$
\sum_{k=0}^{n-1} u^{k} \sum_{j=0}^{k} \omega_{k-j}^{(\gamma)} u^{j} \geq 0, \text { for any } n \geq 1.
$$
\label{le4.3}
\end{lemma}
\begin{proof}
When $\gamma \in(-1,1)$, the proof of this lemma is similar to the proof of Lemma 3.1 in \cite{Yin}. It is clear that the relevant conclusion holds.
\end{proof}

By virtue of this lemma, the following corollary is easily obtained.
\begin{corollary}
Let $\left\{\omega_{k}^{(\gamma)}\right\}, \gamma \in(-1,1)$ be the coefficients of $\omega^{(\gamma)}(x)$ defined by (\ref{FBDF}). Then, for any vector $\left(u^{0}, u^{1}, \ldots, u^{n-1}\right) \in \mathbb{R}^{n}$, we have that
$$
\sum_{k=0}^{n-1}\sum_{j=0}^{k} \omega_{k-j}^{(\gamma)} (u^{j}, u^{k}) \geq 0, \text { for any } n \geq 1.
$$
\label{co4.4}
\end{corollary}

Next, we present the stability theorem for the fully discrete spectral scheme in (\ref{4.5})--(\ref{4.7}) as follows.
\begin{theorem}
Solutions $u_{N}^n, v_{N}^n,\theta_{N}^n $ of the fully discrete spectral scheme in (\ref{4.5})--(\ref{4.7}) are stable and satisfy
\begin{eqnarray}
\begin{split}
\lVert u_{N}^n\lVert^2+\lVert v_{N}^n\lVert^2+\lVert \theta_{N}^n\lVert^2\le C\tau\bigg(\sum^n\limits_{k=1}\lVert {f}^k\lVert^2+\sum^n\limits_{k=1}\lVert {g}^k\lVert^2+\sum^n\limits_{k=1}\lVert\widetilde{p}^k\lVert^2\bigg).\\
\end{split}
\label{s1}
\end{eqnarray}
\label{th4.5}
\end{theorem}

\begin{proof}
For $k\ge 2$, we take $\varphi_{N}=u^{k}_{N}$ in (\ref{4.5}) and use Lemma \ref{le4.1} to obtain
\begin{eqnarray}
\begin{split}
&\frac{1}{4\tau}(\Psi[u_{N}^k]-\Psi[u_{N}^{k-1}])+a_1\big(D^{\gamma}_{\tau}u_{N}^{k},u^{k}_{N}\big)
+a_2\big(D^{\gamma}_{\tau}\nabla u_{N}^{k},\nabla u^{k}_{N}\big)\\&+\big(\nabla u_{N}^{k},\nabla u^{k}_{N}\big)
+a_3\big(u_{N}^{k}, u^{k}_{N}\big)-a_4\big(v^k_{N}, u^{k}_{N}\big)-a_5\big(\theta^k_{N}, u^{k}_{N}\big)\le\big(f^{k},u^{k}_{N}\big).
\end{split}
\label{s2}
\end{eqnarray}

We sum (\ref{s2}) over $k$ from 2 to $n$ and multiply both sides of this inequality by $4\tau$. By the Cauchy--Schwarz inequality and Young's inequality, we have that
\begin{eqnarray}
\begin{split}
&\Psi[u_{N}^n]+4a_1\tau^{1-\gamma}\sum^n\limits_{k=2}\sum^{k}\limits_{j=0}\omega_{k-j}^{(\gamma)}\big(u^{j}_{N},u^{k}_{N}\big)
+4a_2\tau^{1-\gamma}\sum^n\limits_{k=2}\sum^{k}\limits_{j=0}\omega_{k-j}^{(\gamma)}\big(\nabla u^{j}_{N},\nabla u^{k}_{N}\big)\\&+4\tau\sum_{k=2}^n\big(\nabla u_{N}^{k},\nabla u^{k}_{N}\big)
+4a_3\tau\sum_{k=2}^n\big(u_{N}^{k}, u^{k}_{N}\big)-4a_4\tau\sum_{k=2}^n\big(v^k_{N}, u^{k}_{N}\big)-4a_5\tau\sum_{k=2}^n\big(\theta^k_{N}, u^{k}_{N}\big)\\
\le& \Psi[u_{N}^1]+2\tau\sum^n\limits_{k=2}\lVert{f}^k\lVert^2+2\tau\sum^n\limits_{k=2}\lVert u_{N}^k\lVert^2.
\end{split}
\label{s6}
\end{eqnarray}
Next, we consider the case of $k=1$ and take $\varphi_{N}=u_{N}^{1}$ in (\ref{4.5}). Note that
\begin{eqnarray}
\begin{split}
\big(\frac{u_{N}^1-u_{N}^0}{\tau}, u_{N}^{1}\big)
=\frac{\lVert u_{N}^1\lVert^2-\lVert u_{N}^0\lVert^2}{2\tau}+\frac{1}{2\tau}\lVert u_{N}^1- u_{N}^0\lVert^2,
\end{split}
\label{s7}
\end{eqnarray}
and so we obtain
\begin{eqnarray}
\begin{split}
&\frac{\lVert u_{N}^1\lVert^2-\lVert u_{N}^0\lVert^2}{2\tau}+\frac{1}{2\tau}\lVert u_{N}^1- u_{N}^0\lVert^2+a_1\big(D^{\gamma}_{\tau}u_{N}^{1},u_{N}^{1}\big)
+a_2\big(D^{\gamma}_{\tau}\nabla u_{N}^{1},\nabla u_{N}^{1}\big)\\&+\big(\nabla u_{N}^{1},\nabla u_{N}^{1}\big)
+a_3\big(u_{N}^{1}, u_{N}^{1}\big)-a_4\big(v^1_{N}, u_{N}^{1}\big)-a_5\big(\theta^1_{N}, u_{N}^{1}\big)=\big(f^{1},u_{N}^{1}\big).
\end{split}
\label{s8}
\end{eqnarray}
Multiplying (\ref{s8}) by $2\tau$ and using $u_N^0=0$, the Cauchy--Schwarz inequality and Young's inequality then imply that
{
\begin{eqnarray}
\begin{split}
&2\lVert u_{N}^1\lVert^2+2a_1\tau^{1-\gamma}\sum_{j=0}^{1}\omega_{1-j}^{(\gamma)}\big(u_{N}^{j},u_{N}^{1}\big)+2a_2\tau^{1-\gamma}
\sum_{j=0}^{1}\omega_{1-j}^{(\gamma)}\big(\nabla u_{N}^{j},\nabla u_{N}^{1}\big)\\&+2\tau\big(\nabla u_{N}^{1},\nabla u_{N}^{1}\big)
+2a_3\tau\big(u_{N}^{1}, u_{N}^{1}\big)-2a_4\tau\big(v^1_{N}, u_{N}^{1}\big)-2a_5\tau\big(\theta^1_{N}, u_{N}^{1}\big)\le 2\tau\big( \lVert u_{N}^1\lVert^2+\lVert {f}^1\lVert^2\big).
\end{split}
\label{s9}
\end{eqnarray}
Due to $\Psi[u_{N}^1]=5\lVert u_{N}^1\lVert^2$ and Corollary \ref{co4.4}, then we have
\begin{eqnarray}
\begin{split}
\Psi[u_{N}^1]\le&-4a_1\tau^{1-\gamma}\sum_{j=0}^{1}\omega_{1-j}^{(\gamma)}\big(u_{N}^{j},u_{N}^{1}\big)-4a_2\tau^{1-\gamma}
\sum_{j=0}^{1}\omega_{1-j}^{(\gamma)}\big(\nabla u_{N}^{j},\nabla u_{N}^{1}\big)-4\tau\big(\nabla u_{N}^{1},\nabla u_{N}^{1}\big) \\&-4a_3\tau\big(u_{N}^{1}, u_{N}^{1}\big)
+5a_4\tau\big(v^1_{N}, u_{N}^{1}\big)+5a_5\tau\big(\theta^1_{N}, u_{N}^{1}\big)+\frac{5\tau}{2}\big(\lVert u_N^1\lVert^2+\lVert {f}^1\lVert^2\big).
\end{split}
\label{s11}
\end{eqnarray}}
Bringing (\ref{s11}) into (\ref{s6}) and using Lemma \ref{le4.1}, we infer that
\begin{eqnarray}
\begin{split}
\lVert u_{N}^n\lVert^2 \le& -4a_1\tau^{1-\gamma}\sum^n\limits_{k=1}\sum^{k}\limits_{j=0}\omega_{k-j}^{(\gamma)}\big(u^{j}_{N},u^{k}_{N}\big)
-4a_2\tau^{1-\gamma}\sum^n\limits_{k=1}\sum^{k}\limits_{j=0}\omega_{k-j}^{(\gamma)}\big(\nabla u^{j}_{N},\nabla u^{k}_{N}\big)\\&-4\tau\sum_{k=1}^n\big(\nabla u_{N}^{k},\nabla u^{k}_{N}\big)
-4a_3\tau\sum_{k=1}^n\big(u_{N}^{k}, u^{k}_{N}\big)+4a_4\tau\sum_{k=2}^n\big(v^k_{N}, u^{k}_{N}\big)+5a_4\tau\big(v^1_{N}, u^{1}_{N}\big)\\
&+{5a_5\tau}\sum^n\limits_{k=1}\lVert {\theta}^k_N\lVert^2+ C_{s,1}\tau\sum^n\limits_{k=1}\lVert u_{N}^k\lVert^2+C_{s,2}\tau\sum^n\limits_{k=1}\lVert{f}^k\lVert^2.
\end{split}
\label{s12}
\end{eqnarray}
According to Corollary \ref{co4.4},
\begin{eqnarray}
\begin{split}
\lVert u_{N}^n\lVert^2\le &4a_4\tau\sum_{k=2}^n\big(v^k_{N}, u^{k}_{N}\big)+5a_4\tau\big(v^1_{N}, u^{1}_{N}\big)+{5a_5\tau}\sum^n\limits_{k=1}\lVert {\theta}^k_N\lVert^2+  C_{s,1}\tau\sum^n\limits_{k=1}\lVert u_{N}^k\lVert^2+C_{s,2}\tau\sum^n\limits_{k=1}\lVert{f}^k\lVert^2.
\end{split}
\label{s13}
\end{eqnarray}

Similarly, taking $\varphi_{N}=v^{k}_{N}$ in (\ref{4.6}), we obtain
\begin{eqnarray}
\begin{split}
\lVert v_{N}^n\lVert^2\le -4a_4\tau\sum_{k=2}^n\big(u^k_{N}, v^{k}_{N}\big)-5a_4\tau\big(u^1_{N}, v^{1}_{N}\big)+ C_{s,3}\tau\sum^n\limits_{k=1}\lVert v_{N}^k\lVert^2+C_{s,3}\tau\sum^n\limits_{k=1}\lVert{g}^k\lVert^2.
\end{split}
\label{s14}
\end{eqnarray}

Taking $\varphi_{N}=\theta^{k}_{N}$ in (\ref{4.7}), similarly to the previous proof, we obtain
\begin{eqnarray}
\begin{split}
\lVert \theta_{N}^n\lVert^2 \le &5b_3\tau^{1+\beta}\sum^n\limits_{k=1}\sum^{k}\limits_{j=0}\omega_{k-j}^{(-\beta)}\big(\theta^{j}_{N},\theta^{k}_{N}\big)
+5b_4\tau\sum_{k=1}^n\lVert\theta_{N}^{k}\lVert^2+C_{s,4}\tau\sum^n\limits_{k=1}\lVert \theta_{N}^k\lVert^2+C_{s,4}\tau\sum^n\limits_{k=1}\lVert\widetilde{p}^k\lVert^2.
\end{split}
\label{s17}
\end{eqnarray}
By $\omega^{(-\beta)}_0={(\frac{3}{2})}^{\beta}$, $\omega^{(-\beta)}_k=O(k^{\beta-1}) (k\ge 1)$, we then have
\begin{eqnarray}
\begin{split}
\sum^{n}\limits_{k=j}\tau^{\beta}\rvert\omega_{k-j}^{(-\beta)}\rvert \le\tau^{\beta}+\tau^{\beta}\sum^{n}\limits_{k=j+1}(k-j)^{\beta-1}\le C_{s,5}\bar{T}^{\beta},
\end{split}
\label{s20}
\end{eqnarray}
\begin{eqnarray}
\begin{split}
\sum^{k}\limits_{j=0}\tau^{\beta}\rvert\omega_{k-j}^{(-\beta)}\rvert \le\tau^{\beta}+\tau^{\beta}\sum^{k-1}\limits_{j=0}(k-j)^{\beta-1}\le C_{s,5}\bar{T}^{\beta}.
\end{split}
\label{s21}
\end{eqnarray}
Thus, we have that
\begin{eqnarray}
\begin{split}
5b_3\tau^{1+\beta}\sum^n\limits_{k=1}\sum^{k}\limits_{j=0}\omega_{k-j}^{(-\beta)}\big(\theta^{j}_{N},\theta^{k}_{N}\big) \le& \frac{5b_3\tau}{2}\sum^n\limits_{k=0}\sum^{k}\limits_{j=0}\tau^{\beta}\rvert\omega_{k-j}^{(-\beta)}\rvert\lVert\theta^{j}_{N}\lVert^2
+\frac{5b_3\tau}{2}\sum^n\limits_{k=0}\sum^{k}\limits_{j=0}\tau^{\beta}\rvert\omega_{k-j}^{(-\beta)}\rvert\lVert\theta^{k}_{N}\lVert^2\\
\le& \frac{5b_3\tau}{2}\sum^n\limits_{j=0}\sum^{n}\limits_{k=j}\tau^{\beta}\rvert\omega_{k-j}^{(-\beta)}\rvert\lVert\theta^{j}_{N}\lVert^2
+\frac{5b_3\tau}{2}\sum^n\limits_{k=0}\sum^{k}\limits_{j=0}\tau^{\beta}\rvert\omega_{k-j}^{(-\beta)}\rvert\lVert\theta^{k}_{N}\lVert^2\\
\le &5C_{s,5}\bar{T}^{\beta}b_3\tau \sum_{k=1}^n \lVert\theta^{k}_{N}\lVert^2.
\end{split}
\label{s18}
\end{eqnarray}
Then, the following result holds:
\begin{eqnarray}
\begin{split}
\lVert \theta_{N}^n\lVert^2 \le& 5C_{s,5}\bar{T}^{\beta}b_3\tau \sum_{k=1}^n \lVert\theta^{k}_{N}\lVert^2
+5b_4\tau\sum_{k=1}^n\lVert\theta_{N}^{k}\lVert^2+C_{s,4}\tau\sum^n\limits_{k=1}\lVert \theta_{N}^k\lVert^2+C_{s,4}\tau\sum^n\limits_{k=1}\lVert\widetilde{p}^k\lVert^2.
\end{split}
\label{s22}
\end{eqnarray}

After adding (\ref{s13})--(\ref{s14}) and (\ref{s22}), we apply Gronwall inequality for sufficiently small $\tau$ to obtain
\begin{eqnarray}
\begin{split}
\lVert u_{N}^n\lVert^2+\lVert v_{N}^n\lVert^2+\lVert \theta_{N}^n\lVert^2\le C\tau\bigg(\sum^n\limits_{k=1}\lVert {f}^k\lVert^2+\sum^n\limits_{k=1}\lVert {g}^k\lVert^2+\sum^n\limits_{k=1}\lVert\widetilde{p}^k\lVert^2\bigg).\\
\end{split}
\end{eqnarray}
\end{proof}

We now present the convergence analysis for the fully discrete spectral scheme in (\ref{4.5})--(\ref{4.7}).
\begin{theorem}{
Let $u$, $v$, $\theta$, and $\{u_{N}^k\}_{k=0}^n$, $\{v_{N}^k\}_{k=0}^n$, $\{\theta_{N}^k\}_{k=0}^n$ be the solutions of (\ref{1.1})--(\ref{1.3}) and (\ref{4.5})--(\ref{4.7}), respectively. Suppose that $u\in C^2(0,\bar{T};H^r(\Omega))$, $v\in C^2(0,\bar{T};H^r(\Omega))$, $\theta\in C^2(0,\bar{T};H^r(\Omega))$, $\frac{\partial u}{\partial t}\in L^2(0,\bar{T};H^r(\Omega))$, $\frac{\partial v}{\partial t}\in  L^2(0,\bar{T}; H^r(\Omega))$, $\frac{\partial \theta}{\partial t}\in L^2(0,\bar{T};H^r(\Omega))$, ${^{RL}_{~~0}D^{\gamma}_t}u\in L^2(0,\bar{T};H^r(\Omega))$, ${^{RL}_{~~0}D^{\gamma}_t}v\in L^2(0,\bar{T};H^r(\Omega))$, ${^{RL}_{~~0}D^{1-\beta}_t}\theta\in L^2(0,\bar{T}; H^r(\Omega))$, ${^{RL}_{~~0}D^{-\beta}_t}\theta\in L^2(0,\bar{T};H^r(\Omega))$, $r\ge1$.}
Then, we have
\begin{eqnarray}
\begin{split}
\lVert u^n- u_{N}^n\lVert^2+\lVert v^n- v_{N}^n\lVert^2+\lVert \theta^n- \theta_{N}^n\lVert^2\le & C (\tau^4+N^{-2r}).
\end{split}
\label{e1}
\end{eqnarray}
\label{th4.6}
\end{theorem}

\begin{proof}
Denote $e_1^{k}=u-u_N^{k}$, $e_2^{k}=v-v_N^{k}$, $e_3^{k}=\theta-\theta_N^{k}$, for $k\ge1$. By the original equivalent equations (\ref{1.1})--(\ref{1.3}) and the fully discrete spectral scheme (\ref{4.5})--(\ref{4.7}), the error equations can be written as
\begin{eqnarray}
\begin{split}
&\big(\partial_t e_{1}^{k},\varphi_{N}\big)+a_1\big(D^{\gamma}_{\tau}e_{1}^{k},\varphi_{N}\big)
+a_2\big(D^{\gamma}_{\tau}\nabla e_{1}^{k},\nabla \varphi_{N}\big)+\big(\nabla e_{1}^{k},\nabla \varphi_{N}\big)
+a_3\big(e_{1}^{k}, \varphi_{N}\big)-a_4\big(e^k_{2}, \varphi_{N}\big)\\&-a_5\big(e^k_{3}, \varphi_{N}\big)=\big(R_1^{k},\varphi_{N}\big),
\end{split}
\label{e2}
\end{eqnarray}
\begin{eqnarray}
\begin{split}
&\big(\partial_te_{2}^{k},\varphi_{N}\big)+a_1\big(D^{\gamma}_{\tau}e_{2}^{k},\varphi_{N}\big)
+a_2\big(D^{\gamma}_{\tau}\nabla e_{1}^{k},\nabla \varphi_{N}\big)+\big(\nabla e_{1}^{k},\nabla \varphi_{N}\big)
+a_3\big(e_{2}^{k}, \varphi_{N}\big)+a_4\big(e^k_{1}, \varphi_{N}\big)\\&=\big(R_2^{k},\varphi_{N}\big),
\end{split}
\label{e3}
\end{eqnarray}
\begin{eqnarray}
\begin{split}
&\big(\partial_t e_{3}^{k},\varphi_{N}\big)+b_1\big(D^{1-\beta}_{\tau}e_{3}^{k},\varphi_{N}\big)
+b_2\big(D^{-\beta}_{\tau}\nabla e_{3}^{k},\nabla \varphi_{N}\big)-b_3\big(D^{-\beta}_{\tau} e_{3}^{k},\varphi_{N}\big)
-b_4(e_{3}^{k},\varphi_{N}\big)=\big(R^{k}_3,\varphi_{N}\big).
\end{split}
\label{e4}
\end{eqnarray}

Set $e_{1}^{k}=\eta_{1}^{k}+\xi_{1}^{k}$, where $\eta_{1}^{k}=u^{k}-\Pi_{N}^{1,0}u^{k}, \xi_{1}^{k}=\Pi_{N}^{1,0}u^{k}-u_{N}^{k}$; $e_{2}^{k}$ and $e_{3}^{k}$ can be treated similarly.
For $k\ge 2$, we take $\varphi_{N}=\xi_{1}^{k}$ in (\ref{e2}) and use Lemmas \ref{le3.1} and \ref{le4.1} to obtain
\begin{eqnarray}
\begin{split}
&\frac{1}{4\tau}(\Psi[\xi_{1}^k]-\Psi[\xi_{1}^{k-1}])+a_1\big(D^{\gamma}_{\tau}\xi_{1}^{k},\xi_{1}^k\big)
+a_2\big(D^{\gamma}_{\tau}\nabla \xi_{1}^{k},\nabla \xi_{1}^k\big)+\big(\nabla \xi_{1}^{k},\nabla \xi_{1}^k\big)
\\&+a_3\big(\xi_{1}^{k}, \xi_{1}^k\big)-a_4\big(\xi^k_{2}, \xi^k_1\big)-a_5\big(\xi^k_{3}, \xi^k_1\big)\\
\le&\big(R_1^{k},\xi^k_1)-\big(\partial_t \eta_{1}^{k},\xi_{1}^k\big)
-a_1\big(D^{\gamma}_{\tau}\eta_{1}^{k},\xi^k_1\big)
-a_3\big(\eta_{1}^{k}, \xi_{1}^k\big)+a_4\big(\eta^k_{2}, \xi^k_1\big)+a_5\big(\eta^k_{3}, \xi^k_1\big).
\end{split}
\label{e5}
\end{eqnarray}
Summing (\ref{e5}) over $k$ from 2 to $n$ and multiplying the inequality by 4$\tau$, we obtain
\begin{eqnarray}
\begin{split}
&\Psi[\xi_{1}^n]+4a_1\tau\sum^{n}_{k=2}\big(D^{\gamma}_{\tau}\xi_{1}^{k},\xi_{1}^k\big)
+4a_2\tau\sum^{n}_{k=2}\big(D^{\gamma}_{\tau}\nabla \xi_{1}^{k},\nabla \xi_{1}^k\big)
+4\tau\sum^{n}_{k=2}\big(\nabla \xi_{1}^{k},\nabla \xi_{1}^k\big)
\\&+4a_3\tau\sum^{n}_{k=2}\big(\xi_{1}^{k},\xi_{1}^k\big)-4a_4\tau\sum^{n}_{k=2}\big(\xi^k_{2}, \xi^k_1\big)
-4a_5\tau\sum^{n}_{k=2}\big(\xi^k_{3}, \xi^k_1\big)\\
\le&\Psi[\xi_{1}^{1}]+4\tau\sum^{n}_{k=2}\big(R_1^{k},\xi^k_1)-4\tau\sum^{n}_{k=2}\big(\partial_t \eta_{1}^{k},\xi_{1}^k\big)
-4a_1\tau\sum^{n}_{k=2}\big(D^{\gamma}_{\tau}\eta_{1}^{k},\xi^k_1\big)
\\&-4a_3\tau\sum^{n}_{k=2}\big(\eta_{1}^{k}, \xi_{1}^k\big)+4a_4\tau\sum^{n}_{k=2}\big(\eta^k_{2}, \xi^k_1\big)+4a_5\tau\sum^{n}_{k=2}\big(\eta^k_{3}, \xi^k_1\big).
\end{split}
\label{e6}
\end{eqnarray}

Next, we give some estimations of the terms on the right-hand side of inequality (\ref{e6}). Using the Cauchy--Schwarz inequality, Young's inequality, and Lemma \ref{le3.1}, we find that
\begin{eqnarray}
\begin{split}
4\tau\sum_{k=2}^{n}(R_1^{k},\xi_{1}^{k})&\le 2\tau\sum_{k=2}^{n}\lVert R_1^{k}\lVert^2+2\tau\sum_{k=2}^{n}
\lVert\xi_{1}^{k}\lVert^2\le C_{e,1} \tau^4+2\tau\sum_{k=2}^{n}\lVert\xi_{1}^{k}\lVert^2,
\end{split}
\label{e7}
\end{eqnarray}
\begin{eqnarray}
\begin{split}
-4\tau\sum_{k=2}^{n}(\partial_t \eta_{1}^{k},\xi_{1}^{k})&=-4\tau\sum_{k=2}^{n}
\bigg(\frac{3}{2\tau}\int^{t_k}_{t_{k-1}}\frac{\partial\eta_{1}}{\partial t}dt-\frac{1}{2\tau}\int^{t_{k-1}}_{t_{k-2}}\frac{\partial\eta_{1}}{\partial t}dt, \xi_{1}^{k}\bigg)\\
&\le 4\int^{T}_{0}\left\|\frac{\partial\eta_{1}}{\partial t}\right\|^2dt+2\tau\sum_{k=2}^{n}\lVert\xi_{1}^{k}\lVert^2
\\&\le C_{e,2}N^{-2r}\left\|\frac{\partial u}{\partial t}\right\|^2_{L^2(H^r)}+2\tau\sum_{k=2}^{n}\lVert\xi_{1}^{k}\lVert^2,
\end{split}
\label{e8}
\end{eqnarray}
\begin{eqnarray}
\begin{split}
-4a_1\tau\sum_{k=2}^{n}(D^{\gamma}_{\tau}\eta_1^{k},\xi_{1}^{k})&\le 2a_1\tau\sum_{k=2}^{n}\lVert D^{\gamma}_{\tau}\eta_1^{k}\lVert^2+2a_1\tau\sum_{k=2}^{n}\lVert\xi_{1}^{k}\lVert^2\\
&\le 2a_1\tau\sum_{k=2}^{n}\lVert{_{~~0}^{RL}D^{\gamma}_t}\eta_1^{k}\lVert^2+C_{e,3}\tau^4+2a_1\tau\sum_{k=2}^{n}\lVert\xi_{1}^{k}\lVert^2\\&\le C_{e,4}N^{-2r}\lVert_{~~0}^{RL}D^{\gamma}_tu\lVert^2_{L^2(H^r)}+C_{e,3}\tau^4+2a_1\tau\sum_{k=2}^{n}\lVert\xi_{1}^{k}\lVert^2,
\end{split}
\label{e9}
\end{eqnarray}

\begin{eqnarray}
\begin{split}
-4a_3\tau\sum_{k=2}^{n}(\eta_1^{k},\xi_{1}^{k})&\le 2a_3\tau\sum_{k=2}^{n}\lVert \eta_1^{k}\lVert^2+2a_3\tau\sum_{k=2}^{n}\lVert\xi_{1}^{k}\lVert^2\le C_{e,5}N^{-2r}\lVert u\lVert^2_{L^2(H^r)}+2a_3\tau\sum_{k=2}^{n}\lVert\xi_{1}^{k}\lVert^2,
\end{split}
\label{e10}
\end{eqnarray}

\begin{eqnarray}
\begin{split}
4a_4\tau\sum^{n}_{k=2}\big(\eta^k_{2}, \xi^k_1\big)&\le 2a_4\tau\sum_{k=2}^{n}\lVert \eta_2^{k}\lVert^2+2a_4\tau\sum_{k=2}^{n}\lVert\xi_{1}^{k}\lVert^2\le C_{e,6}N^{-2r}\lVert v\lVert^2_{L^2(H^r)}+2a_4\tau\sum_{k=2}^{n}\lVert\xi_{1}^{k}\lVert^2,
\end{split}
\label{e11}
\end{eqnarray}
and
\begin{eqnarray}
\begin{split}
4a_5\tau\sum^{n}_{k=2}\big(\eta^k_{3}, \xi^k_1\big)&\le 2a_5\tau\sum_{k=2}^{n}\lVert \eta_3^{k}\lVert^2+2a_5\tau\sum_{k=2}^{n}\lVert\xi_{1}^{k}\lVert^2\le C_{e,7}N^{-2r}\lVert \theta\lVert^2_{L^2(H^r)}+2a_5\tau\sum_{k=2}^{n}\lVert\xi_{1}^{k}\lVert^2.
\end{split}
\label{e12}
\end{eqnarray}

For $k=1$, we take $\varphi_{N}=\xi_{1}^{1}$ and further obtain that
\begin{eqnarray}
\begin{split}
&\frac{\lVert \xi_{1}^1\lVert^2-\lVert \xi_{1}^0\lVert^2}{2\tau}+\frac{1}{2\tau}\lVert \xi_{1}^1- \xi_{1}^0\lVert^2+a_1\big(D^{\gamma}_{\tau}\xi_{1}^{1},\xi_{1}^1\big)
+a_2\big(D^{\gamma}_{\tau}\nabla \xi_{1}^{1},\nabla \xi_{1}^1\big)\\&+\big(\nabla \xi_{1}^{1},\nabla \xi_{1}^1\big)
+a_3\big(\xi_{1}^{1}, \xi_{1}^1\big)-a_4\big(\xi^1_{2}, \xi^1_1\big)-a_5\big(\xi^1_{3}, \xi^1_1\big)\\
\le&\big(R_1^{1},\xi^1_1)-\big(\partial_t \eta_{1}^{1},\xi_{1}^1\big)
-a_1\big(D^{\gamma}_{\tau}\eta_{1}^{1},\xi^1_1\big)
-a_3\big(\eta_{1}^{1}, \xi_{1}^1\big)+a_4\big(\eta^1_{2}, \xi^1_1\big)+a_5\big(\eta^1_{3}, \xi^1_1\big).
\end{split}
\label{e13}
\end{eqnarray}
Simplifying (\ref{e13}) and applying the Cauchy--Schwarz inequality and Young's inequality, we have that
\begin{eqnarray}
\begin{split}
&2\lVert \xi_{1}^1\lVert^2+ 2a_1\tau\big(D^{\gamma}_{\tau}\xi_{1}^{1},\xi_{1}^1\big)
+2a_2\tau\big(D^{\gamma}_{\tau}\nabla \xi_{1}^{1},\nabla \xi_{1}^1\big)+2\tau\big(\nabla \xi_{1}^{1},\nabla \xi_{1}^1\big)
\\&+2a_3\tau\big(\xi_{1}^{1}, \xi_{1}^1\big)-2a_4\tau\big(\xi^1_{2}, \xi^1_1\big)-2a_5\tau\big(\xi^1_{3}, \xi^1_1\big)\\
\le&2\tau\big(R_1^{1},\xi^1_1)-2\tau\big(\partial_t \eta_{1}^{1},\xi_{1}^1\big)
-2a_1\tau\big(D^{\gamma}_{\tau}\eta_{1}^{1},\xi^1_1\big)
-2a_3\tau\big(\eta_{1}^{1}, \xi_{1}^1\big)+2a_4\tau\big(\eta^1_{2}, \xi^1_1\big)+{2a_5\tau\big(\eta^1_{3}, \xi^1_1\big)}\\
\le&C_{e,8} \tau^4+C_{e,8}N^{-2r}\bigg(\lVert_{~~0}^{RL}D^{\gamma}_tu\lVert^2_{L^2(H^r)}
+\left\|\frac{\partial u}{\partial t}\right\|^2_{L^2(H^r)}+\left\|u\right\|^2_{L^2(H^r)}+\left\|v\right\|^2_{L^2(H^r)}\bigg)+C_{e,8}\tau\lVert\xi_{1}^{1}\lVert^2.
\end{split}
\label{e14}
\end{eqnarray}
Using (\ref{e7})--(\ref{e14}), (\ref{e6}) can be written as
\begin{eqnarray}
\begin{split}
&\Psi[\xi_{1}^n]+4a_1\tau\sum^{n}_{k=1}\big(D^{\gamma}_{\tau}\xi_{1}^{k},\xi_{1}^k\big)
+4a_2\tau\sum^{n}_{k=1}\big(D^{\gamma}_{\tau}\nabla \xi_{1}^{k},\nabla \xi_{1}^k\big)
+4\tau\sum^{n}_{k=1}\big(\nabla \xi_{1}^{k},\nabla \xi_{1}^k\big)
\\&+4a_3\tau\sum^{n}_{k=1}\big(\xi_{1}^{k}, \xi_{1}^k\big)-4a_4\tau\sum^{n}_{k=2}\big(\xi^k_{2}, \xi^k_1\big)-5a_4\tau\big(\xi^1_{2}, \xi^1_1\big)
{-4a_5\tau\sum^{n}_{k=2}\big(\xi^k_{3}, \xi^k_1\big)-5a_5\tau\big(\xi^1_{3}, \xi^1_1\big)}\\
\le& C_{e,9} \tau^4+C_{e,9}N^{-2r}\bigg(\lVert_{~~0}^{RL}D^{\gamma}_tu\lVert^2_{L^2(H^r)}+\left\|\frac{\partial u}{\partial t}\right\|^2_{L^2(H^r)}+\left\|u\right\|^2_{L^2(H^r)}+\left\|v\right\|^2_{L^2(H^r)}+\left\|\theta\right\|^2_{L^2(H^r)}\bigg)\\&+C_{e,9}\tau\sum^{n}_{k=1}\lVert\xi_{1}^{k}\lVert^2.
\end{split}
\label{e15}
\end{eqnarray}
{
Using Lemma \ref{le4.1}, Corollary \ref{co4.4}, we obtain
\begin{eqnarray}
\begin{split}
\lVert\xi_{1}^n\lVert^2\le&4a_4\tau\sum^{n}_{k=2}\big(\xi^k_{2}, \xi^k_1\big)+5a_4\tau\big(\xi^1_{2}, \xi^1_1\big)+{5a_5\tau}\sum^n\limits_{k=1}\lVert \xi^k_{3}\lVert^2+C_{e,10}\tau\sum^{n}_{k=1}\lVert\xi_{1}^{k}\lVert^2+ C_{e,9} \tau^4\\&+C_{e,9}N^{-2r}\bigg(\lVert_{~~0}^{RL}D^{\gamma}_tu\lVert^2_{L^2(H^r)}+\left\|\frac{\partial u}{\partial t}\right\|^2_{L^2(H^r)}+\left\|u\right\|^2_{L^2(H^r)}+\left\|v\right\|^2_{L^2(H^r)}+\left\|\theta\right\|^2_{L^2(H^r)}\bigg).
\end{split}
\label{e17}
\end{eqnarray}}

Similarly, taking $\varphi_{N}=\xi_{2}^{k}$ in (\ref{e2}) yields
\begin{eqnarray}
\begin{split}
\lVert\xi_{2}^n\lVert^2\le&-4a_4\tau\sum^{n}_{k=2}\big(\xi^k_{1}, \xi^k_2\big)-5a_4\tau\big(\xi^1_{1}, \xi^1_2\big)+C_{e,11} \tau^4+C_{e,11}\tau\sum^{n}_{k=1}\lVert\xi_{2}^{k}\lVert^2\\&+C_{e,11} N^{-2r}\bigg(\lVert_{~~0}^{RL}D^{\gamma}_tv\lVert^2_{L^2(H^r)}
+\left\|\frac{\partial v}{\partial t}\right\|^2_{L^2(H^r)}+\left\|u\right\|^2_{L^2(H^r)}+\left\|v\right\|^2_{L^2(H^r)}\bigg).
\end{split}
\label{e18}
\end{eqnarray}
and taking $\varphi_{N}=\xi_{3}^{k}$ in (\ref{e3}) gives
\begin{eqnarray}
\begin{split}
\|\xi_{3}^n\|^2\le& C_{e,12} \tau^4+C_{e,12}\tau\sum^{n}_{k=1}\lVert\xi_{3}^{k}\lVert^2\\&+C_{e,12} N^{-2r}\bigg(\lVert_{~~0}^{RL}D^{1-\beta}_t\theta\lVert^2_{L^2(H^r)}+\lVert_{~~0}^{RL}D^{-\beta}_t\theta\lVert_{L^2(H^r)}^2
+\left\|\frac{\partial \theta}{\partial t}\right\|_{L^2(H^r)}^2+\left\|\theta\right\|_{L^2(H^r)}^2\bigg).
\end{split}
\label{e21}
\end{eqnarray}

{Adding (\ref{e17})--(\ref{e18}) and (\ref{e21}), we apply Gronwall inequality for sufficiently small $\tau$ to have the following result:}
\begin{eqnarray}
\begin{split}
\lVert\xi_{1}^n\lVert^2+\lVert\xi_{2}^n\lVert^2+\lVert\xi_{3}^n\lVert^2&\le  C_{e,13}( \tau^4+N^{-2r}).
\end{split}
\label{e22}
\end{eqnarray}

Finally, we use the triangle inequality $\lVert e_{i}^n\lVert^2\le\lVert\xi_{i}^n\lVert^2+\lVert\eta_{i}^n\lVert^2$, for $i=1,2,3$, and Lemma \ref{le3.1} to obtain the desired results.
\end{proof}
{
\begin{remark}
In Theorem 2, we assume that the solutions of the model (\ref{1.1})-(\ref{1.3}) are smooth enough. Generally speaking, the analytical solutions to fractional models have a weak  singularity near initial time\cite{re1,re2,re3,re4,re5}. For example, Stynes et al.\cite{re1} have found that the solution $u$ of the time fractional diffusion equation has the properties that $\left|\frac{\partial^{k} u}{\partial x^{k}}(x, t)\right| \leq C$ for $k=0,1,2,3,4$ and $\left|\frac{\partial^{\ell} u}{\partial t^{\ell}}(x, t)\right| \leq C\left(1+t^{\alpha-\ell}\right)$ for $\ell=0,1,2$ with proper regularity and compatibility assumptions. It shows $u$ is smooth away from $t = 0$, but it has a singular behavior at $t = 0$. For the case of non-smooth solutions, there are two common strategies, i.e., non-uniform mesh method \cite{re6,re7} and correction method \cite{394,39,re8}. Non-uniform mesh method is to employ smaller time-steps near $t = 0$ to compensate the singular behavior at $t = 0$ of solutions, and it is usually used for L1-type approximation. Considering that our time discretization is based on convolution quadrature, we can use correction method to deal with non-smooth solutions of the model (\ref{1.1})-(\ref{1.3}), in which the appropriate correction terms are added to the FBDF scheme to make the new scheme accurate for the low regular terms of the solutions and maintain high accuracy for the high regular terms. More details on the convolution quadrature with correction method can be found in Refs.\cite{394,39,re8}.
\end{remark}}
\section{Fast method and convergence analysis }
\label{sec:5}
To reduce both the storage requirements and computation time of the numerical calculation, we present a fast method which expresses the quadrature weight $\omega_{k}^{(\gamma)}$  as an integral on the half line, then give a globally uniform approximation of the trapezoidal rule for the integral. We then use this approximation to solve the time-fractional operators.

The convolution weights $\omega_{k}^{(\gamma)}$ can be written as
\begin{eqnarray}
\begin{split}
\omega_{k}^{(\gamma)}=\tau^{1+\gamma} \frac{{\sin (\gamma \pi)}}{\pi} \int_{0}^{\infty} \sigma^{\gamma}(1+\sigma \tau)^{-1-k} \hat{F}_{\omega}(-\sigma) d\sigma,
\end{split}
\label{r1}
\end{eqnarray}
where $\hat{F}_{\omega}(\sigma)=(\tau \sigma)^{-\gamma} \omega^{(\gamma)}(1-\tau \sigma)$ and $\omega^{(\gamma)}$ is the generating function of the FBDF method defined in (\ref{FBDF}).

Let $\sigma=\exp (y)$ \cite{37}, so that (\ref{r1}) can be written as
\begin{eqnarray}
\begin{split}
\omega_{k}^{(\gamma)}=\tau^{1+\gamma} \int_{-\infty}^{\infty} \epsilon_{k}(y) d y,
\end{split}
\label{int1}
\end{eqnarray}
where
$$
\epsilon_{k}(y)=\left(1+e^{y} \tau\right)^{-1-k} \epsilon(y), \quad
\epsilon(y)=\frac{\sin (\gamma \pi)}{\pi} e^{(1+\gamma) y} \hat{F}_{\omega}\left(-e^{y}\right).
$$
Note that $\epsilon_{k}(y)$ decays exponentially as $|y| \rightarrow \infty$, which inspires us to approximate the integral in (\ref{int1}) by the exponentially convergent trapezoidal rule \cite{37}. Then, we obtain
\begin{eqnarray}
\omega_{k}^{(\gamma)}\approx\widetilde{\omega}_{k}^{(\gamma)}=\tau^{1+\gamma} \sum_{i=0}^{Q-1} h_{i}\left(1+\sigma_{i} \tau\right)^{-1-k},
\label{ff1}
\end{eqnarray}
where $\sigma_{i}=e^{y_{i}}, y_{i}=i \Delta y, \Delta y>0$, $h_{i}=\epsilon(y_{i})\Delta y $, and $Q>1$ is the number of quadrature points used in the trapezoidal rule. The determination of $h_{i}$ and $\sigma_{i}$ is described in \cite{Guo}.

Define the operator $_FD_{\tau}^{\gamma} u^{k}$ as 
\begin{eqnarray}
_FD_{\tau}^{\gamma} u^{k}=\frac{1}{\tau^\gamma}\sum_{j=k-k_0}^k\omega_{k-j}^{(\gamma)}u^j+\frac{1}{\tau^\gamma}
\sum_{j=0}^{k-k_0-1}\widetilde{\omega}_{k-j}^{(\gamma)}u^j.
\label{def1}
\end{eqnarray}
Then, the local part $\frac{1}{\tau^{\gamma}} \sum\limits_{j=k-k_{0}}^{k} \omega_{k-j}^{(\gamma)} u^{j}$ can be computed directly. For the historical part $\frac{1}{\tau^{\gamma}} \sum\limits_{j=0}^{k-k_{0}-1} \widetilde{\omega}_{k-j}^{(\gamma)} u^{j}$,
we have that
\begin{eqnarray}
\begin{aligned}
\frac{1}{\tau^{\gamma}} \sum_{j=0}^{k-k_{0}-1} \widetilde{\omega}_{k-j}^{(\gamma)} u^{j} &=\tau \sum_{j=0}^{k-k_{0}-1} u^{j} \sum_{i=0}^{Q-1} h_{i}\left(1+\sigma_{i} \tau\right)^{-1-(k-j)}\\&= \sum_{i=0}^{Q-1} \left(1+\sigma_{i} \tau\right)^{-k_{0}-1}h_{i} q_{k-k_{0}}^{(i)},
\end{aligned}
\end{eqnarray}
where $q_{k-k_{0}}^{(i)}=\tau \sum\limits_{j=0}^{k-k_{0}-1}\left(1+\sigma_{i} \tau\right)^{-\left(k-k_{0}-j\right)} u^{j}$ satisfies
\begin{eqnarray}
\begin{aligned}
q_{k}^{(i)}=\frac{1}{1+\sigma_{i} \tau}\left(q_{k-1}^{(i)}+\tau u^{k-j}\right),\quad q_{0}^{(i)}=0.
\end{aligned}
\label{wh1}
\end{eqnarray}

{In the fast computation, the memory requirement and computational cost of (\ref{def1}) are $O(Q)$ and $O\left(Q{\bar{K}}\right)\\(Q\ll\bar{K})$, respectively, which are much less than the direct method with $O(\bar{K})$ memory and $O(\bar{K}^2)$ operation. Therefore, the fast method can effectively save the memory requirement and the computational cost, especially for long time large-scale computation.} The operators $_FD_{\tau}^{\gamma} \frac{\partial^{2} u^{k}}{\partial z^{2}}$, $_FD_{\tau}^{\gamma} v^{k}$, $_FD_{\tau}^{\gamma} \frac{\partial^{2} v^{k}}{\partial z^{2}}$, $_FD_{\tau}^{1-\beta} \theta^{k}$, $_FD_{\tau}^{-\beta} \theta^{k}$, and $_FD_{\tau}^{-\beta} \frac{\partial^{2} v^{k}}{\partial \theta^{2}}$ can be defined in the same way.

We replace $D_{\tau}^{\gamma}, D_{\tau}^{1-\beta}, D_{\tau}^{-\beta}$ in (\ref{4.5})--(\ref{4.7}) with ${_FD_{\tau}^{\gamma}}, { _FD_{\tau}^{1-\beta}}, {_FD_{\tau}^{-\beta}}$ to obtain the fast time-stepping spectral method for (\ref{1.1})--(\ref{1.3}). This method determines $_Fu_N^k, {_Fv}_N^k, {_F\theta}_N^k\in V_{N}^0$ for any $\varphi_{N}\in V_{N}^0$, $k\ge k_0+1$, such that
\begin{eqnarray}
\begin{split}
&\big(\partial_t {_Fu}_{N}^{k},\varphi_{N}\big)+a_1\big(_FD^{\gamma}_{\tau}{_Fu}_{N}^{k},\varphi_{N}\big)
+a_2\big(_FD^{\gamma}_{\tau}\nabla {_Fu}_{N}^{k},\nabla \varphi_{N}\big)+\big(\nabla {_Fu}_{N}^{k},\nabla \varphi_{N}\big)
+a_3\big({_Fu}_{N}^{k}, \varphi_{N}\big)\\&-a_4\big({_Fv}^k_{N}, \varphi_{N}\big)-a_5\big({_F\theta}^k_{N}, \varphi_{N}\big)=\big(f^{k},\varphi_{N}\big),
\end{split}
\label{6.1}
\end{eqnarray}
\begin{eqnarray}
\begin{split}
&\big(\partial_t {_Fv}_{N}^{k},\varphi_{N}\big)+a_1\big(_FD^{\gamma}_{\tau}{_Fv}_{N}^{k},\varphi_{N}\big)
+a_2\big(_FD^{\gamma}_{\tau}\nabla {_Fv}_{N}^{k},\nabla \varphi_{N}\big)+\big(\nabla {_Fv}_{N}^{k},\nabla \varphi_{N}\big)
+a_3\big({_Fv}_{N}^{k}, \varphi_{N}\big)\\&+a_4\big({_Fu}^k_{N}, \varphi_{N}\big)=\big(g^{k},\varphi_{N}\big),
\end{split}
\label{6.2}
\end{eqnarray}
\begin{eqnarray}
\begin{split}
&\big(\partial_t {_F\theta_{N}^{k}},\varphi_{N}\big)+b_1\big({_FD}^{1-\beta}_{\tau}{_F\theta}_{N}^{k},\varphi_{N}\big)
+b_2\big({_FD}^{-\beta}_{\tau}\nabla _F\theta_{N}^{k},\nabla \varphi_{N}\big)-b_3\big({_FD}^{-\beta}_{\tau} {_F\theta_{N}^{k}},\varphi_{N}\big)
\\&-b_4({_F\theta_{N}^{k}},\varphi_{N}\big)=\big({\widetilde{p}^{k}},\varphi_{N}\big),
\end{split}
\label{6.3}
\end{eqnarray}
and
\begin{eqnarray}
_Fu_N^k=u^k_N, \quad _Fv_N^k=v^k_N, \quad _F\theta_N^k=\theta^k_N, \quad 0\le k\le k_0,
\label{6.4}
\end{eqnarray}
where $u_{N}^{k}$, $v_{N}^{k}$, $\theta_{N}^{k}$ are the solutions given by the direct method. According to \cite{Guo,38}, the error produced by the quadrature in (\ref{ff1}) becomes arbitrarily small when $Q$ is sufficiently large. Hence, we can express $\widetilde{\omega}_{k}^{(\gamma)}$ as
\begin{eqnarray}
\widetilde{\omega}_{k}^{(\gamma)}=\left(1+\varepsilon^{(\gamma)}_{k}\right) \omega_{k}^{(\gamma)},\quad
\end{eqnarray}
where $\varepsilon^{(\gamma)}_{k}$ is the error (which can be made arbitrarily small) and $\varepsilon_{k}^{(\gamma)}=0$ for $0 \leq k<k_{0}$. Both $\widetilde{\omega}_{k}^{(1-\beta)}$ and $\widetilde{\omega}_{k}^{(-\beta)}$ can be treated similarly.

Next, we analyze the convergence of the fast time-stepping spectral scheme of (\ref{6.1})--(\ref{6.4}). The following
lemma shows that $\nabla u_{N}^{n}$, $\nabla v_{N}^{n}$, and $\nabla \theta_{N}^{n}$ are bounded.
\begin{lemma}
Solutions $u_{N}^n, v_{N}^n, \theta_{N}^n$ of the fully discrete spectral scheme in (\ref{4.5})--(\ref{4.7}) satisfy
\begin{eqnarray}
\begin{split}
&\tau\sum_{k=1}^n\|\nabla u_{N}^{k}\|^2+\tau\sum_{k=1}^n\|\nabla v_{N}^{k}\|^2+\tau\sum_{k=1}^n\|\nabla \theta_{N}^{k}\|^2
\\\le& C\tau\bigg(\sum^n\limits_{k=1}\lVert {f}^k\lVert^2+\sum^n\limits_{k=1}\lVert {g}^k\lVert^2+\sum^n\limits_{k=1}\lVert\widetilde{p}^k\lVert^2+\sum^n\limits_{k=1}\lVert {\nabla \widetilde{p}}^k\lVert^2\bigg).
\end{split}
\end{eqnarray}
\label{le5.1}
\end{lemma}

\begin{proof}
We take $\varphi_{N}=u^{k}_{N}$ in (\ref{4.5}) and sum it over $k$ from 2 to $n$. By Lemma \ref{le4.1}, the Cauchy--Schwarz inequality, and Young's inequality, we obtain
\begin{eqnarray}
\begin{split}
&\Psi[u_{N}^n]+4a_1\tau^{1-\gamma}\sum^n\limits_{k=2}\sum^{k}\limits_{j=0}\omega_{k-j}^{(\gamma)}\big(u^{j}_{N},u^{k}_{N}\big)
+4a_2\tau^{1-\gamma}\sum^n\limits_{k=2}\sum^{k}\limits_{j=0}\omega_{k-j}^{(\gamma)}\big(\nabla u^{j}_{N},\nabla u^{k}_{N}\big)\\&+4\tau\sum_{k=2}^n\|\nabla u_{N}^{k}\|^2
+4a_3\tau\sum_{k=2}^n\|u_{N}^{k}\|^2-4a_4\tau\sum_{k=2}^n\big(v^k_{N}, u^{k}_{N}\big)-4a_5\tau\sum_{k=2}^n\big(\theta^k_{N}, u^{k}_{N}\big)\\
\le& \Psi[u_{N}^1]+\frac{2\tau}{a_3}\sum^n\limits_{k=2}\lVert{f}^k\lVert^2+2a_3\tau\sum^n\limits_{k=2}\lVert u_{N}^k\lVert^2.
\end{split}
\label{p2}
\end{eqnarray}
Considering the case of $k=1$, we take $\varphi_{N}=u_{N}^{1}$ in (\ref{4.5}) to obtain
{
\begin{eqnarray}
\begin{split}
\Psi[u_{N}^1]\le&-4a_1\tau^{1-\gamma}\sum_{j=0}^{1}\omega_{1-j}^{(\gamma)}\big(u_{N}^{j},u_{N}^{1}\big)-4a_2\tau^{1-\gamma}
\sum_{j=0}^{1}\omega_{1-j}^{(\gamma)}\big(\nabla u_{N}^{j},\nabla u_{N}^{1}\big)-4\tau\big(\nabla u_{N}^{1},\nabla u_{N}^{1}\big)\\&-4a_3\tau\big(u_{N}^{1}, u_{N}^{1}\big)+5a_4\tau\big(v^1_{N}, u_{N}^{1}\big)
+5a_5\tau\big(\theta^1_{N}, u_{N}^{1}\big)+C_{b,1}\tau\lVert {f}^1\lVert^2+2a_3\tau\lVert u_{N}^1\lVert^2.
\end{split}
\label{p4}
\end{eqnarray}
According to Lemma \ref{le4.1}, we have that
\begin{eqnarray}
\begin{split}
&4\tau\sum_{k=1}^n\|\nabla u_{N}^{k}\|^2
+4a_3\tau\sum_{k=1}^n\|u_{N}^{k}\|^2\\\le & -4a_1\tau^{1-\gamma}\sum^n\limits_{k=1}\sum^{k}\limits_{j=0}\omega_{k-j}^{(\gamma)}\big(u^{j}_{N},u^{k}_{N}\big)
-4a_2\tau^{1-\gamma}\sum^n\limits_{k=1}\sum^{k}\limits_{j=0}\omega_{k-j}^{(\gamma)}\big(\nabla u^{j}_{N},\nabla u^{k}_{N}\big)\\&
+4a_4\tau\sum_{k=2}^n\big(v^k_{N}, u^{k}_{N}\big)+5a_4\tau\big(v^1_{N}, u^{1}_{N}\big)
+\frac{5a_5\tau}{2a_3}\sum_{k=1}^n\lVert\theta^k_{N}\lVert^2
+C_{b,2}\tau\sum^n\limits_{k=1}\lVert{f}^k\lVert^2+4a_3\tau\sum^n\limits_{k=1}\lVert u_{N}^k\lVert^2.
\end{split}
\label{p5}
\end{eqnarray}}
{
From (\ref{s22}), applying Gronwall inequality for an enough small $\tau$, we can get
\begin{eqnarray}
\begin{split}
\lVert \theta_{N}^n\lVert^2 \le& C_{b,3}\tau\sum^n\limits_{k=1}\lVert\widetilde{p}^k\lVert^2.
\end{split}
\label{p51}
\end{eqnarray}
By Corollary \ref{co4.4} and  (\ref{p51}),  (\ref{p5}) follows that}
\begin{eqnarray}
\begin{split}
4\tau\sum_{k=1}^n\|\nabla u_{N}^{k}\|^2&\le 4a_4\tau\sum_{k=2}^n\big(v^k_{N}, u^{k}_{N}\big)+5a_4\tau\big(v^1_{N}, u^{1}_{N}\big)+C_{b,2}\tau\sum^n\limits_{k=1}\lVert{f}^k\lVert^2
+C_{b,4}\tau\sum^n\limits_{k=1}\lVert\widetilde{p}^k\lVert^2.
\end{split}
\label{p6}
\end{eqnarray}

Similarly, taking $\varphi_{N}=v^{k}_{N}$ in (\ref{4.6}), we obtain
\begin{eqnarray}
\begin{split}
4\tau\sum_{k=1}^n\|\nabla v_{N}^{k}\|^2
\le -4a_4\tau\sum_{k=1}^n\big(u^k_{N}, v^{k}_{N}\big)-5a_4\tau\big(u^1_{N}, v^{1}_{N}\big)+C_{b,5}\tau\sum^n\limits_{k=1}\lVert {g}^k\lVert^2.
\end{split}
\label{p7}
\end{eqnarray}

For the temperature equation (\ref{4.7}), we take $\varphi_{N}=\Delta\theta^{k}_{N}$ and use Lemma \ref{le4.1} to obtain
\begin{eqnarray}
\begin{split}
&\frac{1}{4\tau}(\Psi[\nabla\theta_{N}^k]-\Psi[\nabla\theta_{N}^{k-1}])+b_1\big(D^{1-\beta}_{\tau}\nabla\theta_{N}^{k},\nabla\theta^{k}_{N}\big)
+b_2\big(D^{-\beta}_{\tau}\Delta \theta_{N}^{k},\Delta \theta^{k}_{N}\big)\\&-b_3\big(D^{-\beta}_{\tau} \nabla\theta_{N}^{k},\nabla\theta^{k}_{N}\big)
-b_4(\nabla\theta_{N}^{k},\nabla\theta^{k}_{N}\big)\le\big(\nabla \widetilde{p}^{k},\nabla\theta^{k}_{N}\big).
\end{split}
\label{p8}
\end{eqnarray}

After summing (\ref{p8}) over $k$ from 2 to $n$ and multiplying both sides of this inequality by $4\tau$, the application of the Cauchy--Schwarz inequality and Young's inequality yields
\begin{eqnarray}
\begin{split}
&\Psi[\nabla \theta_{N}^n]+4b_1\tau\sum_{k=2}^n\big(D^{1-\beta}_{\tau}\nabla\theta_{N}^{k},\nabla\theta^{k}_{N}\big)
+4b_2\tau\sum_{k=2}^n\big(D^{-\beta}_{\tau}\Delta\theta_{N}^{k},\Delta \theta^{k}_{N}\big)\\&-4b_3\tau\sum_{k=2}^n\big(D^{-\beta}_{\tau} \nabla\theta_{N}^{k},\nabla\theta^{k}_{N}\big)
-4b_4\tau\sum_{k=2}^n(\nabla\theta_{N}^{k},\nabla\theta^{k}_{N}\big)
\\\le& \Psi[\nabla\theta_{N}^1]+2\tau\sum^n\limits_{k=2}\lVert{\nabla \widetilde{p}}^k\lVert^2+2\tau\sum^n\limits_{k=2}\lVert \nabla \theta_{N}^k\lVert^2.
\end{split}
\label{p9}
\end{eqnarray}
Similar to the previous proof, we can derive
\begin{eqnarray}
\begin{split}
\lVert \nabla\theta_{N}^n\lVert^2 \le& 5C_{b,3}\bar{T}^{\beta}b_3\tau \sum_{k=1}^n \lVert\nabla\theta^{k}_{N}\lVert^2
+5b_4\tau\sum_{k=1}^n\lVert\nabla\theta_{N}^{k}\lVert^2+C_{b,6}\tau\sum^n\limits_{k=1}\lVert{\nabla \widetilde{p}}^k\lVert^2+C_{b,6}\tau\sum^n\limits_{k=1}\lVert \nabla\theta_{N}^k\lVert^2,
\end{split}
\label{p10}
\end{eqnarray}
whereupon the application of Gronwall inequality for sufficiently small $\tau$ gives
\begin{eqnarray}
\begin{split}
\lVert \nabla \theta_{N}^n\lVert^2\le C_{b,7}\tau\sum^n\limits_{k=1}\lVert {\nabla \widetilde{p}}^k\lVert^2.\\
\end{split}
\label{p13}
\end{eqnarray}

Then, (\ref{p6})--(\ref{p7}) and (\ref{p13}) yield
\begin{eqnarray}
\begin{split}
&\tau\sum_{k=1}^n\|\nabla u_{N}^{k}\|^2+\tau\sum_{k=1}^n\|\nabla v_{N}^{k}\|^2+\tau\sum_{k=1}^n\|\nabla \theta_{N}^{k}\|^2
\\ \le &C\tau\bigg(\sum^n\limits_{k=1}\lVert {f}^k\lVert^2+\sum^n\limits_{k=1}\lVert {g}^k\lVert^2+\sum^n\limits_{k=1}\lVert\widetilde{p}^k\lVert^2+\sum^n\limits_{k=1}\lVert {\nabla \widetilde{p}}^k\lVert^2\bigg).\\
\end{split}
\end{eqnarray}
\end{proof}

The following theorem specifies the error bound of ${_Fu^n_N}, {_Fv^n_N}, {_F\theta^n_N}$.
\begin{theorem}
Let $u^n_N, v^n_N, \theta^n_N$ and ${_Fu^n_N}, {_Fv^n_N}, {_F\theta^n_N}$ be the solutions of (\ref{4.5})--(\ref{4.6}) and (\ref{6.1})--(\ref{6.3}), respectively. If the conditions in Theorem \ref{th4.6} hold, and $|\varepsilon^{(\gamma)}_k|\le\varepsilon$, $|\varepsilon^{(1-\beta)}_k|\le \varepsilon$, $|\varepsilon_k^{(-\beta)}|\le \varepsilon$, with $\varepsilon \le \frac{N^{-4}}{2C_{f,1}a_2}$. Then,
\begin{eqnarray}
\|u^n_N-{_Fu^n_N}\|^2+\|v^n_N-{_Fv^n_N}\|^2+\|\theta^n_N-{_F\theta^n_N}\|^2\le C \varepsilon.
\end{eqnarray}
\label{th5.2}
\end{theorem}
\begin{proof}
Let $\zeta_1^k={_Fu_N^k}-u^k_N, \zeta_2^k={_Fv_N^k}-v^k_N, \zeta_3^k={_F\theta_N^k}-\theta^k_N $ for $ k \ge k_0+1$. Then, we have that
\begin{eqnarray}
\begin{split}
&\big(\partial_t \zeta_1^{k}, \varphi_{N}\big)+a_1\big(D^{\gamma}_{\tau}\zeta_1^{k},\varphi_{N}\big)
+a_2\big(D^{\gamma}_{\tau}\nabla \zeta_1^{k},\nabla \varphi_{N}\big)\\&+\big(\nabla \zeta_1^{k},\nabla \varphi_{N}\big)
+a_3\big(\zeta_1^{k}, \varphi_{N}\big)-a_4\big(\zeta_2^k, \varphi_{N}\big)-a_5\big(\zeta_3^k, \varphi_{N}\big)\\=&-a_1\tau^{-\gamma}\sum^k_{j=k_0+1}\varepsilon^{(\gamma)}_{k-j}\omega_{k-j}^{(\gamma)}(\zeta_1^j, \varphi_{N})
-a_2\tau^{-\gamma}\sum^k_{j=k_0+1}\varepsilon^{(\gamma)}_{k-j}\omega_{k-j}^{(\gamma)}(\nabla\zeta_1^j, \nabla\varphi_{N})-H_1^k(\varphi_{N}),
\end{split}
\label{f4}
\end{eqnarray}
\begin{eqnarray}
\begin{split}
&\big(\partial_t \zeta_2^{k},\varphi_{N}\big)+a_1\big(D^{\gamma}_{\tau}\zeta_2^{k},\varphi_{N}\big)
+a_2\big(D^{\gamma}_{\tau}\nabla \zeta_2^{k},\nabla \varphi_{N}\big)+\big(\nabla \zeta_2^{k},\nabla \varphi_{N}\big)
+a_3\big(\zeta_2^{k}, \varphi_{N}\big)+a_4\big(\zeta_1^k, \varphi_{N}\big)\\=&-a_1\tau^{-\gamma}\sum^k_{j=k_0+1}\varepsilon^{(\gamma)}_{k-j}\omega_{k-j}^{(\gamma)}(\zeta_2^j, \varphi_{N})
-a_2\tau^{-\gamma}\sum^k_{j=k_0+1}\varepsilon^{(\gamma)}_{k-j}\omega_{k-j}^{(\gamma)}(\nabla\zeta_2^j, \nabla\varphi_{N})-H_2^k(\varphi_{N}),
\end{split}
\label{f5}
\end{eqnarray}
\begin{eqnarray}
\begin{split}
&\big(\partial_t \zeta_3^{k},\varphi_{N}\big)+b_1\big({D}^{1-\beta}_{\tau}\zeta_3^{k},\varphi_{N}\big)
+b_2\big({D}^{-\beta}_{\tau}\nabla \zeta_3^{k},\nabla \varphi_{N}\big)-b_3\big({D}^{-\beta}_{\tau} \zeta_3^{k},\varphi_{N}\big)
-b_4(\zeta_3^{k},\varphi_{N}\big)\\
=&-b_1\tau^{\beta-1}\sum^k_{j=k_0+1}\varepsilon^{(1-\beta)}_{k-j}\omega_{k-j}^{(1-\beta)}(\zeta_3^{j}, \varphi_{N})+b_3\tau^{\beta}\sum^k_{j=k_0+1}\varepsilon^{(-\beta)}_{k-j}\omega_{k-j}^{(-\beta)}(\zeta_3^{j},\varphi_{N})\\&-b_2\tau^{\beta}\sum^k_{j=k_0+1}\varepsilon^{(-\beta)}_{k-j}\omega_{k-j}^{(-\beta)}(\nabla\zeta_3^{j}, \nabla\varphi_{N})-H_3^k(\varphi_{N}),
\end{split}
\label{f6}
\end{eqnarray}
where
\begin{eqnarray*}
\begin{split}
H^k_1(\varphi_{N})=&a_1\tau^{-\gamma}\sum^k_{j=k_0+1}\varepsilon^{(\gamma)}_{k-j}\omega_{k-j}^{(\gamma)}(u_N^j, \varphi_{N})
+a_2\tau^{-\gamma}\sum^k_{j=k_0+1}\varepsilon^{(\gamma)}_{k-j}\omega_{k-j}^{(\gamma)}(\nabla u_N^j, \nabla\varphi_{N}),\\
H^k_2(\varphi_{N})=&a_1\tau^{-\gamma}\sum^k_{j=k_0+1}\varepsilon^{(\gamma)}_{k-j}\omega_{k-j}^{(\gamma)}(v_N^j, \varphi_{N})
+a_2\tau^{-\gamma}\sum^k_{j=k_0+1}\varepsilon^{(\gamma)}_{k-j}\omega_{k-j}^{(\gamma)}(\nabla v_N^j, \nabla\varphi_{N}),\\
H^k_3(\varphi_{N})=&b_1\tau^{\beta-1}\sum^k_{j=k_0+1}\varepsilon^{(1-\beta)}_{k-j}\omega_{k-j}^{(1-\beta)}(\theta_N^{j}, \varphi_{N})+b_2\tau^{\beta}\sum^k_{{j}=k_0+1}\varepsilon^{(-\beta)}_{k-j}\omega_{k-j}^{(-\beta)}(\nabla\theta_N^{j}, \nabla\varphi_{N})
\\&-b_3\tau^{\beta}\sum^k_{j=k_0+1}\varepsilon^{(-\beta)}_{k-j}\omega_{k-j}^{(-\beta)}(\theta_N^{j},\varphi_{N}).
\end{split}
\end{eqnarray*}

We take $\varphi_{N}=\zeta_1^k$ in (\ref{f4}) and sum it over $k$ from $k_0+1$ to $n$. By Lemma \ref{le4.1} and Corollary \ref{co4.4}, we obtain
\begin{eqnarray}
\begin{split}
&\|\zeta_{1}^n\|^2+4\tau\sum_{k=k_0+1}^n\|\nabla \zeta_1^{k}\|^2
+4a_3\tau\sum_{k=k_0+1}^n\|\zeta_1^{k}\|^2-4a_4\tau\sum_{k=k_0+1}^n\big(\zeta_2^k, \zeta_1^{k}\big)
-4a_5\tau\sum_{k=k_0+1}^n\big(\zeta_3^k, \zeta_1^{k}\big)\\ \le & -4a_1\tau^{1-\gamma}\sum_{k=k_0+1}^n\sum^k_{j=k_0+1}\varepsilon^{(\gamma)}_{k-j}\omega_{k-j}^{(\gamma)}(\zeta_1^j, \zeta_1^{k})
-4a_2\tau^{1-\gamma}\sum_{k=k_0+1}^n\sum^k_{j=k_0+1}\varepsilon^{(\gamma)}_{k-j}\omega_{k-j}^{(\gamma)}(\nabla\zeta_1^j, \nabla\zeta_1^{k})\\&-4\tau\sum_{k=k_0+1}^n H_1^k(\zeta^k_{1}).
\end{split}
\label{f8}
\end{eqnarray}
Note that $\omega^{(\gamma)}_0=(\frac{3}{2})^{\gamma}$ and $\omega_k^{(\gamma)}=O(k^{-\gamma-1})(k\ge1)$, so we have that
\begin{eqnarray}
\begin{split}
\sum^n_{k=j}\tau^{-\gamma}|\varepsilon_{k-j}^{(\gamma)}\omega_{k-j}^{(\gamma)}|\le C_{f,1}\varepsilon, \sum^k_{j=k_0+1}\tau^{-\gamma}|\varepsilon^{(\gamma)}_{k-j}\omega_{k-j}^{(\gamma)}|\le C_{f,1}\varepsilon.
\end{split}
\end{eqnarray}
Thus, it follows that
\begin{eqnarray}
\begin{split}
&-4a_1\tau^{1-\gamma}\sum_{k=k_0+1}^n\sum^k_{j=k_0+1}\varepsilon^{(\gamma)}_{k-j}\omega_{k-j}^{(\gamma)}(\zeta_1^j, \zeta_1^{k}) \\ \le& 2a_1\tau\sum_{j=k_0+1}^n\sum^n_{k=j}\tau^{-\gamma}|\varepsilon^{(\gamma)}_{k-j}\omega_{k-j}^{(\gamma)}| \|\zeta_1^j\|^2+ 2a_1\tau\sum_{k=k_0+1}^n\sum^k_{j=k_0+1}\tau^{-\gamma}|\varepsilon^{(\gamma)}_{k-j}\omega_{k-j}^{(\gamma)}| \|\zeta_1^{k}\|^2
 \\ \le &  4C_{f,1}a_1\varepsilon\tau\sum_{k=k_0+1}^n \|\zeta_1^{k}\|^2.
\end{split}
\label{f9}
\end{eqnarray}
Similarly, we have
\begin{eqnarray}
\begin{split}
-4a_2\tau^{1-\gamma}\sum_{k=k_0+1}^n\sum^k_{j=k_0+1}\varepsilon^{(\gamma)}_{k-j}\omega_{k-j}^{(\gamma)}(\nabla\zeta_1^j, \nabla\zeta_1^{k})\le
4C_{f,1}a_2\varepsilon\tau\sum_{k=k_0+1}^n \|\nabla\zeta_1^{k}\|^2,
\end{split}
\label{f11}
\end{eqnarray}
and
\begin{eqnarray}
\begin{split}
-4\tau\sum_{k=k_0+1}^nH_1^k(\zeta^k_{1}) \le& 2C_{f,1}a_1 \varepsilon\tau\sum_{k=k_0+1}^n \left(\|u_N^k\|^2+ \|\zeta_1^{k}\|^2\right)+ 2C_{f,1} a_2\varepsilon\tau\sum_{k=k_0+1}^n \left( \|\nabla u_N^k\|^2+ \|\nabla \zeta_1^{k}\|^2\right).
\end{split}
\label{f12}
\end{eqnarray}
Then, (\ref{f8}) becomes
\begin{eqnarray}
\begin{split}
&\|\zeta_{1}^n\|^2+4\tau\sum_{k=k_0+1}^n\|\nabla \zeta_1^{k}\|^2
+4a_3\tau\sum_{k=k_0+1}^n\|\zeta_1^{k}\|^2-4a_4\tau\sum_{k=k_0+1}^n\big(\zeta_2^k, \zeta_1^{k}\big)\\ \le&2a_5\tau\sum_{k=k_0+1}^n\| \zeta_3^{k}\|^2+2\tau\sum_{k=k_0+1}^n\| \zeta_1^{k}\|^2+4C_{f,1}a_1\varepsilon\tau\sum_{k=k_0+1}^n \|\zeta_1^{k}\|^2+4C_{f,1}a_2\varepsilon\tau\sum_{k=k_0+1}^n \|\nabla\zeta_1^{k}\|^2\\&+C_{f,2} \varepsilon\tau\sum_{k=k_0+1}^n (\|u_N^k\|^2+\|\nabla u_N^k\|^2)+{2C_{f,1}a_1 \varepsilon\tau\sum_{k=k_0+1}^n \|\zeta_1^{k}\|^2}+2C_{f,1}a_2 \varepsilon\tau\sum_{k=k_0+1}^n \|\nabla \zeta_1^{k}\|^2,
\end{split}
\label{f13}
\end{eqnarray}
{
which can be simplified to
\begin{eqnarray}
\begin{split}
\|\zeta_{1}^n\|^2 \le& 4a_4\tau\sum_{k=k_0+1}^n\big(\zeta_2^k, \zeta_1^{k}\big)+ 2a_5\tau\sum_{k=k_0+1}^n\| \zeta_3^{k}\|^2+2\tau\sum_{k=k_0+1}^n\|\zeta_1^{k}\|^2+C_{f,3}\varepsilon\tau\sum_{k=k_0+1}^n \|\zeta_1^{k}\|^2\\&+C_{f,2} \varepsilon\tau\sum_{k=k_0+1}^n (\|u_N^k\|^2+ \|\nabla u_N^k\|^2),
\end{split}
\label{f14}
\end{eqnarray}
where we have used the relation $\varepsilon \le \frac{N^{-4}}{2C_{f,1}a_2}$ so that the term $\|\nabla \zeta_1^{k}\|^2$ on the right-hand side of (\ref{f13}) can be eliminated.}


Similarly, taking $\varphi_{N}=\zeta_2^k$ in (\ref{f5}), we have that
\begin{eqnarray}
\begin{split}
\|\zeta_{2}^n\|^2&\le-4a_4\tau\sum_{k=k_0+1}^n\big(\zeta_1^k, \zeta_2^{k}\big)+C_{f,4}\varepsilon\tau\sum_{k=k_0+1}^n \|\zeta_2^{k}\|^2+ C_{f,4} \varepsilon\tau\sum_{k=k_0+1}^n (\|v_N^k\|^2+\|\nabla v_N^k\|^2).
\end{split}
\label{f16}
\end{eqnarray}

Then, taking $\varphi_{N}=\zeta_3^k$ in (\ref{f6}) and summing it over $k$ from $k_0+1$ to $n$, Lemma \ref{le4.1} and Corollary \ref{co4.4} imply that
\begin{eqnarray}
\begin{split}
\|\zeta_{3}^n\|^2 \le& 4b_3\tau\sum_{k=k_0+1}^n\big({D}^{-\beta}_{\tau} \zeta_3^{k},\zeta_3^k\big)+4b_4\tau\sum_{k=k_0+1}^n\|\zeta_3^{k}\|^2
-4b_1\tau^\beta\sum_{k=k_0+1}^n\sum^k_{j=k_0+1}\varepsilon^{(1-\beta)}_{k-j}\omega_{k-j}^{(1-\beta)}(\zeta_3^{j},\zeta_3^k)
\\&-4b_2\tau^{1+\beta}\sum_{k=k_0+1}^n\sum^k_{j=k_0+1}\varepsilon^{(-\beta)}_{k-j}\omega_{k-j}^{(-\beta)}(\nabla\zeta_3^{j}, \nabla\zeta_3^k)\\&
+4b_3\tau^{1+\beta}\sum_{k=k_0+1}^n\sum^k_{j=k_0+1}\varepsilon^{(-\beta)}_{k-j}\omega_{k-j}^{(-\beta)}(\zeta_3^{j},\zeta_3^k)
-4\tau\sum_{k=k_0+1}^n H_3^k(\zeta_3^k).
\end{split}
\label{f18}
\end{eqnarray}
Note that
\begin{eqnarray}
\begin{split}
4b_3\tau^{1+\beta}\sum^n\limits_{k=k_0+1}\sum^{k}\limits_{j=k_0+1}\omega_{k-j}^{(-\beta)}\big(\zeta^{j}_{3},\zeta^{k}_{3}\big)\le 4C_{f,5}\bar{T}^{\beta}b_3\tau \sum_{k=k_0+1}^n \lVert\zeta^{k}_{3}\lVert^2,
\end{split}
\label{f19}
\end{eqnarray}
\begin{eqnarray}
\begin{split}
-4b_1\tau^{\beta}\sum_{k=k_0+1}^n\sum^k_{j=k_0+1}\varepsilon_{k-j}^{(1-\beta)}\omega_{k-j}^{(1-\beta)}(\zeta_3^j, \zeta_3^{k})\le
4C_{f,6}b_1\varepsilon\tau\sum_{k=k_0+1}^n\|\zeta_3^{k}\|^2,
\end{split}
\label{f21}
\end{eqnarray}
\begin{eqnarray}
\begin{split}
&-4b_2\tau^{1+\beta}\sum_{k=k_0+1}^n\sum^k_{j=k_0+1}\varepsilon^{(-\beta)}_{k-j}\omega_{k-j}^{(-\beta)}(\nabla\zeta_3^{j}, \nabla\zeta_3^k)\le
4C_{f,7}b_2\varepsilon\tau\sum_{k=k_0+1}^n\|\nabla\zeta_3^{k}\|^2,
\end{split}
\label{f22}
\end{eqnarray}
\begin{eqnarray}
\begin{split}
4b_3\tau^{1+\beta}\sum_{k=k_0+1}^n\sum^k_{j=k_0+1}\varepsilon^{(-\beta)}_{k-j}\omega_{k-j}^{(-\beta)}(\zeta_3^{j},\zeta_3^k)\le
4C_{f,7}b_3\varepsilon\tau\sum_{k=k_0+1}^n\|\zeta_3^{k}\|^2,
\end{split}
\label{f23}
\end{eqnarray}
\begin{eqnarray}
\begin{split}
-4\tau\sum_{k=k_0+1}^nH_3^k(\zeta^k_{3})\le& 2C_{f,6}b_1 \varepsilon\tau\sum_{k=k_0+1}^n \left(\|\theta_N^k\|^2+ \|\zeta_3^{k}\|^2\right)+2C_{f,7} b_2\varepsilon\tau\sum_{k=k_0+1}^n \left(\|\nabla \theta_N^k\|^2+ \|\nabla \zeta_3^{k}\|^2\right)\\&
+2C_{f,7} b_3\varepsilon\tau\sum_{k=k_0+1}^n \left(\|\theta_N^k\|^2+ \|\zeta_3^{k}\|^2\right),
\end{split}
\label{f24}
\end{eqnarray}
and so (\ref{f18}) can be written as
\begin{eqnarray}
\begin{split}
\|\zeta_{3}^n\|^2 \le& 4C_{f,5}\bar{T}^{\beta}b_3\tau \sum_{k=k_0+1}^n \lVert\zeta^{k}_{3}\lVert^2+4b_4\tau\sum_{k=k_0+1}^n\|\zeta_3^{k}\|^2
\\&+4C_{f,6}b_1\varepsilon\tau\sum_{k=k_0+1}^n\|\zeta_3^{k}\|^2
+4C_{f,7}b_2\varepsilon\tau\sum_{k=k_0+1}^n\|\nabla\zeta_3^{k}\|^2
\\&+4C_{f,7}b_3\varepsilon\tau\sum_{k=k_0+1}^n\|\zeta_3^{k}\|^2+2C_{f,6}b_1 \varepsilon\tau\sum_{k=k_0+1}^n \left(\|\theta_N^k\|^2+ \|\zeta_3^{k}\|^2\right)\\&+2C_{f,7} b_2\varepsilon\tau\sum_{k=k_0+1}^n \left(\|\nabla \theta_N^k\|^2+ \|\nabla \zeta_3^{k}\|^2\right)
+2C_{f,7} b_3\varepsilon\tau\sum_{k=k_0+1}^n \left(\|\theta_N^k\|^2+ \|\zeta_3^{k}\|^2\right).
\end{split}
\label{f25}
\end{eqnarray}
This can be further simplified to
\begin{eqnarray}
\begin{split}
\|\zeta_{3}^n\|^2\le& C_{f,8}\tau\sum_{k=k_0+1}^n\|\zeta_3^{k}\|^2+C_{f,9}\varepsilon\tau \sum_{k=k_0+1}^n \lVert\zeta^{k}_{3}\lVert^2
\\&+C_{f,10} \varepsilon\tau\sum_{k=k_0+1}^n\|\theta_N^k\|^2+C_{f,11}\varepsilon\tau\sum_{k=k_0+1}^n\|\nabla \theta_N^k\|^2,
\end{split}
\label{f26}
\end{eqnarray}
where we have used the inequality $\lVert\nabla \theta\lVert\le CN^2\lVert\theta\lVert$ \cite{40} and the relation $\varepsilon \le\frac{N^{-4}}{2C_{f,1}a_2}$.


{
From (\ref{f14})--(\ref{f16}) and (\ref{f26}), and using Theorem \ref{th4.5}, Lemma \ref{le5.1} and Gronwall inequality for sufficiently small $\tau$, we obtain the final result that
\begin{eqnarray}
\begin{split}
\|u^n_N-{_Fu^n_N}\|^2+\|v^n_N-{_Fv^n_N}\|^2+\|\theta^n_N-{_F\theta^n_N}\|^2\le C \varepsilon.
\end{split}
\end{eqnarray}}
\end{proof}
{
\begin{remark}
Theorem 3 shows that the errors between the solutions $u_N^n, v_N^n, \theta_N^n$ of the direct method  (\ref{4.5})-(\ref{4.7}) and the solutions $_Fu_N^n, _Fv_N^n, _F\theta_N^n$ of the fast time-stepping spectral scheme (\ref{6.1})-(\ref{6.4}) are small enough. Then according to Theorem 1, we can prove the fast time-stepping spectral scheme (\ref{6.1})-(\ref{6.4}) is stable.
\end{remark}}
\section{Numerical experiments and results}
\label{sec:6}
We now describe the detailed numerical implementation and present some numerical results to verify the theoretical analysis of the FBDF Legendre spectral method.
\subsection{Numerical implementation}
Let $\hat{z}\in[-1,1]$ and let $L_j(\hat{z})$ denote the Legendre polynomial of degree $j$ which is defined by the recurrence relation \cite{40}
\begin{eqnarray}
L_{0}(\hat{z})=1, \quad L_{1}(\hat{z})=\hat{z}, \quad L_{j+1}(\hat{z})=\frac{(2j+1)\hat{z}}{j+1}L_{j}(\hat{z})-\frac{j}{j+1}L_{j-1}(\hat{z}).
\end{eqnarray}
For the spatial discretization, the basis functions are considered in the form
\begin{eqnarray*}
\psi_j(z)=L_j(\hat{z})-L_{j+2}(\hat{z}),\quad j=0,1,\cdots,N-2,
\end{eqnarray*}
where $z=\frac{L}{2}(\hat{z}+1)\in[0,L]$.\\

The approximation space $V_{N}^0$ can be given as $$V_{N}^0=\operatorname{span}\{\psi_j(z), 0\le j\le N-2\},$$
and for any functions $u^k_{N}, v^k_{N}, \theta^k_{N}\in V_{N}^0$, we have
$$
u^k_{N}=\sum\limits_{j=0}^{N-2}\widehat{u}^k_{j}\psi_j(y),\quad v^k_{N}=\sum\limits_{j=0}^{N-2}\widehat{v}^k_{j}\psi_j(y),\quad \theta^k_{N}=\sum\limits_{j=0}^{N-2}\widehat{\theta}^k_{j}\psi_j(y),
$$
where $\{\widehat{u}^k_{j}\}_{j=0}^{N-2}$, $\{\widehat{v}^k_{j}\}^{N-2}_{j=0}$, and $\{\widehat{\theta}^k_{j}\}^{N-2}_{j=0}$ are the expansion coefficients of $u_{N}^k$, $v_{N}^k$, and $\theta_{N}^k$, respectively.

Denoting $\varpi_1=\frac{1}{\tau}+\frac{a_1\omega_{0}^{(\gamma)}}{\tau^{\gamma}}+a_3$, $\varpi_k=\frac{3}{2\tau}+\frac{a_1\omega_{0}^{(\gamma)}}{\tau^{\gamma}}+a_3(k\ge 2)$,
$\varrho_1=\frac{1}{\tau}+\frac{b_1\omega_{0}^{(1-\beta)}}{\tau^{1-\beta}}-\frac{b_3\omega_{0}^{(-\beta)}}{\tau^{-\beta}}-b_4$,
and $\varrho_k=\frac{3}{2\tau}+\frac{b_1\omega_{0}^{(1-\beta)}}{\tau^{1-\beta}}-\frac{b_3\omega_{0}^{(-\beta)}}{\tau^{-\beta}}-b_4(k\ge 2)$,
we represent the fully discrete scheme in (\ref{4.5})--(\ref{4.7}) as follows:
\begin{eqnarray}
\begin{split}
\varpi_k\big( u_{N}^{k},\varphi_{N}\big)+\bigg(\frac{a_2\omega_{0}^{(\gamma)}}{\tau^{\gamma}}+1\bigg)\big( \nabla u_{N}^{k}, \nabla \varphi_{N}\big)-a_4\big( v_{N}^{k},\varphi_{N}\big)-a_5\big( \theta_{N}^{k},\varphi_{N}\big)=F^{k}(\varphi_{N}),
\end{split}
\label{n1}
\end{eqnarray}
\begin{eqnarray}
\begin{split}
\varpi_k\big( v_{N}^{k},\varphi_{N}\big)+\bigg(\frac{a_2\omega_{0}^{(\gamma)}}{\tau^{\gamma}}+1\bigg)\big( \nabla v_{N}^{k}, \nabla \varphi_{N}\big)+a_4\big( u_{N}^{k},\varphi_{N}\big)=G^{k}(\varphi_{N}),
\end{split}
\label{n2}
\end{eqnarray}
\begin{eqnarray}
\begin{split}
\varrho_k\big( \theta_{N}^{k},\varphi_{N}\big)+\frac{b_2\omega_{0}^{(-\beta)}}{\tau^{-\beta}}\big( \nabla \theta_{N}^{k}, \nabla \varphi_{N}\big)=P^{k}(\varphi_{N}),
\end{split}
\label{n3}
\end{eqnarray}
where
$$
F^{1}(\varphi_{N})=\bigg(\frac{1}{\tau}-\frac{a_1\omega_{1}^{(\gamma)}}{\tau^{\gamma}}\bigg)\big( u_{N}^{0},\varphi_{N}\big)
-\frac{a_2\omega_{1}^{(\gamma)}}{\tau^{\gamma}}\big( \nabla u_{N}^{0}, \nabla \varphi_{N}\big)+(f^1,\varphi_{N}),
$$
$$
G^{1}(\varphi_{N})=\bigg(\frac{1}{\tau}-\frac{a_1\omega_{1}^{(\gamma)}}{\tau^{\gamma}}\bigg)\big( v_{N}^{0},\varphi_{N}\big)
-\frac{a_2\omega_{1}^{(\gamma)}}{\tau^{\gamma}}\big( \nabla v_{N}^{0}, \nabla \varphi_{N}\big)+(g^1,\varphi_{N}),
$$
$$
P^{1}(\varphi_{N})=\bigg(\frac{1}{\tau}-\frac{b_1\omega_{1}^{(1-\beta)}}{\tau^{1-\beta}}+\frac{b_3\omega_{1}^{(-\beta)}}{\tau^{-\beta}}\bigg)\big( \theta_{N}^{0},\varphi_{N}\big)
-\frac{b_2\omega_{1}^{(-\beta)}}{\tau^{-\beta}}\big( \nabla \theta_{N}^{0}, \nabla \varphi_{N}\big)+(\widetilde{p}^1,\varphi_{N}),
$$
and for $k\ge2$,
$$
F^{k}(\varphi_{N})=\frac{2}{\tau}\big( u_{N}^{k-1},\varphi_{N}\big)-\frac{1}{2\tau}\big( u_{N}^{k-2},\varphi_{N}\big)
-\frac{a_1}{\tau^{\gamma}}\sum^{k-1}_{j=0}\omega_{k-j}^{(\gamma)}( u_{N}^{j},\varphi_{N}\big)
-\frac{a_2}{\tau^{\gamma}}\sum^{k-1}_{j=0}\omega_{k-j}^{(\gamma)}( \nabla u_{N}^{j},\nabla \varphi_{N}\big)+(f^k,\varphi_{N}),
$$
$$
G^{k}(\varphi_{N})=\frac{2}{\tau}\big( v_{N}^{k-1},\varphi_{N}\big)-\frac{1}{2\tau}\big( v_{N}^{k-2},\varphi_{N}\big)
-\frac{a_1}{\tau^{\gamma}}\sum^{k-1}_{j=0}\omega_{k-j}^{(\gamma)}( v_{N}^{j},\varphi_{N}\big)
-\frac{a_2}{\tau^{\gamma}}\sum^{k-1}_{j=0}\omega_{k-j}^{(\gamma)}( \nabla v_{N}^{j},\nabla \varphi_{N}\big)+(g^k,\varphi_{N}),
$$
\begin{eqnarray*}
\begin{split}
P^{k}(\varphi_{N})=&\frac{2}{\tau}\big( \theta_{N}^{k-1},\varphi_{N}\big)-\frac{1}{2\tau}\big( \theta_{N}^{k-2},\varphi_{N}\big)
-\frac{b_1}{\tau^{1-\beta}}\sum^{k-1}_{j=0}\omega_{k-j}^{(1-\beta)}( \theta_{N}^{j},\varphi_{N}\big)
-\frac{b_2}{\tau^{-\beta}}\sum^{k-1}_{j=0}\omega_{k-j}^{(-\beta)}( \nabla \theta_{N}^{j},\nabla \varphi_{N}\big)\\&
+\frac{b_3}{\tau^{-\beta}}\sum^{k-1}_{j=0}\omega_{k-j}^{(-\beta)}( \theta_{N}^{j}, \varphi_{N}\big)
+(\widetilde{p}^k,\varphi_{N}),
\end{split}
\end{eqnarray*}
in which the inner product $(f^{k},\varphi_{N}), (g^{k},\varphi_{N}),(\widetilde{p}^k,\varphi_{N}), k\ge 1$, can be calculated using the numerical quadrature formula.

Let $\varphi_{N}=\psi_l(z), l=0,1,\cdots,N-2$. Then, (\ref{n1})--(\ref{n3}) can be expressed as
\begin{eqnarray}
\begin{split}
&\sum^{N-2}_{j=0}\Bigg(\varpi_k\big( \psi_j(z),\psi_l(z)\big)+\bigg(\frac{a_2\omega_{0}^{(\gamma)}}{\tau^{\gamma}}+1\bigg)\big( \nabla \psi_j(z), \nabla \psi_l(z)\big)\Bigg)\widehat{u}^k_{j}
-a_4\sum^{N-2}_{j=0}\big( \psi_j(z),\psi_l(z)\big)\widehat{v}^k_{j}
\\&-a_5\sum^{N-2}_{j=0}\big(\psi_j(z),\psi_l(z)\big)\widehat{\theta}^k_{j}=F^{k}(\psi_l(z)),
\end{split}
\label{n4}
\end{eqnarray}
\begin{eqnarray}
\begin{split}
&\sum^{N-2}_{j=0}\Bigg(\varpi_k\big( \psi_j(z),\psi_l(z)\big)+\bigg(\frac{a_2\omega_{0}^{(\gamma)}}{\tau^{\gamma}}+1\bigg)\big( \nabla \psi_j(z), \nabla \psi_l(z)\big)\Bigg)\widehat{v}^k_{j}
+a_4\sum^{N-2}_{j=0}\big( \psi_j(z),\psi_l(z)\big)\widehat{u}^k_{j}\\&=G^{k}(\psi_l(z)),
\end{split}
\label{n5}
\end{eqnarray}
\begin{eqnarray}
\begin{split}
\sum_{j=0}^{N-2}\Bigg(\varrho_k\big( \psi_j(z),\psi_l(z)\big)+\frac{b_2\omega_{0}^{(-\beta)}}{\tau^{-\beta}}\big( \nabla \psi_j(z), \nabla \psi_l(z)\big)\Bigg)\widehat{\theta}^k_{j}={P}^{k}(\psi_l(z)).
\end{split}
\label{n6}
\end{eqnarray}

Therefore, the matrix representation of (\ref{n4})--(\ref{n6}) can be inferred as
\begin{eqnarray}
\begin{split}
\Bigg(\varpi_k S+\bigg(\frac{a_2\omega_{0}^{(\gamma)}}{\tau^{\gamma}}+1\bigg)D\Bigg)\mathbf{\widehat{U}^k}
-a_4S\mathbf{\widehat{V}^k}-a_5S\mathbf{\widehat{\Theta}^k}=\mathbf{F^{k}},
\end{split}
\label{n7}
\end{eqnarray}
\begin{eqnarray}
\begin{split}
\Bigg(\varpi_k S+\bigg(\frac{a_2\omega_{0}^{(\gamma)}}{\tau^{\gamma}}+1\bigg)D\Bigg)\mathbf{\widehat{V}^k}
+a_4S\mathbf{\widehat{U}^k}=\mathbf{G^{k}},
\end{split}
\label{n8}
\end{eqnarray}
\begin{eqnarray}
\begin{split}
\Bigg(\varrho_kS+\frac{b_2\omega_{0}^{(-\beta)}}{\tau^{-\beta}}D\Bigg)\mathbf{\widehat{\Theta}^k}=\mathbf{{P}^{k}},
\end{split}
\label{n9}
\end{eqnarray}
where
$$
S=\big(\psi_j(z),\psi_l(z)\big)^{N-2}_{j,l=0}, \quad D=\big(\nabla\psi_j(z),\nabla\psi_l(z)\big)^{N-2}_{j,l=0},\quad
\mathbf{\widehat{U}^k}=[\widehat{u}^k_{0},\widehat{u}^k_{1},\cdots,\widehat{u}^k_{N-2}]^T,$$
$$ \mathbf{\widehat{V}^k}=[\widehat{v}^k_{0},\widehat{v}^k_{1},\cdots,\widehat{v}^k_{N-2}]^T,\quad
\mathbf{\widehat{\Theta}^k}=[\widehat{\theta}^k_{0},\widehat{\theta}^k_{1},\cdots,\widehat{\theta}^k_{N-2}]^T,\quad \mathbf{F^{k}}=[{F}^{k}(\psi_0),{F}^{k-1}(\psi_1),\cdots,{F}^{k}(\psi_{N-2})]^T,$$
$$
\mathbf{G^{k}}=[{G}^{k}(\psi_0),{G}^{k}(\psi_1),\cdots,{G}^{k}(\psi_{N-2})]^T,\quad
\quad
\mathbf{{P}^{k}}=[{P}^{k}(\psi_0),
{P}^{k}(\psi_1),\cdots,{P}^{k}(\psi_{N-2})]^T.
$$

For the fast  algorithm, when $k\ge k_0+1$, $\omega^{(\gamma)}_{k}, \omega^{(1-\beta)}_{k}, \omega^{(-\beta)}_{k}$ are calculated using (\ref{ff1}), which reduces the computation time and memory requirements compared with the direct calculation method. Solving (\ref{n7})--(\ref{n9}) gives numerical solutions to the time-fractional MHD coupled flow and heat transfer model of the generalized second-grade fluid.
\subsection{Numerical results}
To verify the correctness of our theoretical analysis, we now present some results using the numerical scheme given by (\ref{4.5})--(\ref{4.7}). The error is approximated as $\text{Error}(\tau, N)=\text{error}_u(\tau, N)+\text{error}_v(\tau, N)+\text{error}_{\theta}(\tau, N)$, where $\text{error}_u(\tau, N), \text{error}_v(\tau, N), \text{error}_{\theta}(\tau, N)$ are $L^2$ errors related to $u, v, \theta$, respectively, and the temporal convergence order is derived as $\text{Order}=\text{log}_{\frac{\tau_1}{\tau_2}}\frac{\text{Error}(\tau_1,N)}{\text{Error}(\tau_2,N)}$.
\begin{example} We consider the following time-fractional coupled equations with $(z,t)\in[0,1]\times[0,1]$:
\begin{equation}
\begin{aligned}
\frac{\partial u}{\partial t}+{_{~~0}^{RL}D_t^{\gamma}}u-{_{~~0}^{RL}D_t^{\gamma}}\frac{\partial^{2} u}{\partial z^{2}}-\frac{\partial^{2} u}{\partial z^{2}}+u-v-\theta=f(z,t),
\end{aligned}
\end{equation}
\begin{equation}
\begin{aligned}
\frac{\partial v}{\partial t}+{_{~~0}^{RL}D_t^{\gamma}}v-{_{~~0}^{RL}D_t^{\gamma}}\frac{\partial^{2} v}{\partial z^{2}}-\frac{\partial^{2} u}{\partial z^{2}}+v+u=g(z,t),
\end{aligned}
\end{equation}
\begin{equation}
\begin{aligned}
\frac{\partial \theta}{\partial t}+{_{~~0}^{RL}D_t^{\beta}}\frac{\partial \theta}{\partial t}-\frac{\partial^2 \theta}{\partial z^2}-\theta- {_{~~0}^{RL}D_t^{\beta}}\theta=p(z,t),
\end{aligned}
\end{equation}
{
\begin{equation}
u(z,0)=0, v(z,0)=0, {\theta}(z,0)=0,\quad z\in[0,L],
\end{equation}
\begin{equation}
u(0,t)=0, v(0,t)=0, {\theta}(0,t)=0, \quad t\in[0,\bar{T}],
\end{equation}
\begin{equation}
u(L,t)=0, v(L,t)=0, {\theta}(L,t)=0, \quad t\in[0,\bar{T}],
\end{equation}}
where
\begin{eqnarray*}
\begin{split}
f(z,t)=&3t^2\sin(2\pi z)+\frac{6}{\Gamma(4-\gamma)}t^{3-\gamma}\sin(2\pi z)+\frac{24\pi^2}{\Gamma(4-\gamma)}t^{3-\gamma}\sin(2\pi z)+4\pi^2t^3\sin(2\pi z)\\&+t^3\sin(2\pi z)-t^2(z^2-z)-t^2\sin(2\pi z),\\
g(z,t)=&2t(z^2-z)+\frac{2}{\Gamma(3-\gamma)}t^{2-\gamma}(z^2-z)-\frac{4}{\Gamma(3-\gamma)}t^{2-\gamma}(z^2-z)-2t^2+t^2(z^2-z)\\&+t^3\sin(2\pi z),\\
p(z,t)=&2t\sin(2\pi z)+\frac{2}{\Gamma(2-\beta)}t^{1-\beta}\sin(2\pi z)+4\pi^2t^2\sin(2\pi z)-t^2\sin(2\pi z)\\&-\frac{2}{\Gamma(3-\beta)}t^{2-\beta}\sin(2\pi z),
\end{split}
\end{eqnarray*}
and $0<\gamma,\beta<1$. The exact solutions of this system are $$u(z,t)=t^3\sin(2\pi z),\quad v(z,t)=t^2(z^2-z),\quad \theta(z,t)=t^2\sin(2\pi z).$$
\end{example}

We first fix $N=32$ and examine different values of the time step $\tau$ to verify the temporal convergence accuracy.  Table \ref{ext1} lists the temporal $L^2$ errors and convergence order for different values of $\gamma,\beta$ with direct method and fast method. {These results demonstrate that the convergence of the two numerical schemes achieves second-order accuracy in time, which shows the fast method can't lower the convergence accuracy. Next, a time step of $\tau=1/1000$ is considered to test the spatial accuracy. The $L^2$ errors related to $N$ obtained by the direct method and fast method on a semi-logarithmic scale for different values of $\gamma,\beta$ are plotted in Figure \ref{exf1}. The spatial errors decrease exponentially with $N$, demonstrating the spectral accuracy in space.} To further reveal the efficiency of the fast algorithm, Figure \ref{exf2} displays the computation times of the fast method and the direct method with different time steps. Obviously, the computation time of the fast method increases linearly, whereas that of the direct method increases super-linearly (quadratic complexity). {Theoretically, the computational time of the
direct method increases proportional to $\bar{K}^2$ and the computational time of the fast method increases proportional to $Q\bar{K}$, which is consistent with Figure \ref{exf2}. Therefore, both theoretically and experimentally, it shows that the fast method is more efficient than direct method.} Furthermore, The difference between the fast method solutions and the direct method solutions for different values of $\gamma,\beta$ is illustrated in Figure \ref{exf3}. Clearly, the magnitude of the error is very small. These results indicate that the fast method can significantly reduce the computation time and can't introduce additional errors.

\begin{table}[htbp]
\caption{\label{tab:test}  {$L^2$-errors, convergence order for different $\gamma,\beta$ with the direct method and the fast method.}}
\centering
\begin{tabular}{llllllllllllllllllllllll}
\hline
&&&\multicolumn{1}{l}{\multirow{2}{*}{$\tau$}}   &&\multicolumn{3}{c}{\multirow{1}{*}{Direct method}} &&\multicolumn{3}{c}{\multirow{1}{*}{Fast method}} \\
                                           \cline{6-8}  \cline{10-12}
 &  && && Error       && Order     && Error       && Order                   \\
\hline
 &                          && 1/200           &&7.5330e-05     &&   -        && 7.5374e-05     &&    -                     \\
 &                          && 1/400           && 1.8846e-05    && 1.9990     &&  1.8868e-05   &&   1.9981                  \\
 & $\gamma=0.4, \beta=0.6$  && 1/800          &&4.7134e-06    &&   1.9994       &&4.7539e-06    &&    1.9888              \\
 &                          && 1/1600          && 1.1787e-06     && 1.9996        && 1.1820e-06     &&  2.0079               \\
 &                          && 1/3200           &&   2.9472e-07    &&  1.9998       &&  2.5748e-07    &&  2.1987                   \\
 \hline

&                           && 1/200           &&  1.0499e-04    &&    -        &&  1.0499e-04    &&   -  \\
&                           && 1/400           &&   2.6290e-05   &&   1.9976      &&    2.6290e-05   &&    1.9976              \\
&$\gamma=0.8,\beta=0.3$     && 1/800           &&   6.5795e-06   &&  1.9985     &&   6.5795e-06   &&  1.9985                  \\
&                           && 1/1600          &&   1.6460e-06    && 1.9990    &&   1.6461e-06    &&   1.9989                  \\
&                           && 1/3200          &&  4.1172e-07    &&   1.9993     && 4.1191e-07    &&     1.9987                \\

\hline
\end{tabular}
\label{ext1}
\end{table}\label{3}

\begin{figure}[htbp]
\centering
\includegraphics[width=7cm]{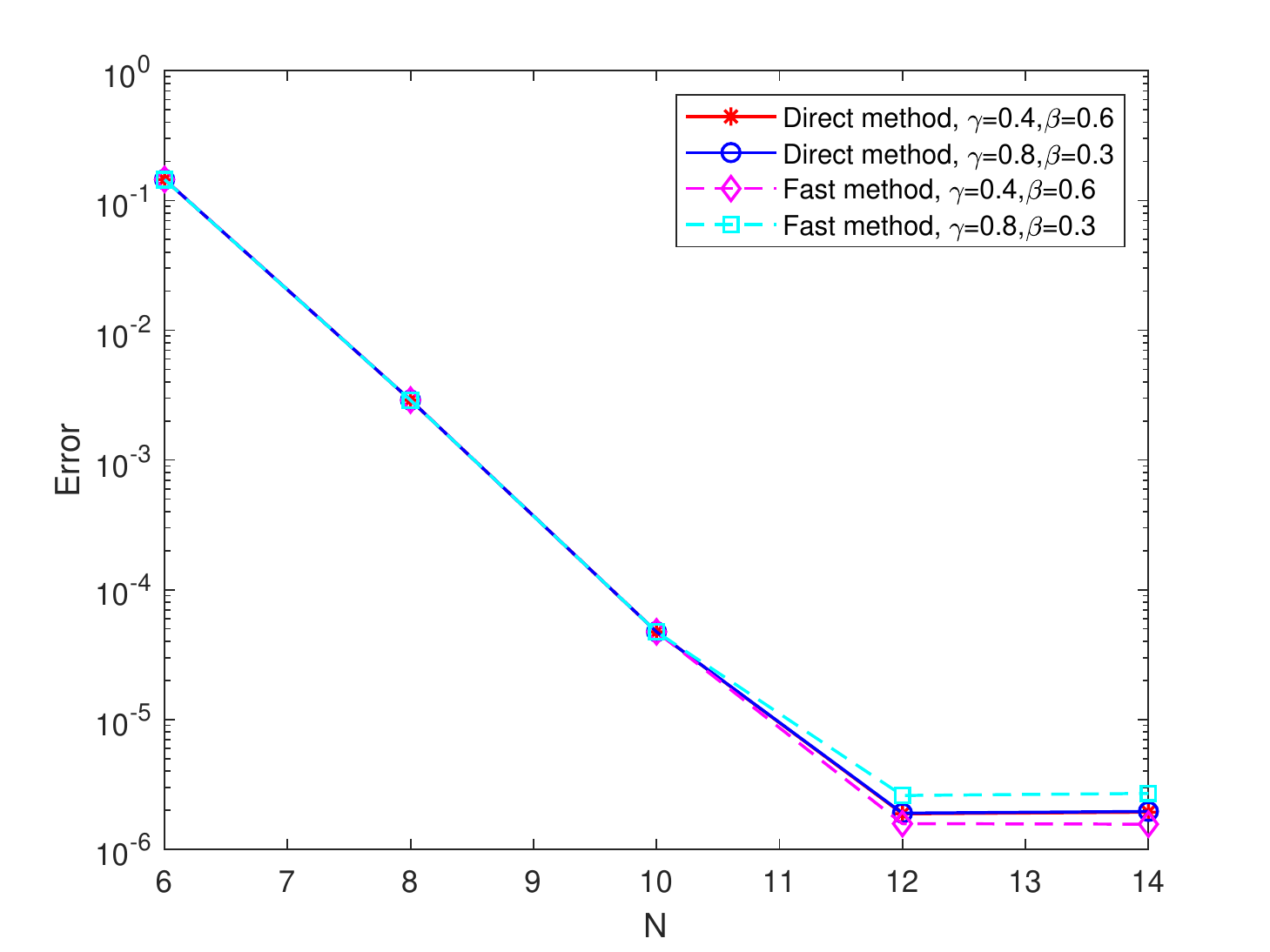}
\caption{ $L^2$-errors related to $N$ for different $\gamma,\beta$.}
\label{exf1}
\end{figure}

\begin{figure}[htbp]
\centering
\includegraphics[width=7cm]{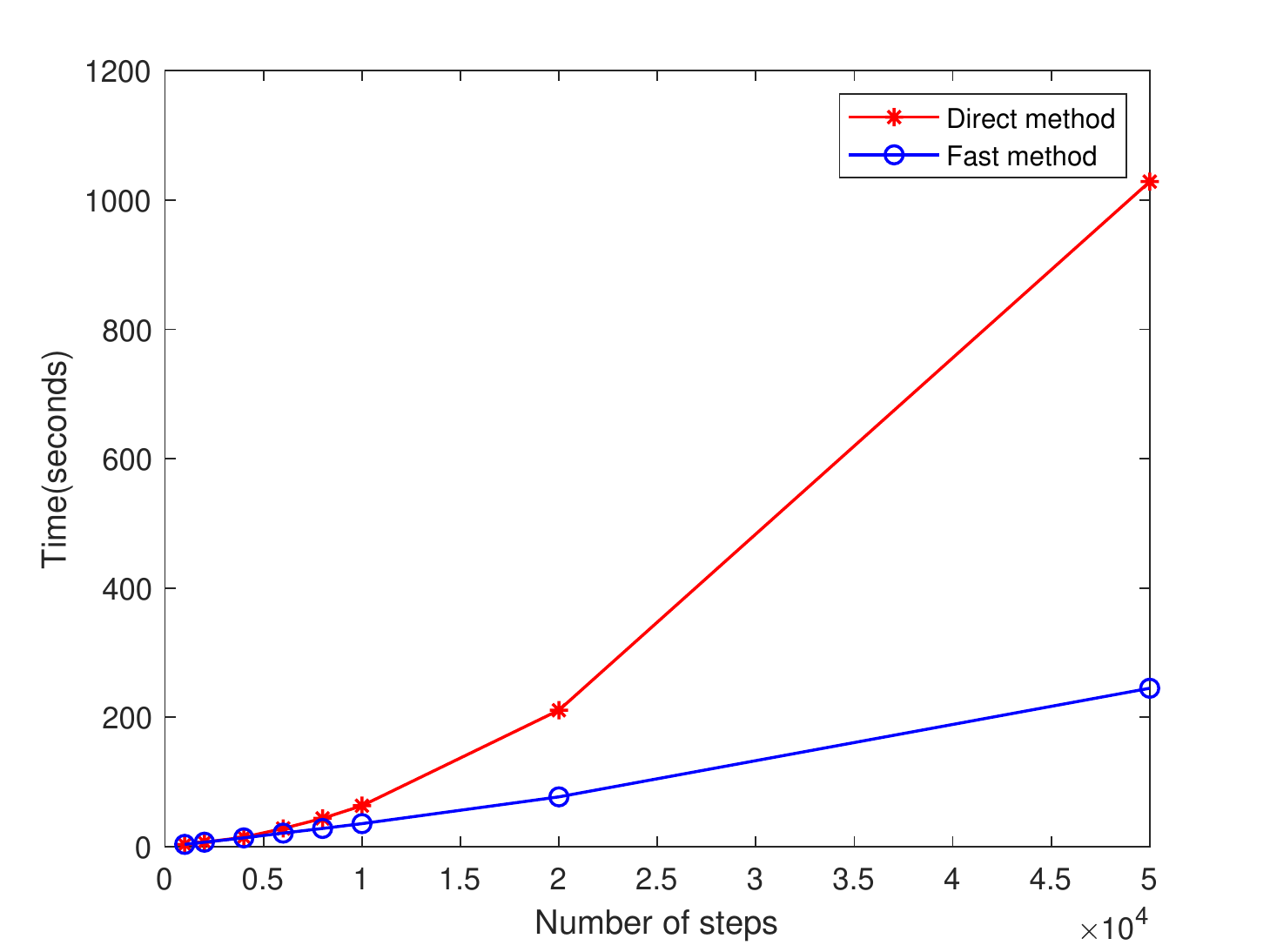}
\caption{ Computational time of the fast method and the direct method with $\gamma=0.6, \beta=0.4$.}
\label{exf2}
\end{figure}
\begin{figure}[htbp]
\centering
\includegraphics[width=7cm]{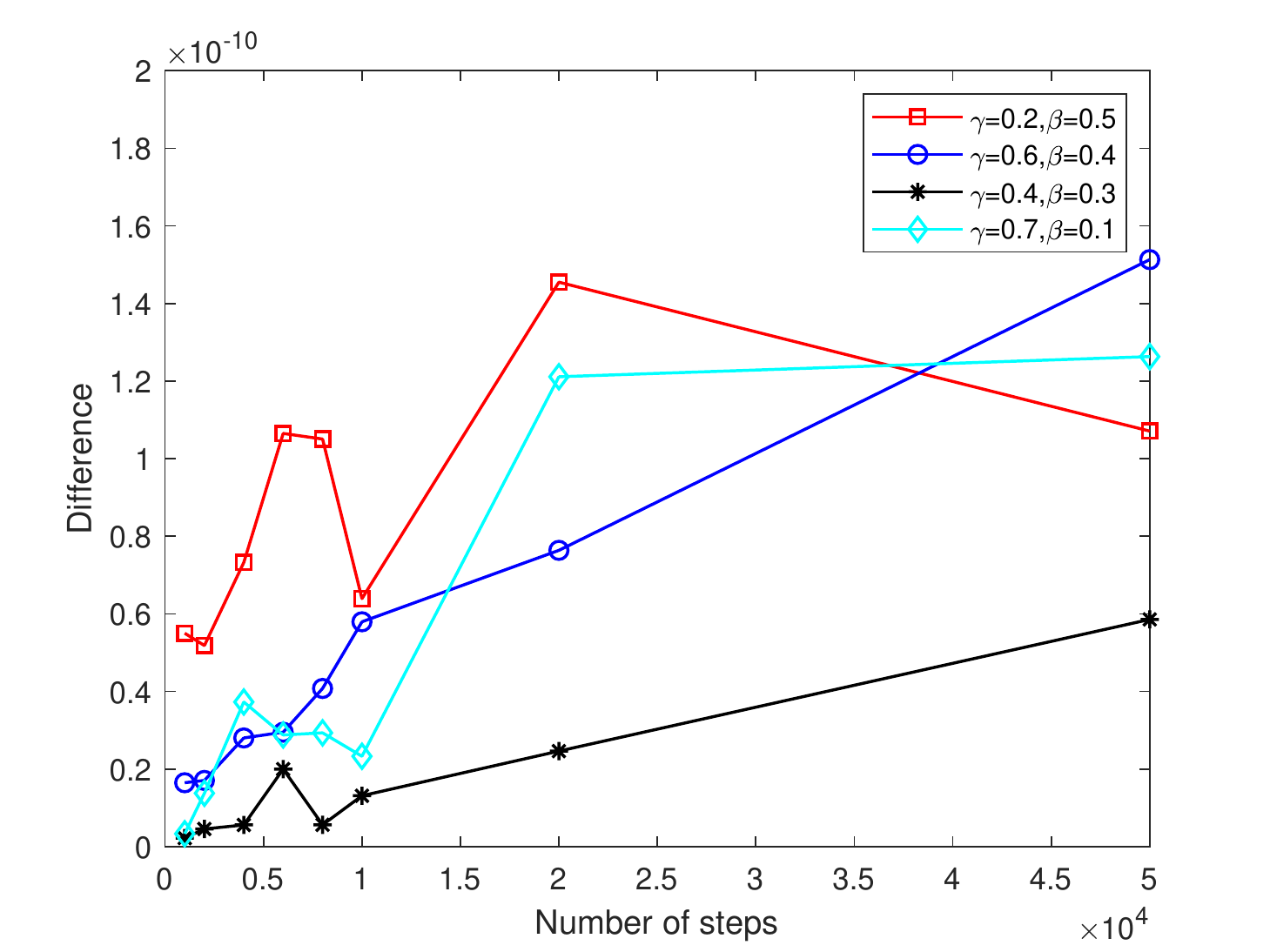}
\caption{ Difference between the numerical solutions of the fast method and the direct method.}
\label{exf3}
\end{figure}

\begin{example} We consider the unsteady free convection fractional MHD flow and heat transfer of a generalized second-grade fluid with a Hall current passing through a porous medium near a vertical infinite plate:
\begin{equation}
\begin{aligned}
\frac{\partial u}{\partial t}&=\left(1+\alpha\ {_{~~0}^{RL}D_{t}^{\gamma}}\right) \frac{\partial^{2} u}{\partial z^{2}}+\frac{M^{2}}{1+m^{2}}(m v-u)
-\frac{1}{K}\left(1+\alpha\ {_{~~0}^{RL}D_{t}^{\gamma}}\right) u+Gr\theta,
\end{aligned}
\end{equation}
\begin{equation}
\begin{aligned}
\frac{\partial v}{\partial t}&=\left(1+\alpha\ {_{~~0}^{RL}D_{t}^{\gamma}}\right) \frac{\partial^{2} v}{\partial z^{2}}-\frac{M^{2}}{1+m^{2}}(v+m u)
-\frac{1}{K}\left(1+\alpha\ {_{~~0}^{RL}D_{t}^{\gamma}}\right) v,
\end{aligned}
\end{equation}
\begin{equation}
\begin{aligned}
(1+\lambda\ {_{~~0}^{RL}D_t^{\beta}})\frac{\partial \theta}{\partial t}&=\frac{1+R}{Pr}\frac{\partial^2 \theta}{\partial z^2}+\frac{H}{Pr}(1+\lambda\ {_{~~0}^{RL}D_t^{\beta}})\theta,
\end{aligned}
\end{equation}
\begin{equation}
u=0, v=0, \theta=0, \text { as } z \geq 0 \text { and } t=0,
\end{equation}
\begin{equation}
u=t^{3}, v=0, \text { as } z=0 \text { and } t > 0,
\end{equation}
\begin{equation}
\theta=\left\{\begin{array}{ll}
t^2 & \text { if } 0<t<1, {z=0},\\
1 & \text { if } t \geq 1,{z=0},
\end{array}\right.
\end{equation}
\begin{equation}
u \rightarrow 0, v \rightarrow 0, \theta \rightarrow 0, \text { as } z \rightarrow \infty \text { and } t > 0.
\end{equation}
\end{example}

In the numerical simulation, we take $0\le z\le4$, $t=0.5$, and set $ m=1, M=2,K=2, R=1, Pr=2, H=1, Gr=10, \lambda=1$, several of these values are reasonable according to \cite{Jiang}. As no exact solutions can be determined for this problem, we use the numerical solutions with $\tau=1/40000, N=80$ obtained by the fast method for comparison. Table \ref{ex2t1} presents the temporal $L^2$ errors and convergence order for different values of $\gamma,\beta$ with $N=80$. From Table \ref{ex2t1}, it is clear that the convergence of the numerical scheme achieves second-order accuracy in time. Thus, the numerical scheme is stable and effective for solving the problem of unsteady free convection fractional MHD flow and heat transfer for a generalized second-grade fluid. Next, we set $\gamma=0.8, \beta=0.6, \tau=1/200, N=80$ to analyze the effects of the relevant parameters on the fractional MHD flow and heat transfer behavior in detail.

\begin{table}[htbp]\small
\caption{\label{tab:test}  $L^2$ errors, convergence order, CPU time  for different $\gamma,\beta$.}
\centering
\begin{tabular}{llllllllllllll}

\hline
 &                          &&& $\tau$      &&& Error          &&& Order                         \\
\hline

 &                          &&& 1/40        &&&  5.2020e-02     &&&                                     \\
 &                          &&& 1/80        &&& 1.2733e-02    &&&   2.0305                                    \\
 & $\gamma=0.4,\beta=0.6$   &&& 1/160        &&&  3.1208e-03      &&&    2.0286                          \\
 &                          &&& 1/320       &&& 7.6509e-04      &&&  2.0282                           \\
 &                          &&& 1/640       &&& 1.8556e-04     &&&  2.0438                           \\
 \hline
&                           &&& 1/40        &&&    5.2295e-02    &&&                   \\
&                           &&& 1/80        &&&  1.2690e-02     &&&    2.0430                \\
 & $\gamma=0.5,\beta=0.5$   &&& 1/160        &&&  3.0808e-03     &&&  2.0423                        \\
&                           &&& 1/320       &&&  7.4862e-04     &&& 2.0410                      \\
&                           &&& 1/640       &&&  1.8020e-04     &&&  2.0547                      \\
\hline
\end{tabular}
\label{ex2t1}
\end{table}
The profiles of velocity $u,v$ for various parameter combinations are shown in Figures \ref{gamma}--\ref{t}. Figure \ref{gamma} shows that the magnitudes of $u$ and $v$ decrease as the fractional order $\gamma$ increases. We can see that $\gamma$ plays an inhibiting role in the fluid flow. It is also apparent that an increment in $\gamma$ causes the thickness of the velocity boundary layer to increase, which indicates that the fractional equation with a relaxation time has a short memory of obvious states and responds slowly to an external body force. In Figure \ref{alpha}, the magnitudes of $u, v$ decrease and the profiles change rapidly as the second-grade parameter $\alpha$ increases, suggesting that $\alpha$ has a negative effect on the velocity. The effect of the permeability parameter $K$ on the velocity is presented in Figure \ref{K}. The magnitudes of $u$ and $v$ increase as $K$ becomes larger. The effect of the Hartmann number $M$ on the velocity profiles is illustrated in Figure \ref{M}. As the value of $M$ increases, the magnitudes of $u, v$ decrease. This is because the applied magnetic field produces a drag in the form of the Lorentz force, which slows the fluid flow. The effect of the Hall parameter $m$ on $u, v$ is shown in Figure \ref{mm}. There is a clear enhancement in the magnitudes of $u, v$ as $m$ increases. This is mostly because the Lorentz force decreases as $m$ increases, weakening the restraining effect of the magnetic field on the fluid flow. Figure \ref{Gr} displays the changes in $u, v$ with respect to the thermal Grashof number $Gr$. The velocity components $u, v$ increase in magnitude as $Gr$ becomes larger, which implies that $Gr$ can promote the fluid flow. Figure \ref{t} displays the profiles of velocity $u, v$ with respect to $t$, and shows that the magnitudes of $u, v$ increase over time.

\begin{figure*}[htbp]
\centering
\subfigure{
\begin{minipage}[t]{0.42\textwidth}
\centering
\includegraphics[width=6cm]{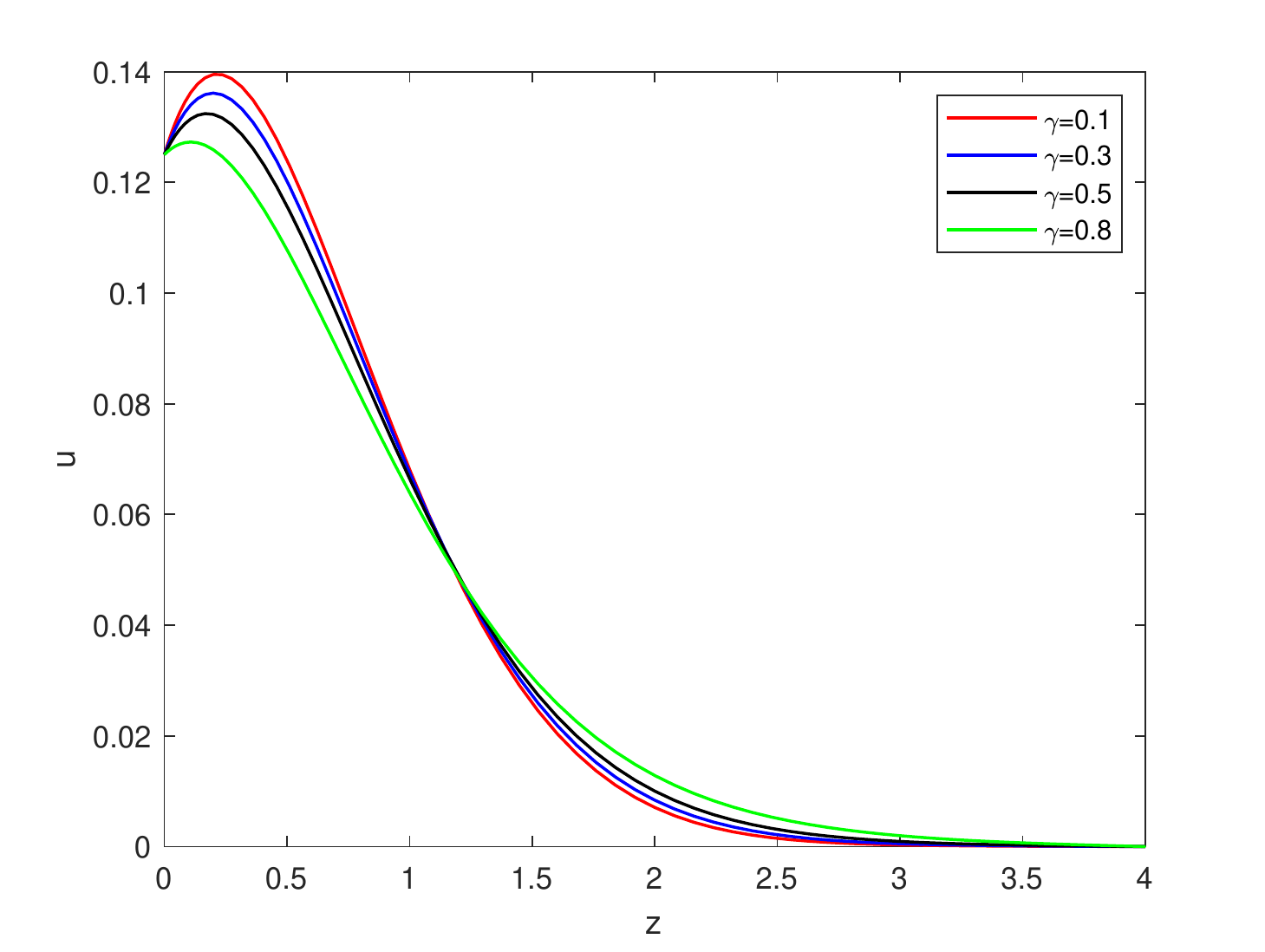}
\end{minipage}
}
\subfigure{
\begin{minipage}[t]{0.42\textwidth}
\centering
\includegraphics[width=6cm]{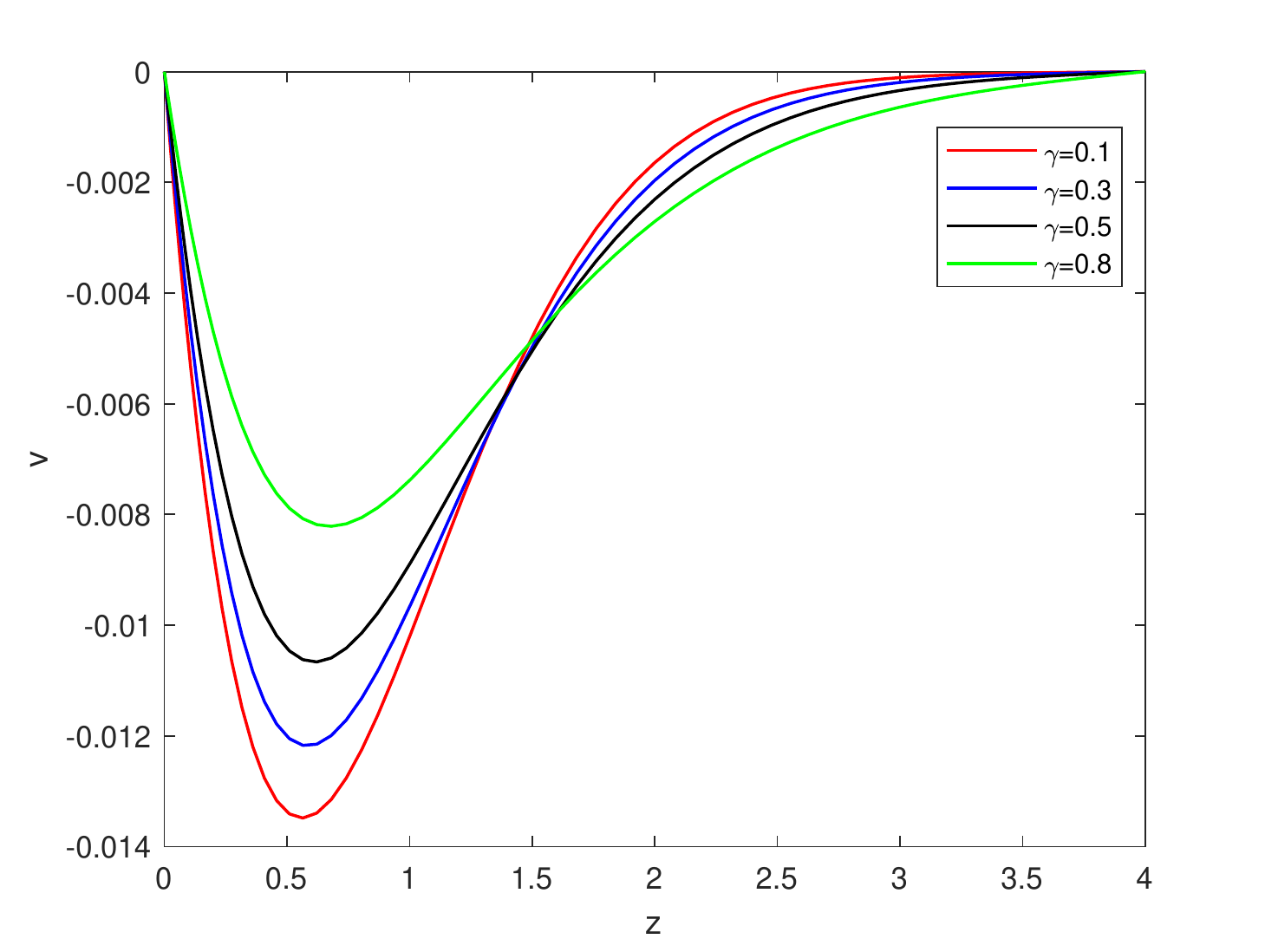}
\end{minipage}
}
\caption{ Velocity $u$ and $v$ for different $\gamma$. }
\label{gamma}
\end{figure*}
\begin{figure*}[htbp]
\centering
\subfigure{
\begin{minipage}[t]{0.42\textwidth}
\centering
\includegraphics[width=6cm]{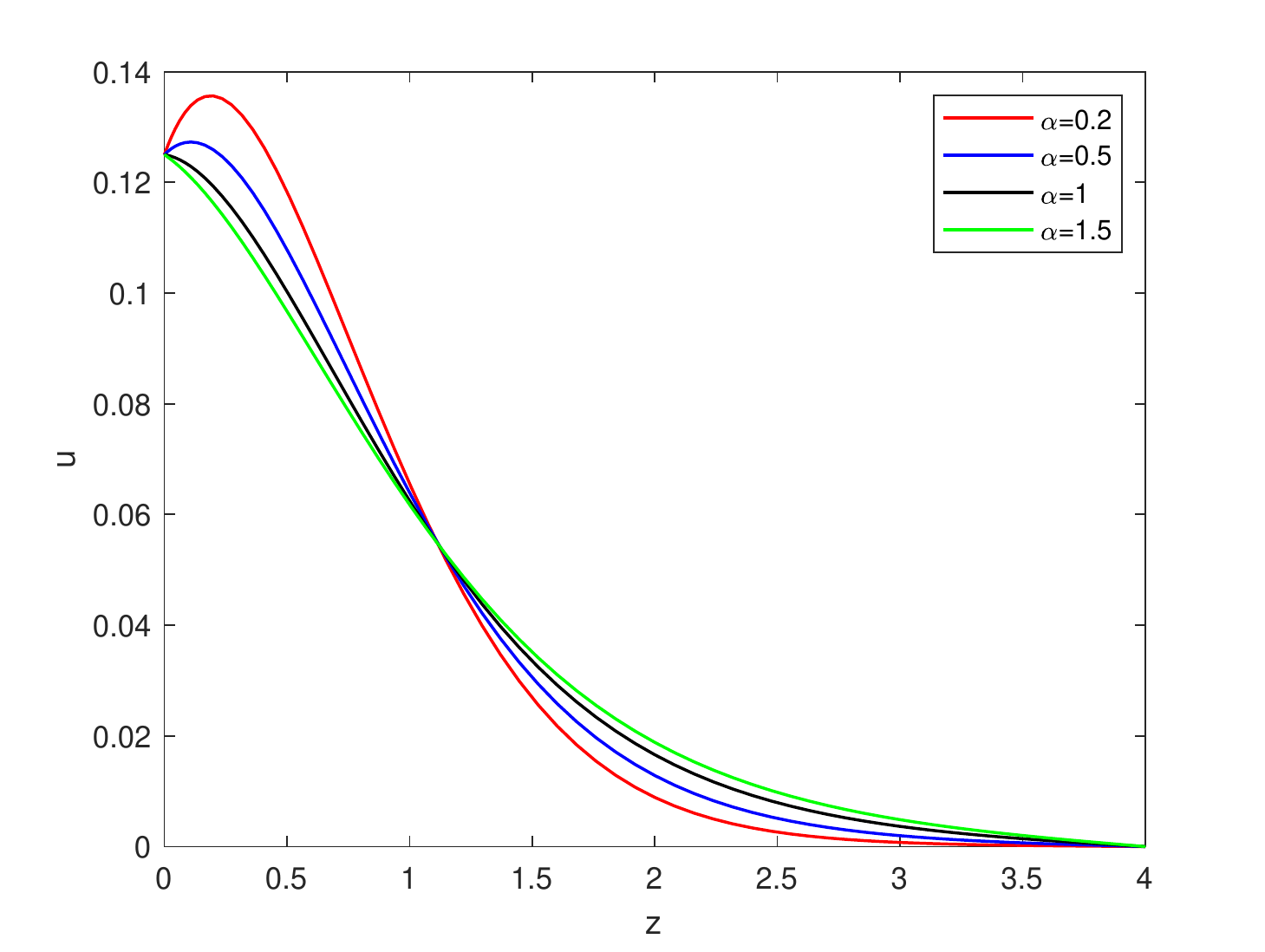}
\end{minipage}
}
\subfigure{
\begin{minipage}[t]{0.42\textwidth}
\centering
\includegraphics[width=6cm]{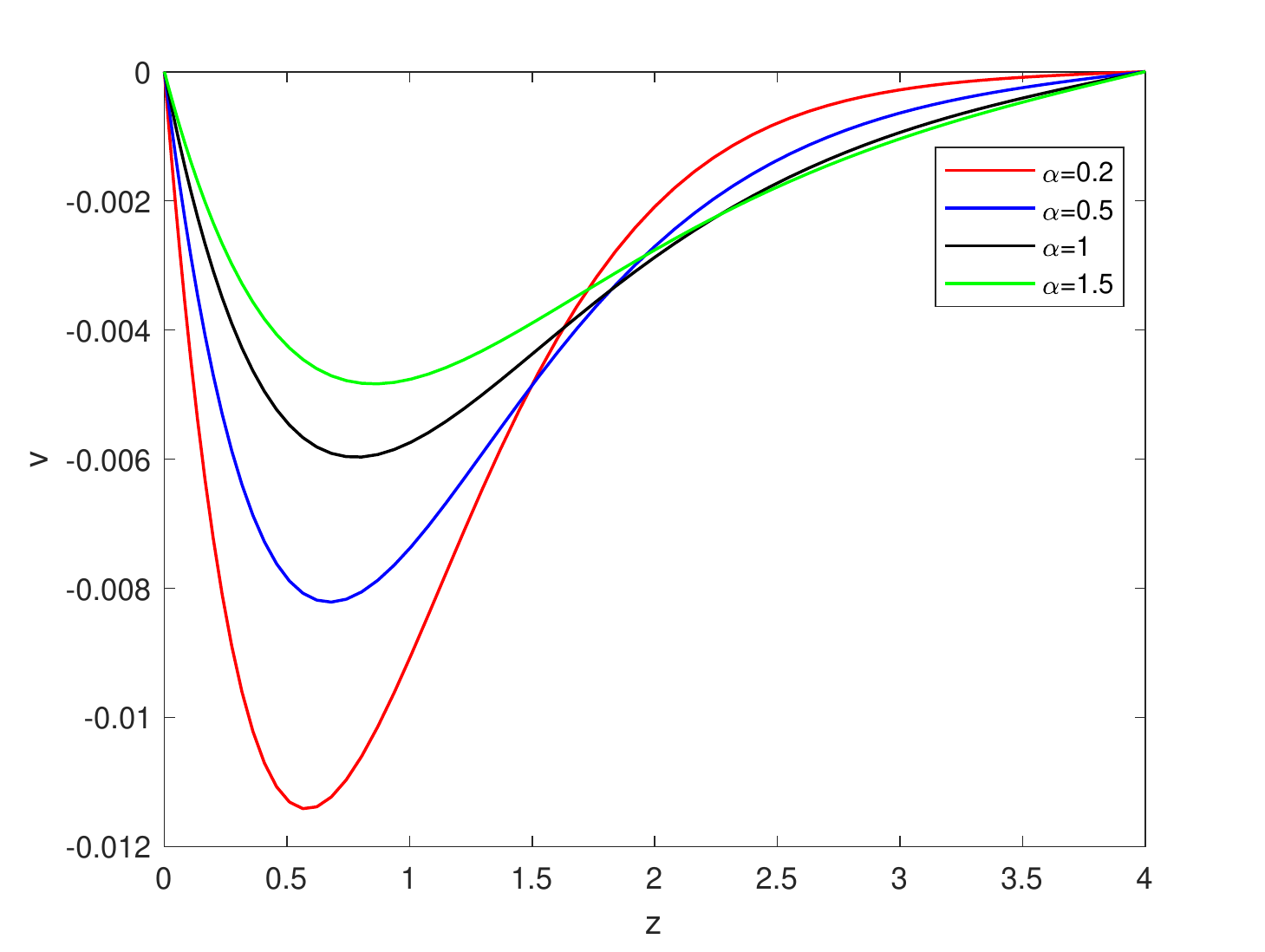}
\end{minipage}
}
\caption{ Velocity $u$ and $v$ for different $\alpha$. }
\label{alpha}
\end{figure*}
\begin{figure*}[htbp]
\centering
\subfigure{
\begin{minipage}[t]{0.42\textwidth}
\centering
\includegraphics[width=6cm]{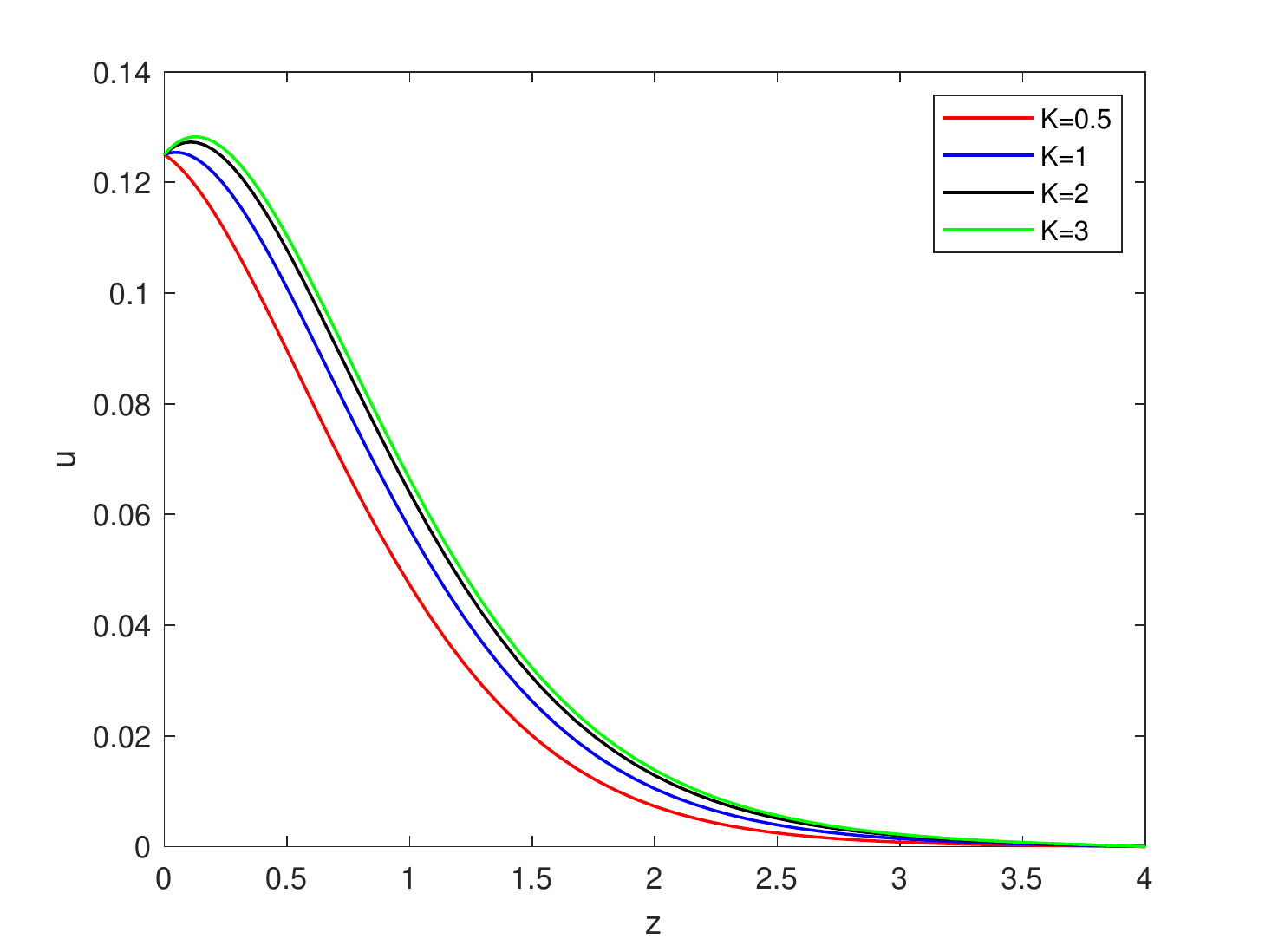}
\end{minipage}
}
\subfigure{
\begin{minipage}[t]{0.42\textwidth}
\centering
\includegraphics[width=6cm]{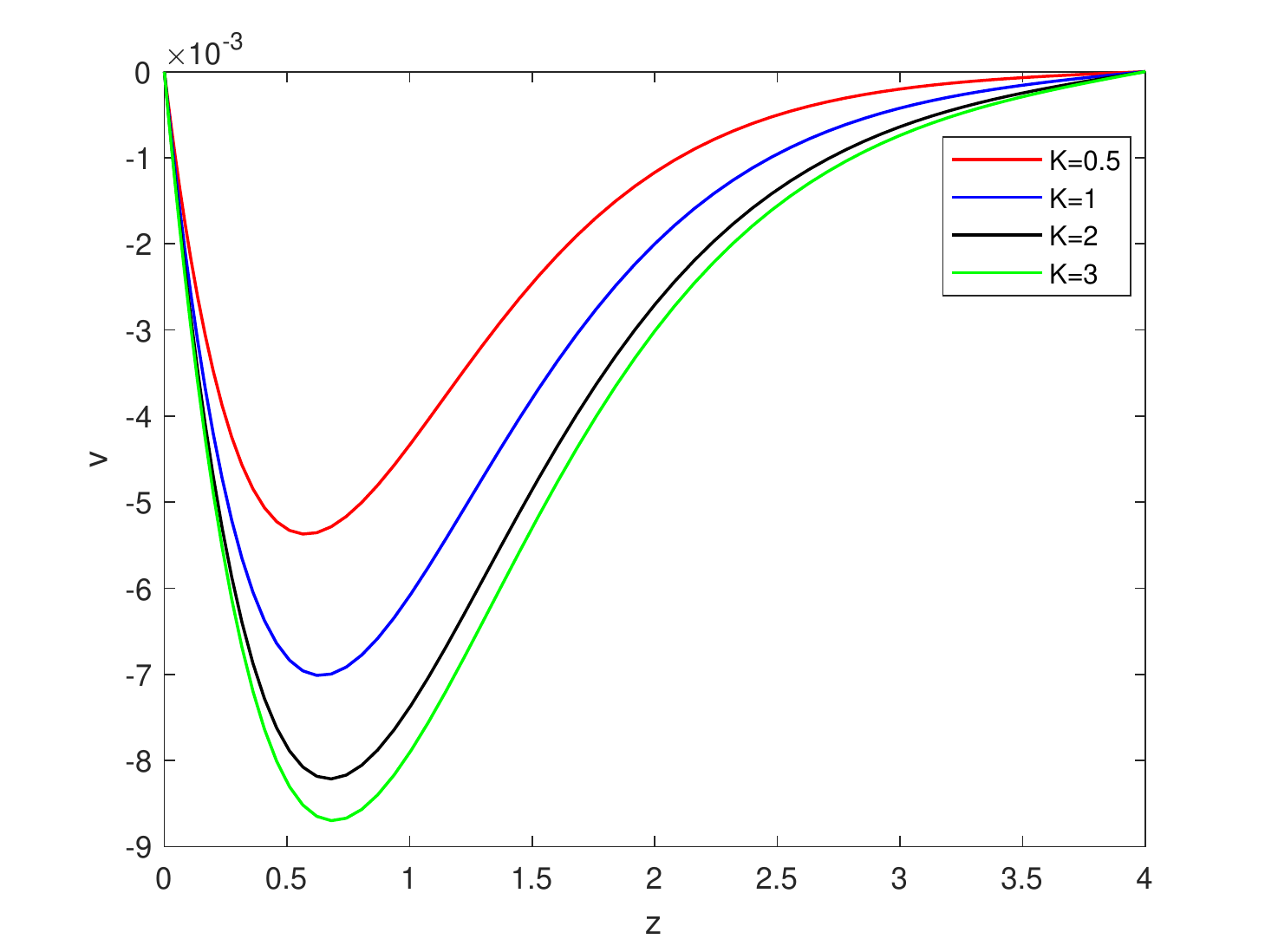}
\end{minipage}
}
\caption{ Velocity $u$ and $v$ for different $K$.}
\label{K}
\end{figure*}
\begin{figure*}[htbp]
\centering
\subfigure{
\begin{minipage}[t]{0.42\textwidth}
\centering
\includegraphics[width=6cm]{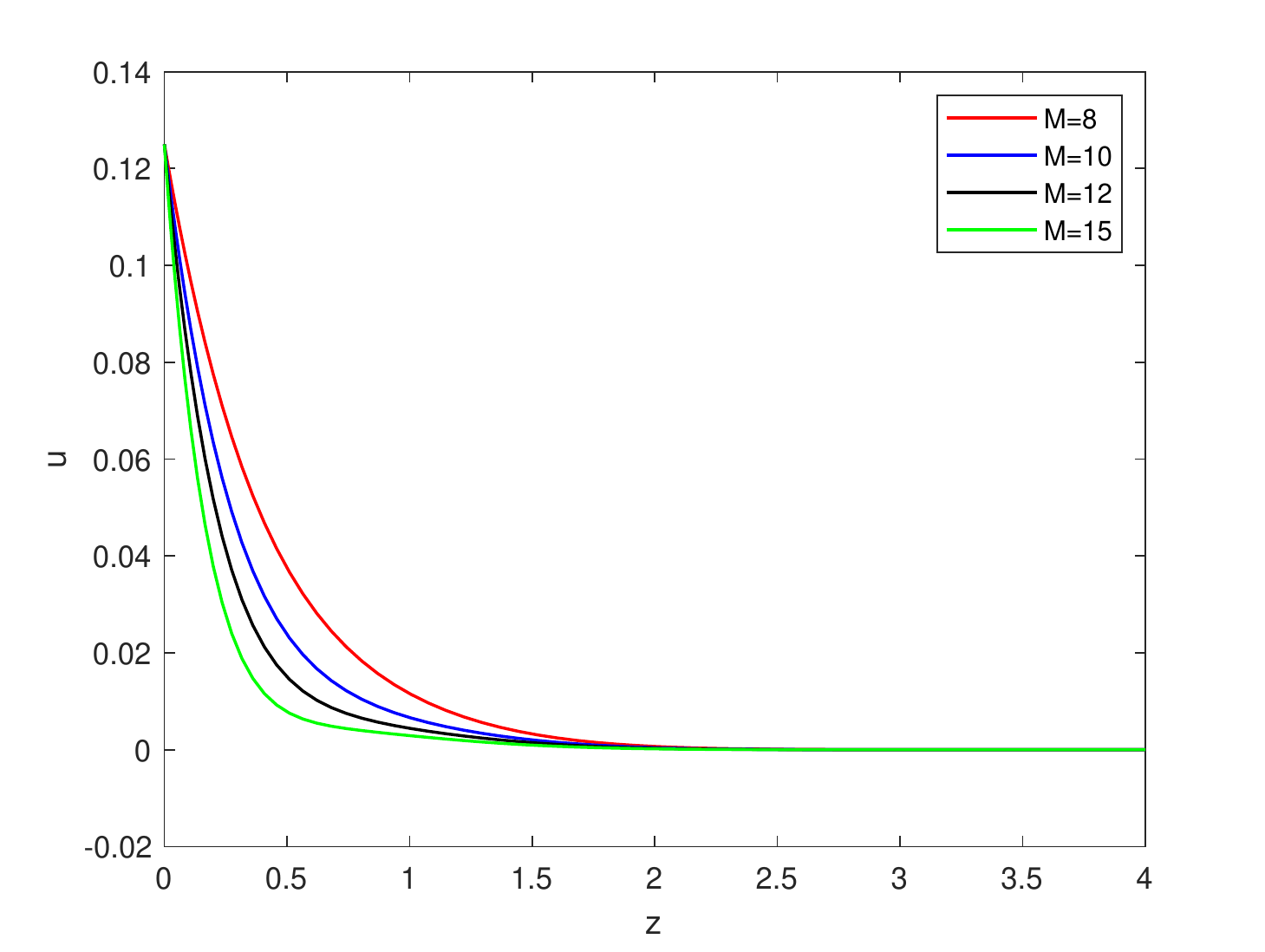}
\end{minipage}
}
\subfigure{
\begin{minipage}[t]{0.42\textwidth}
\centering
\includegraphics[width=6cm]{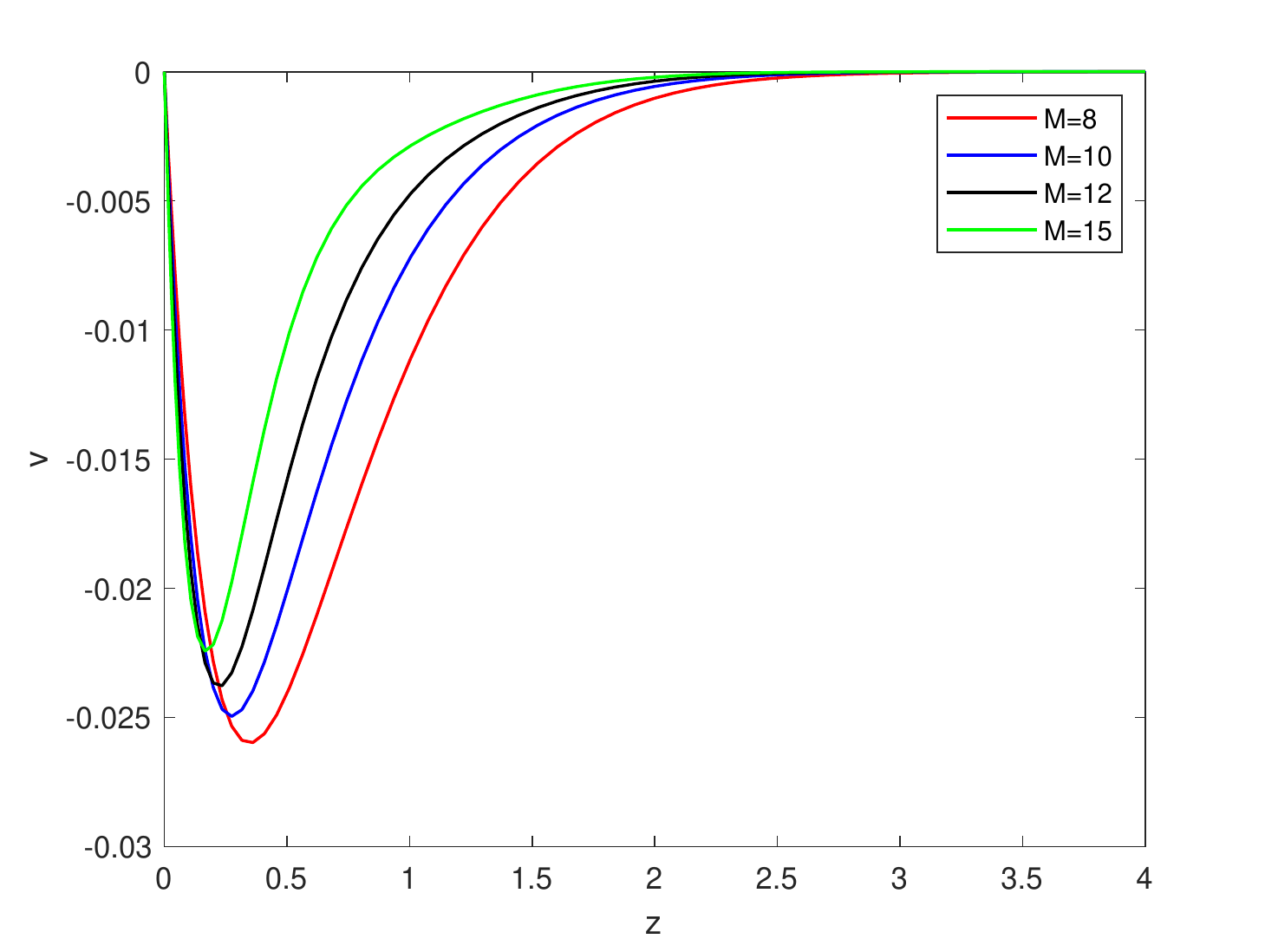}
\end{minipage}
}
\caption{ Velocity $u$ and $v$ for different $M$. }
\label{M}
\end{figure*}
\begin{figure*}[htbp]
\centering
\subfigure{
\begin{minipage}[t]{0.42\textwidth}
\centering
\includegraphics[width=5.8cm]{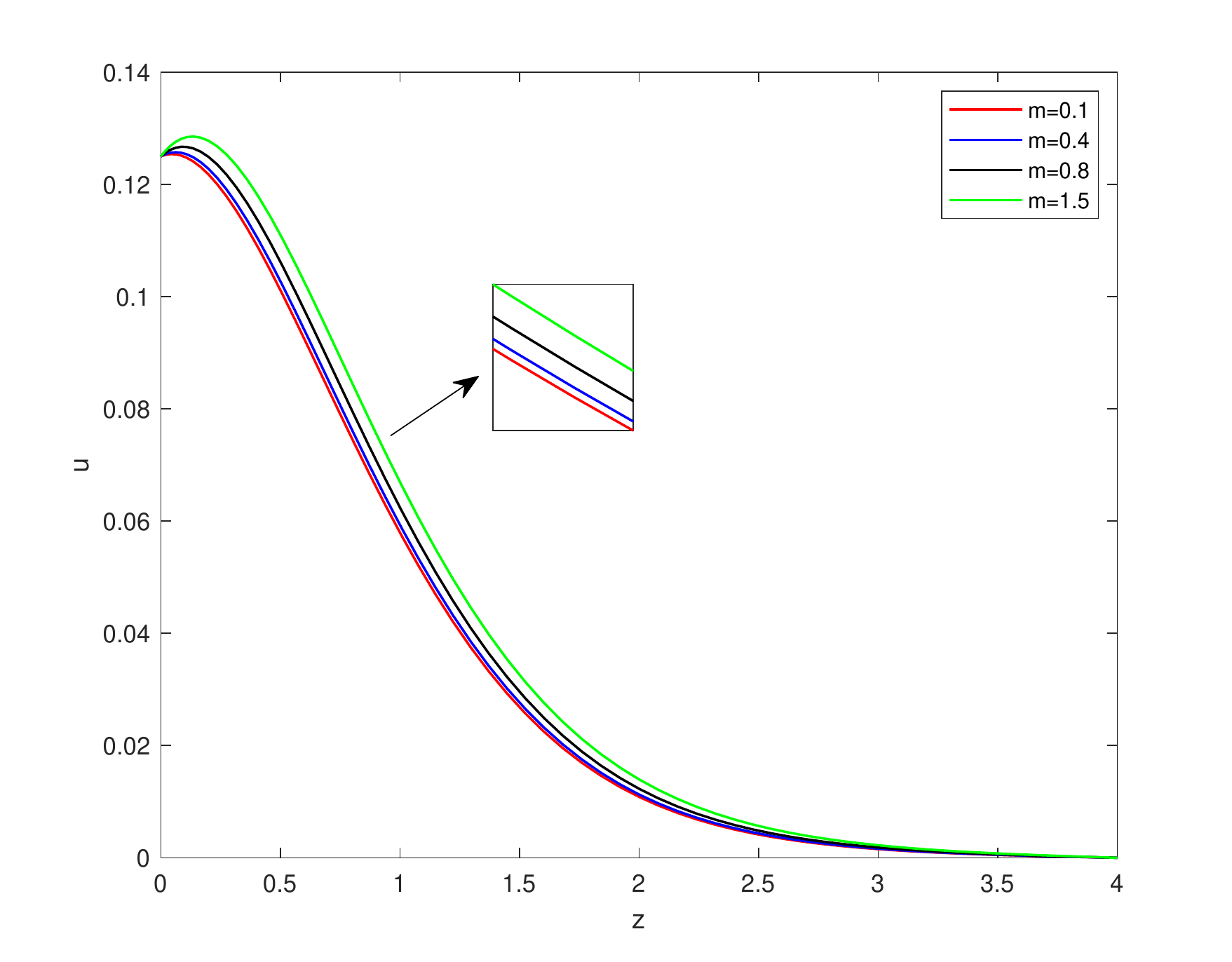}
\end{minipage}
}
\subfigure{
\begin{minipage}[t]{0.42\textwidth}
\centering
\includegraphics[width=6cm]{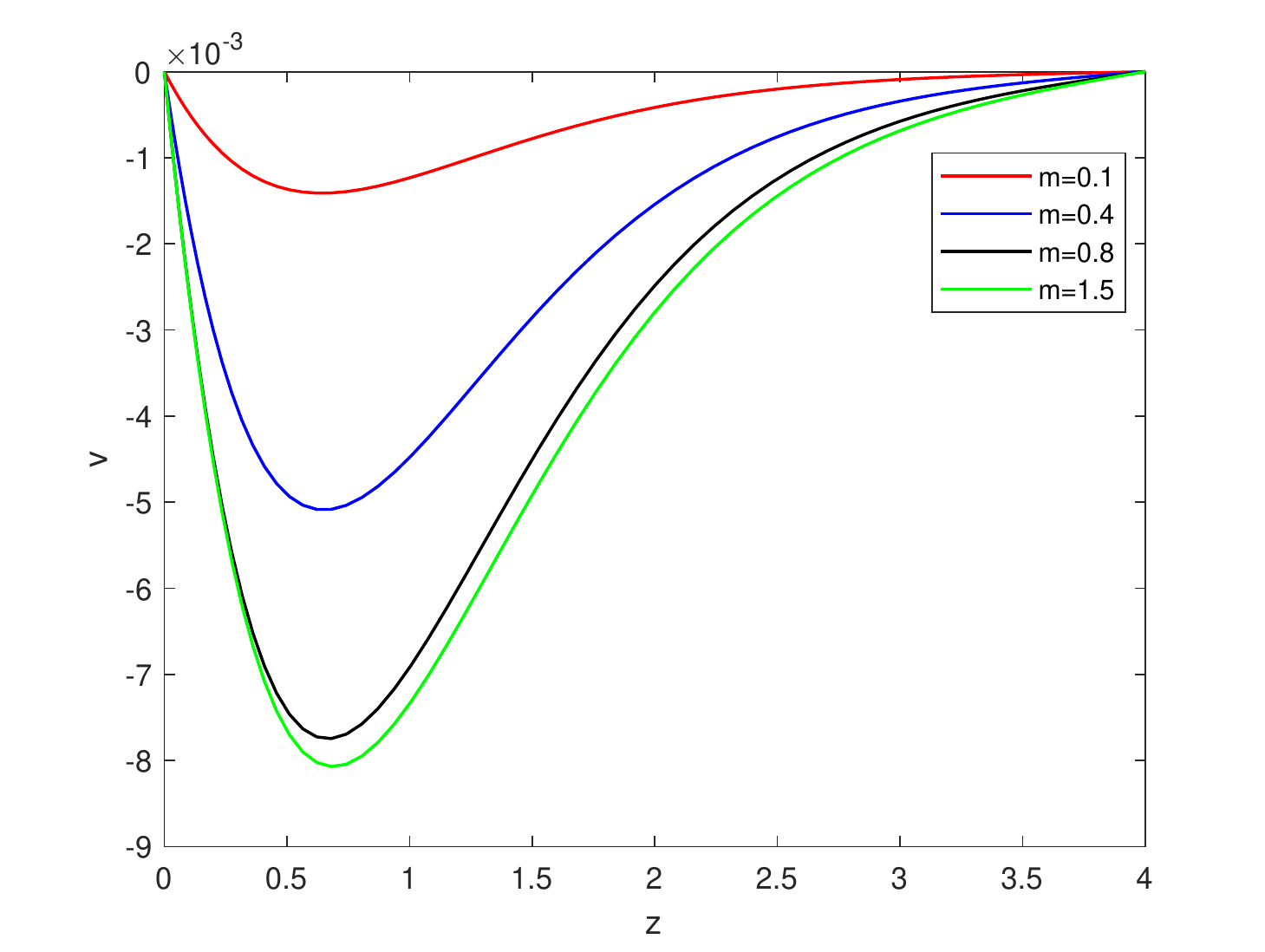}
\end{minipage}
}
\caption{ Velocity $u$ and $v$ for different $m$.}
\label{mm}
\end{figure*}
\begin{figure*}[htbp]
\centering
\subfigure{
\begin{minipage}[t]{0.42\textwidth}
\centering
\includegraphics[width=6cm]{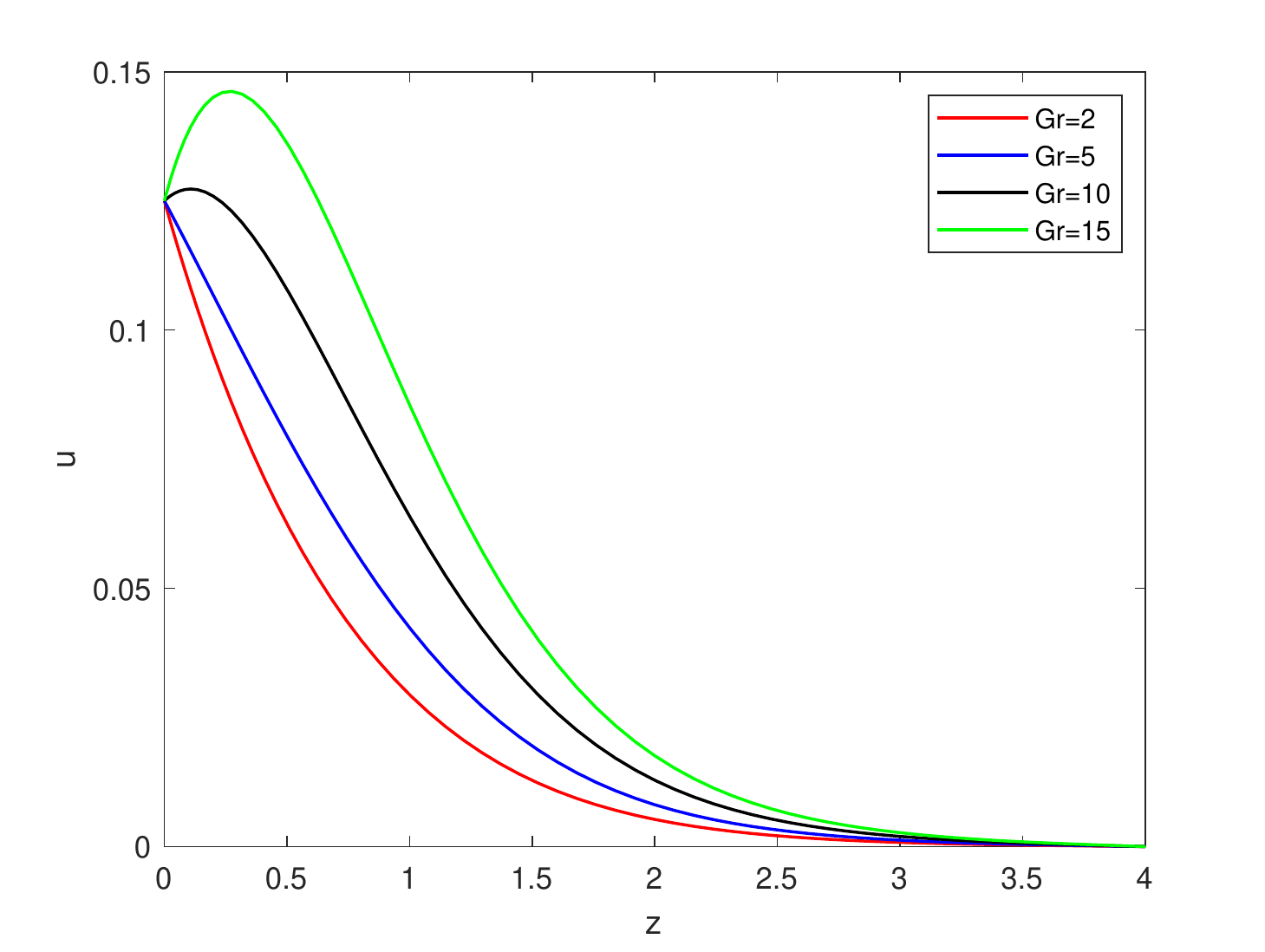}
\end{minipage}
}
\subfigure{
\begin{minipage}[t]{0.42\textwidth}
\centering
\includegraphics[width=6cm]{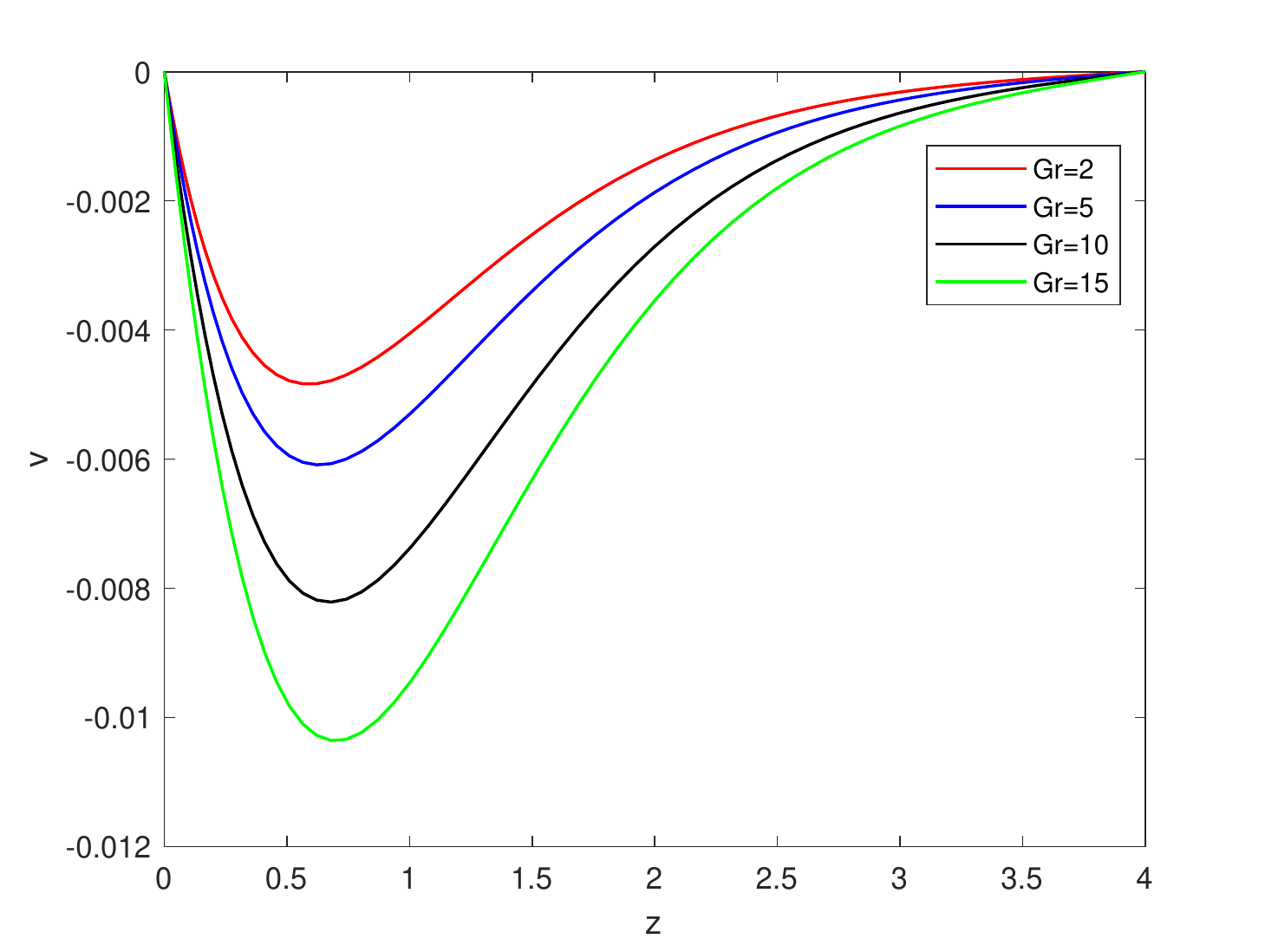}
\end{minipage}
}
\caption{ Velocity $u$ and $v$ for different $Gr$.}
\label{Gr}
\end{figure*}
\begin{figure*}[htbp]
\centering
\subfigure{
\begin{minipage}[t]{0.42\textwidth}
\centering
\includegraphics[width=6cm]{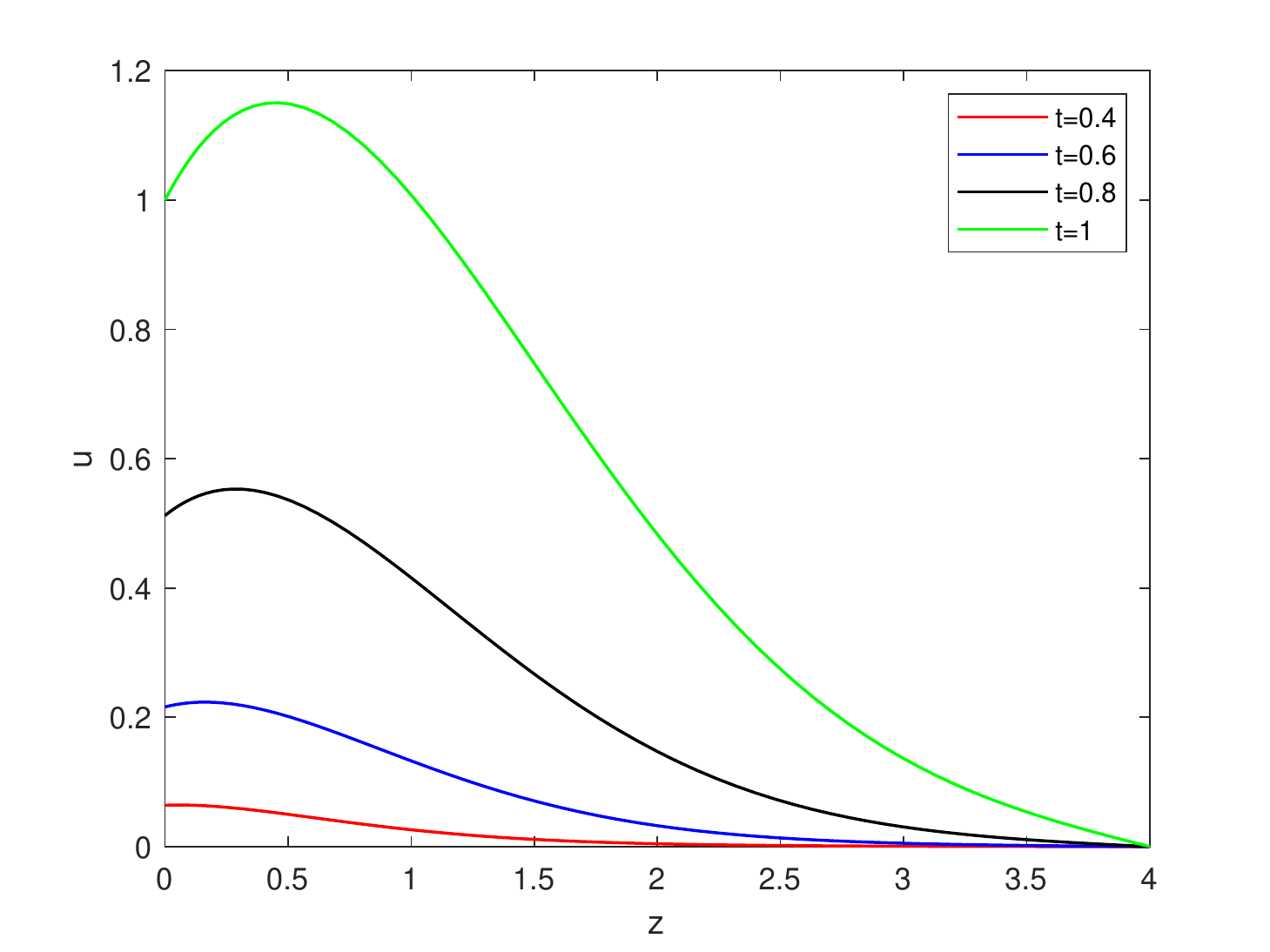}
\end{minipage}
}
\subfigure{
\begin{minipage}[t]{0.42\textwidth}
\centering
\includegraphics[width=6cm]{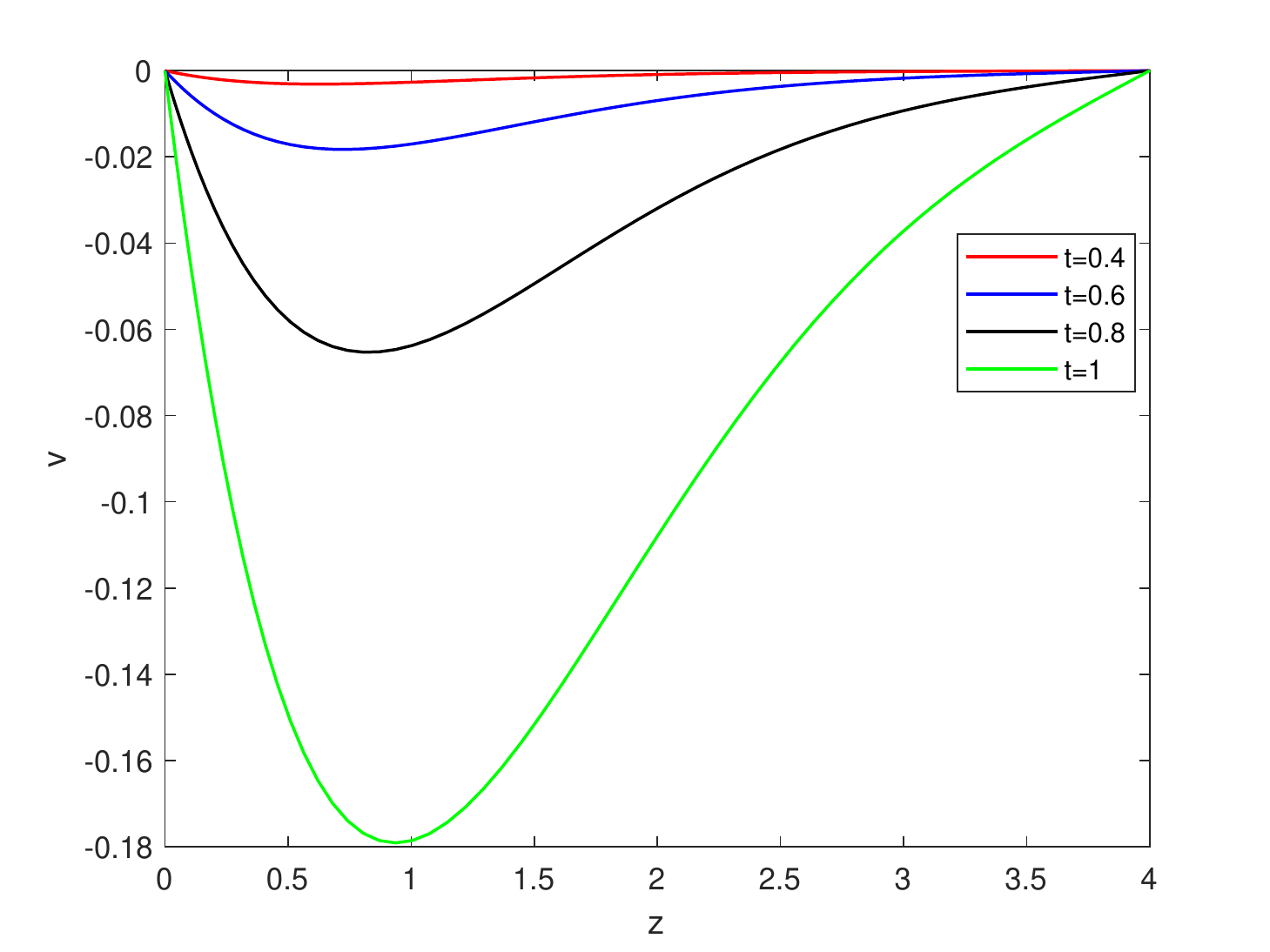}
\end{minipage}
}
\caption{ Velocity $u$ and $v$ for different $t$.}
\label{t}
\end{figure*}

The changes in temperature $\theta$ under different parameter combinations are shown in Figures \ref{betatheta}--\ref{ttheta}. In Figure \ref{betatheta}, $\theta$ decreases as the fractional order $\beta$ increases, which indicates that $\beta$ has a retarding influence on the temperature profile. The same trend occurs in Figure \ref{lambdatheta}, showing that the temperature relaxation parameter $\lambda$ has a similar inhibiting effect on the temperature. Figure \ref{Rtheta} reveals that $\theta$ increases as the thermal radiation parameter $R$ increases. This is because the thermal radiation parameter reduces the thermal buoyancy and minimizes the thickness of the thermal boundary layer. Thus, $\theta$ increases as $R$ grows. From Figure \ref{Prtheta}, we find that the magnitude of $\theta$ decreases with increasing Prandtl number $Pr$, which implies that an increase in $Pr$ causes a reduction in heat transfer. The effect of the heat absorption/generation parameter $H$ on the temperature $\theta$ is presented in Figure \ref{Htheta}. $H>0$ indicates heat generation (heat source), whereas $H<0$ denotes heat absorption (heat sink). Physically, a heat source means that heat is produced, which will increase the temperature of the fluid. Therefore, the temperature rises sharply with any increase in the heat source parameter. The influence of the heat source parameter $H> 0$ on the temperature profile is closely related to the heat sink parameter $H < 0$. These results are physically reasonable, because heat is generated at the surface of the region and the Hall effect of the porous medium also increases, which enhances the temperature of the fluid. The profiles of $\theta$ with respect to $t$ are plotted in Figure \ref{ttheta}, showing that the $\theta$ increases over time $t$.

\begin{figure*}[htbp]
\centering
\begin{minipage}[t]{0.42\textwidth}
\centering
\includegraphics[width=6cm]{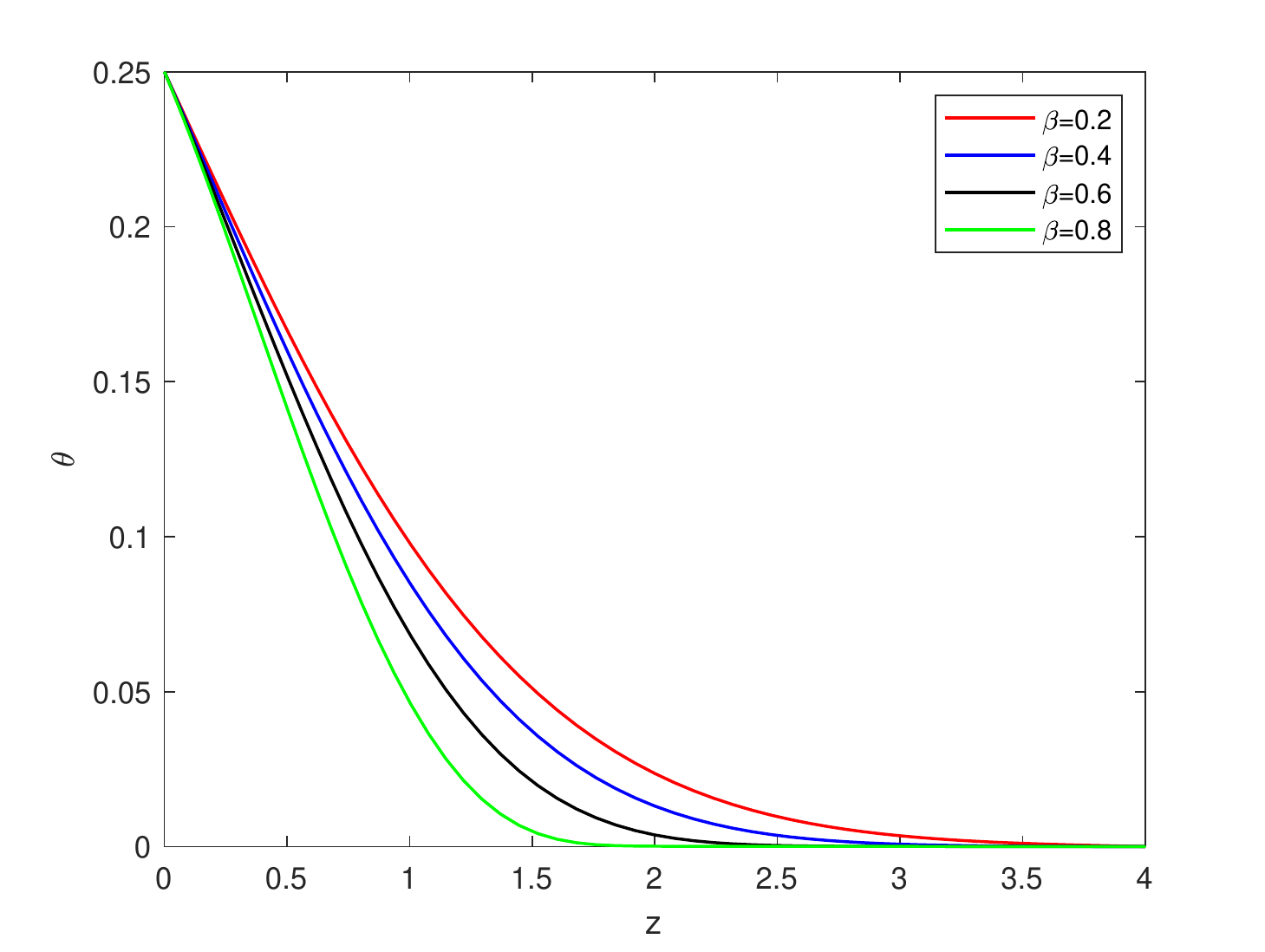}
\caption{ Temperature $\theta$ for different $\beta$.}
\label{betatheta}
\end{minipage}
\begin{minipage}[t]{0.42\textwidth}
\centering
\includegraphics[width=6cm]{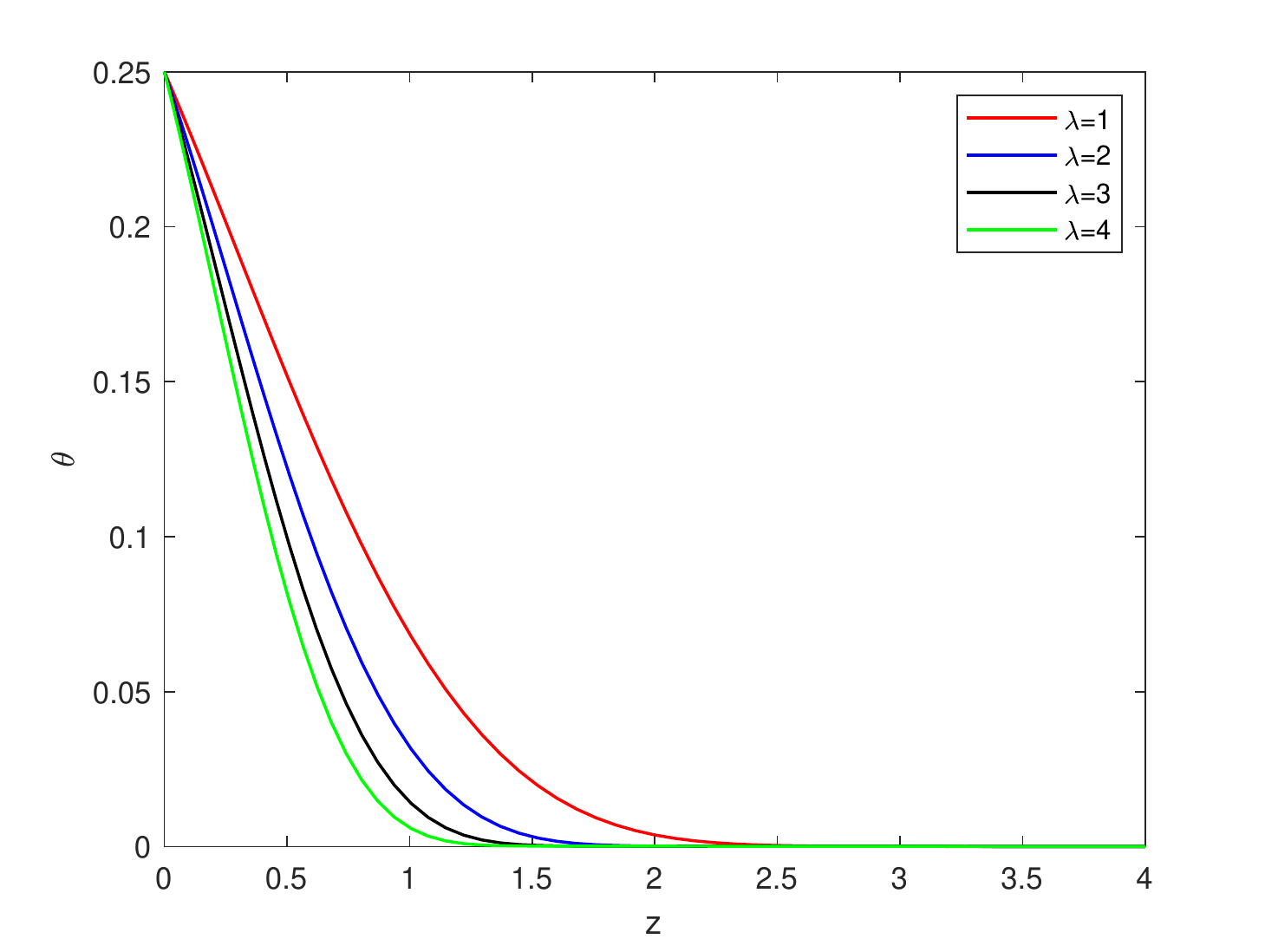}
\caption{ Temperature $\theta$ for different $\lambda$.}
\label{lambdatheta}
\end{minipage}

\centering
\begin{minipage}[t]{0.42\textwidth}
\centering
\includegraphics[width=6cm]{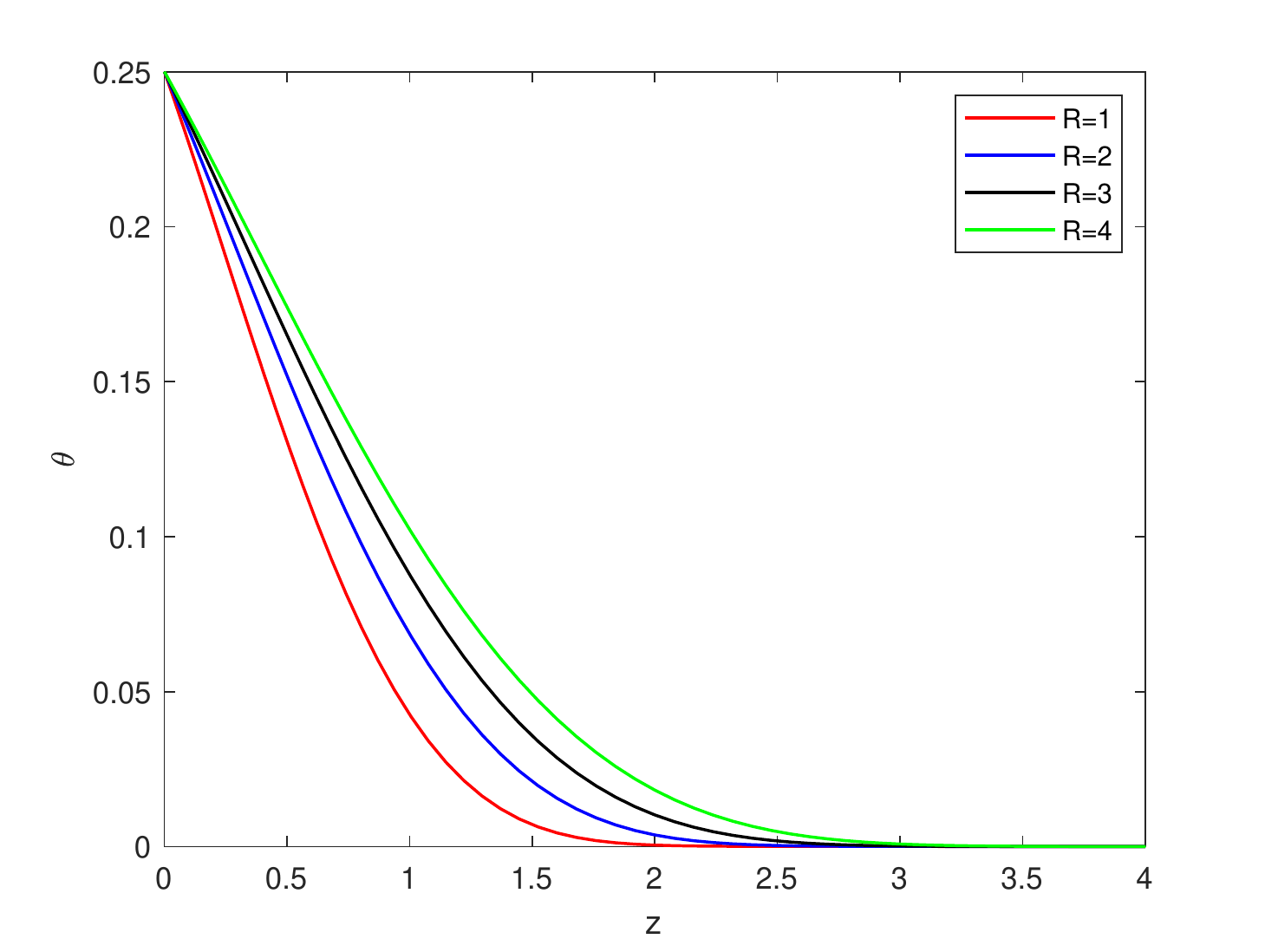}
\caption{ Temperature $\theta$ for different $R$.}
\label{Rtheta}
\end{minipage}
\centering
\begin{minipage}[t]{0.42\textwidth}
\centering
\includegraphics[width=6cm]{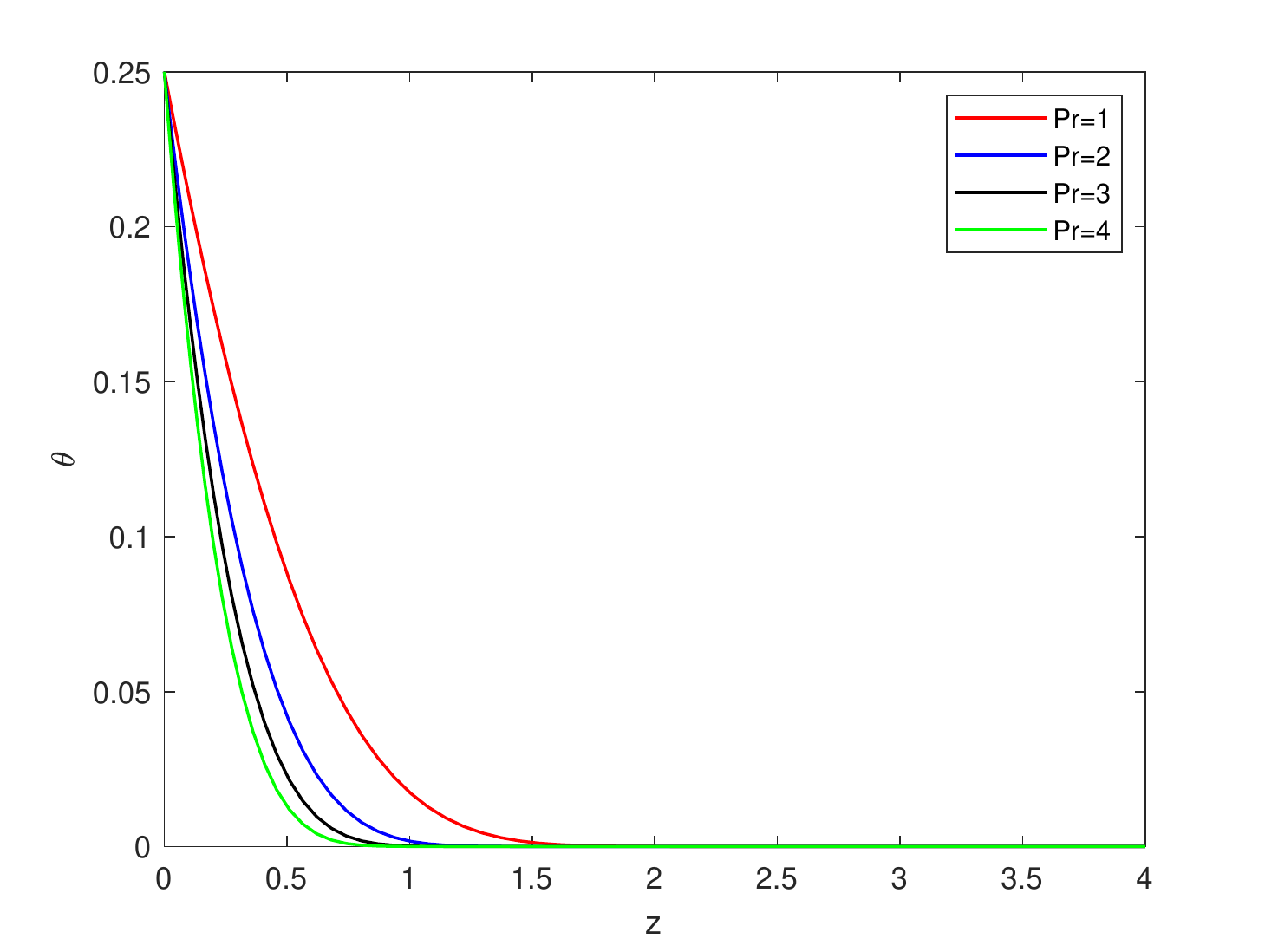}
\caption{ Temperature $\theta$ for different $Pr$.}
\label{Prtheta}
\end{minipage}

\centering
\begin{minipage}[t]{0.42\textwidth}
\centering
\includegraphics[width=6cm]{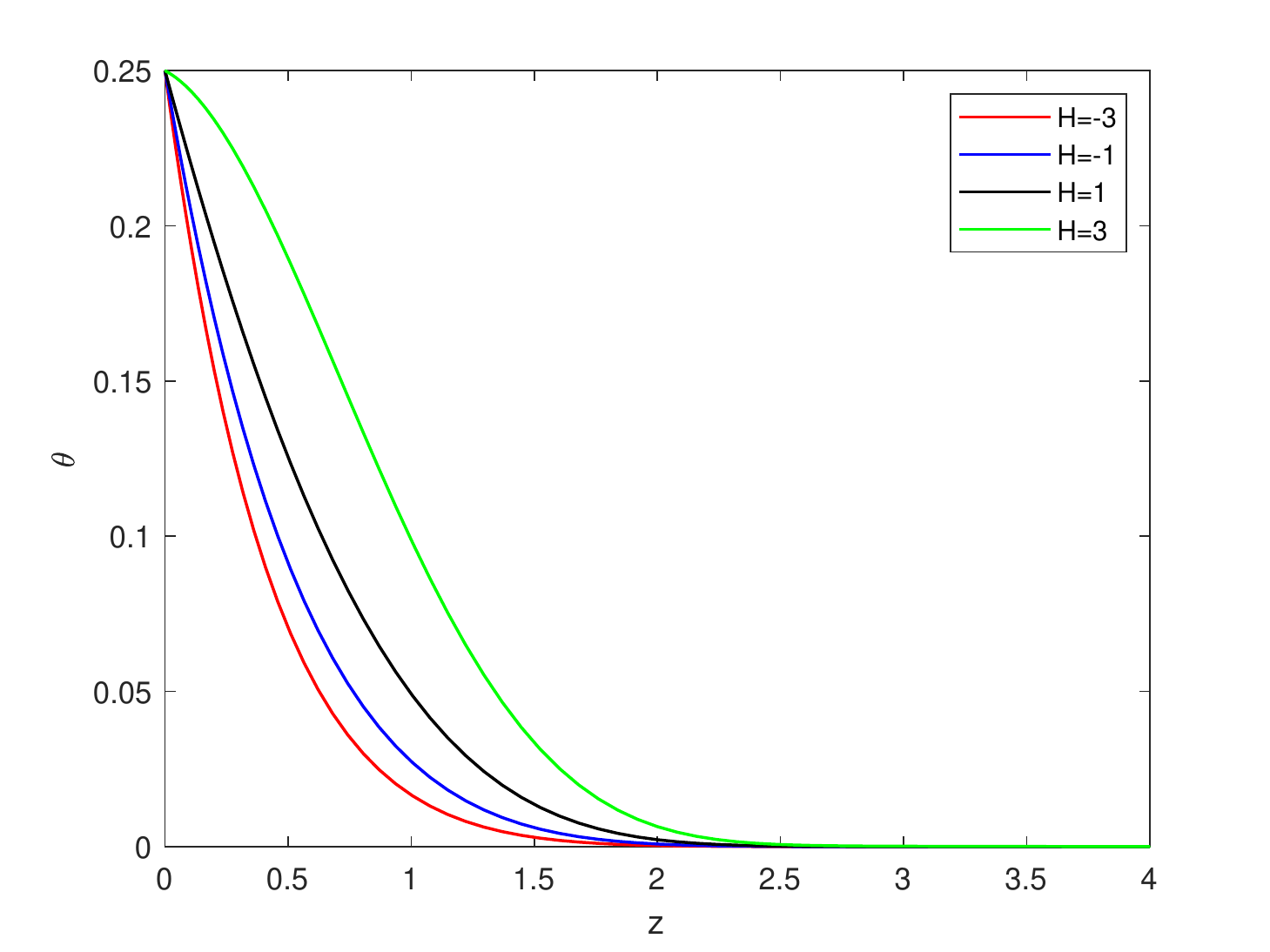}
\caption{ Temperature $\theta$ for different $H$.}
\label{Htheta}
\end{minipage}
\centering
\begin{minipage}[t]{0.42\textwidth}
\centering
\includegraphics[width=6cm]{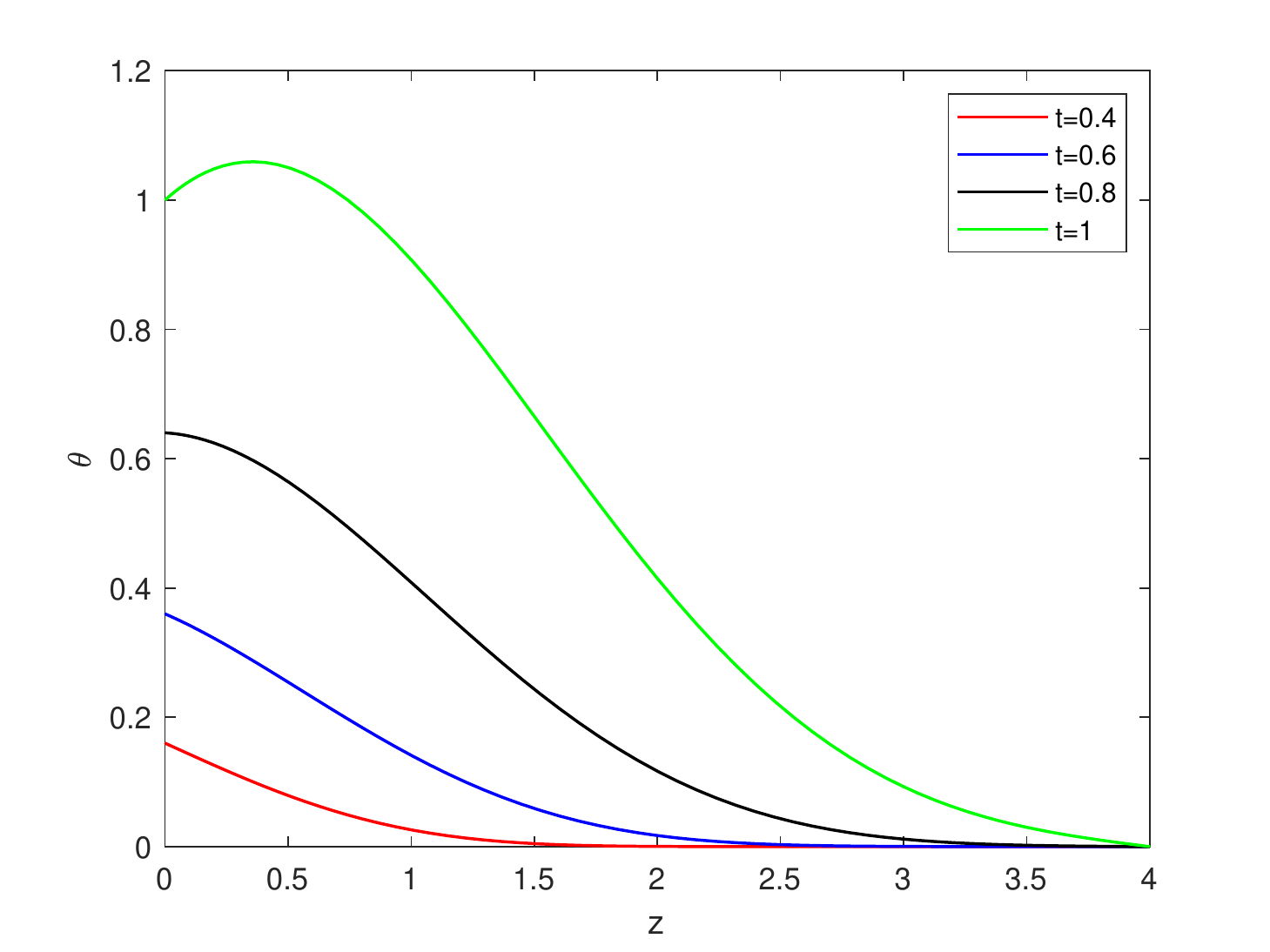}
\caption{ Temperature $\theta$ for different $t$.}
\label{ttheta}
\end{minipage}
\end{figure*}

\section{Conclusions}
\label{sec:7}
We have established a new fractional MHD coupled flow and heat transfer model for a generalized second-grade fluid including the effects of a magnetic field, radiation, and heat source. The coupled model consists of a fractional momentum equation based on the modified constitutive relationship and a heat conduction equation combined with a generalized form of Fourier law. We also presented a numerical scheme in which the second-order FBDF method is applied in the temporal direction and the Legendre spectral method is applied in the spatial direction. The fully discrete scheme was proved to be stable and convergent with an accuracy of $O(\tau^2+N^{-r})$. To reduce the memory requirement and computation time, a fast method was developed and the strict convergence of the numerical scheme with this fast method was proved. Some numerical results were presented to support the theoretical analysis. Finally, we simulate the unsteady fractional MHD flow and heat transfer of the generalized second-grade fluid through a porous medium, and analyzed the effects of the related parameters on the velocity and temperature profiles.
{
\begin{remark}
Although our  numerical scheme and theoretical analysis are for the one-dimensional fractional coupled model which is established according to  the magnetic fluid flow and heat transfer problem, the scheme and analysis in our paper can be fully generalized to the two-dimensional models and even higher dimensional models. For the high-dimensional models, we can  provide the second-order fractional backward difference formula for the temporal discretization and the Legendre spectral method for the spatial discretization, then the stability and convergence analysis are similar to the one-dimensional case. Furthermore, the fast method in our paper is proposed for time discretization and is independent of spatial dimension, so it can be applied to solve the  high-dimensional models.
\end{remark}}
\begin{acknowledgements}
This work has been supported by the Project of the National Key R\&D Program (Grants No. 2021YFA1000202), the National Natural Science Foundation of China (Grants Nos. 12120101001, 12001326, 12171283),  Natural Science Foundation of Shandong Province (Grants Nos. ZR2021ZD03, ZR2020QA032, ZR2019ZD42), China Postdoctoral Science Foundation (Grants Nos. BX20190191, 2020M672038), the startup fund from Shandong University (Grant No. 11140082063130).

\end{acknowledgements}
%
\section*{Conflict of interest}
 The authors declare that they have no conflict of interests regarding the publication of this paper.

\end{document}